\documentclass[final,onefignum,onetabnum]{siamart171218}

\usepackage[english]{babel}
\usepackage[utf8]{inputenc}
\usepackage{ifthen}
\usepackage{color}
\usepackage{graphicx}
\usepackage{cases}
\usepackage{appendix}
\usepackage{xcolor}
\usepackage{epsfig}
\usepackage{tikz,overpic}
\usepackage{multirow}
\usepackage{booktabs}
\usepackage{array}
\usepackage{empheq}
\usepackage{soul}
\usepackage{float}
\usepackage{etaremune}
\usepackage{subfigure}
\crefname{subsection}{Section}{Sections}

\newcommand\eps{\varepsilon}

\newcommand\FF{\mathcal{F}}

\newcommand\MM{\mathcal{M}}

\newcommand\OO{\mathcal{O}}

\renewcommand\SS{\mathcal{S}}

\newcommand\VV{\mathcal{V}}
\newcommand\WW{\mathcal{W}}

\def\set#1#2{\big\{ #1 \, \big| \, #2 \big\}}

\usepackage{lipsum}
\usepackage{amsfonts}
\usepackage{graphicx}
\usepackage{epstopdf}
\usepackage{algorithmic}
\ifpdf
  \DeclareGraphicsExtensions{.eps,.pdf,.png,.jpg}
\else
  \DeclareGraphicsExtensions{.eps}
\fi

\newcommand{\creflastconjunction}{, and~}

\newsiamremark{remark}{Remark}
\newsiamremark{hypothesis}{Hypothesis}
\crefname{hypothesis}{Hypothesis}{Hypotheses}
\newsiamthm{claim}{Claim}

\headers{Singular perturbation analysis of a regularized MEMS model}{A. Iuorio, N. Popovi\'c, and P. Szmolyan}

\title{Singular perturbation analysis of a regularized MEMS model\thanks{\funding{This work was funded by the Fonds zur F\"orderung der 
wissenschaftlichen Forschung (FWF) via the doctoral school ``Dissipation and Dispersion in Nonlinear PDEs'' (project number W1245).}}}

\author{Annalisa Iuorio\thanks{Institute for Analysis and Scientific Computing, 
Vienna University of Technology, Austria
  (\email{annalisa.iuorio@tuwien.ac.at}).}
\and Nikola Popovic\thanks{School of Mathematics and Maxwell Institute for Mathematical Sciences, University of Edinburgh, United Kingdom
  (\email{nikola.popovic@ed.ac.uk}, \url{http://www.maths.ed.ac.uk/~npopovic/}).}
\and Peter Szmolyan\thanks{Institute for Analysis and Scientific Computing, 
Vienna University of Technology, Austria
  (\email{peter.szmolyan@tuwien.ac.at}).}}

\usepackage{amsopn}

\ifpdf
\hypersetup{
  pdftitle={Singular perturbation analysis of a regularized MEMS model},
  pdfauthor={A. Iuorio, N. Popovi\'c, and P. Szmolyan}
}
\fi

\begin{document}

\maketitle

\begin{abstract}
Micro-Electro Mechanical Systems (MEMS) are defined as very small structures that combine electrical and mechanical components on a common substrate. Here, the electrostatic-elastic case is considered, where an elastic membrane is allowed to deflect above a ground plate under the action of an electric potential, whose strength is proportional to a parameter $\lambda$. Such devices are commonly described by a parabolic partial differential equation that contains a singular nonlinear source term. The singularity in that term corresponds to the so-called ``touchdown" phenomenon, where the membrane establishes contact with the ground plate. Touchdown is known to imply the non-existence of steady-state solutions and blow-up of solutions in finite time.

We study a recently proposed extension of that canonical model, where such singularities are avoided due to the introduction of a regularizing term involving a small ``regularization" parameter $\varepsilon$.
Methods from dynamical systems and geometric singular perturbation theory, in particular the desingularization technique known as ``blow-up", allow for a precise description of steady-state solutions of the regularized model, as well as for a detailed resolution of the resulting bifurcation diagram. The interplay between the two principal model parameters $\eps$ and $\lambda$ is emphasized; in particular, the focus is on the singular limit as both parameters tend to zero.
\end{abstract}



\section{Introduction}\label{sec:intro}
Micro-Electro Mechanical Systems (MEMS) are very small structures that combine electrical
and mechanical components on a common substrate to perform various tasks. 
In particular, electrostatic-elastic devices have found important applications in 
drug delivery~\cite{Ts07}, micro pumps~\cite{Iv08}, optics~\cite{Do04}, and 
micro-scale actuators~\cite{Wa09}. In these devices, an elastic 
membrane is allowed to deflect above a ground plate under the action of an electric potential 
$V$, where the distance between plate and membrane is typically much smaller than their diameter; see \cref{fig:eem}. When a
critical voltage threshold $V^\ast$ (``\emph{pull-in voltage}'') is reached, a phenomenon called 
\emph{touchdown} or \emph{snap-through} can occur, {\it i.e.}, the membrane touches the ground plate, which
may cause a short circuit. 

\begin{figure}[!ht]
\centering
\includegraphics[scale=1.0]{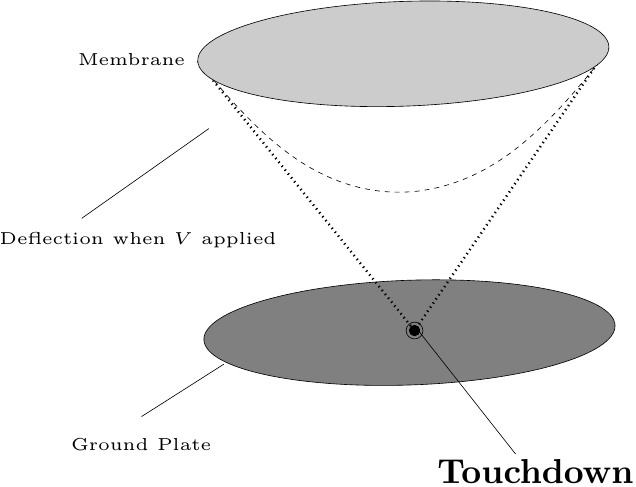}
\caption{Schematic representation of an electrostatic-elastic MEMS device. The elastic membrane deflects towards the ground plate when an electric potential $V$ is applied (dashed curve). If $V$ exceeds a critical value $V^*$
(the so-called ``pull-in voltage''), the membrane touches the ground plate, causing touchdown (dotted line).}
\label{fig:eem}      
\end{figure}

The physical forces 
acting between the elastic components of the device -- which can, {\it e.g.}, be of Casimir or Van der 
Waals type -- may lead to \emph{stiction}, which causes complications in reverting the process in order to 
return to the original state. In the canonical mathematical models proposed in the 
literature~\cite{GW05, LY07, Pe02, PB02}, such systems are described by partial differential equations involving the Laplacian or the bi-Laplacian and a singular source term. The touchdown phenomenon leads to 
non-existence of steady states, or blow-up of solutions in finite time, or both. Hence, no information on 
post-touchdown configurations can be captured by these models.

Recently, an extension of the canonical model has been proposed, where the introduction of a 
potential mimicking the effect of a thin insulating layer above the ground plate prevents physical 
contact between the elastic membrane and the substrate \cite{Li14}.
Mathematically, a nonlinear source term that depends on a small ``regularization'' parameter $\eps$ is added to the partial differential equation. The resulting regularized models 
have been studied in relevant work by Lindsay {\it et al.}; see {\it e.g.} \cite{Li14,Li15} for the membrane case, while the case where the elastic structure is modelled as a beam is  
discussed in \cite{LL12,Li14,Li15}. In one spatial dimension, the governing equations are given by
\begin{align}
\begin{split}
u_t &=u_{xx}-\frac\lambda{(1+u)^2} + \frac{\lambda \eps^{m-2}}{(1+u)^{m}} \\
& \quad \text{for }x\in [-1,1],\text{ with }u=0\text{ when }x=\mp 1\quad (\text{membrane}) \\\label{LiLap}
\end{split}
\end{align}
and
\begin{align}
\begin{split}
u_t &=-u_{xxxx}-\frac\lambda{(1+u)^2} + \frac{\lambda \eps^{m-2}}{(1+u)^{m}} \\
& \quad \text{for }x\in [-1,1], \text{ with }u = \partial_n u = 0\text{ when }x=\mp 1\quad (\text{beam}), \label{LiBiLap}
\end{split}
\end{align}
respectively.
Physically speaking, the variable $u$ denotes the (dimensionless) deflection of the surface, while the 
parameter $\lambda$ is proportional to the square of the applied voltage $V$. The regularizing
term $\lambda \eps^{m-2}(1+u)^{-m}$ with $\eps>0$ and $m>2$, as introduced in \cite{Li14}, accounts for various physical effects that are 
of particular relevance in the vicinity of the ground plate, {\it i.e.}, at $u=-1$; that term induces a 
potential which simulates the effect of an insulating layer whose 
non-dimensional width is proportional to $\eps$. In the following, 
we will consider $m=4$, which corresponds to 
a Casimir effect; alternative choices describe other physical phenomena and can be
studied in a similar fashion. 

Here, we focus on steady-state solutions of the Laplacian case corresponding to a membrane; see \Cref{LiLap}:
\begin{align}\label{eq-2}
u_{xx}=\frac\lambda{(1+u)^2}\bigg[1-\frac{\eps^2}{(1+u)^2}\bigg]\qquad\text{ for }x\in[-1,1],
\text{ with }u=0\text{ when }x=\mp1.
\end{align}
For literature on the bi-Laplacian case, \Cref{LiBiLap}, we refer to~\cite{Li14,Li16,Li15}.

\begin{remark}\label{rem:even}
Due to the symmetry of the boundary value problem \cref{eq-2} under the transformation 
$x\mapsto -x$, all solutions thereof must be even; the proof is straightforward, and is omitted here.
\end{remark}
Before addressing the novel features of the regularized model which are the focus of the 
present article, we briefly summarize the main properties of the non-regularized case 
corresponding to $\eps=0$ in \cref{eq-2}, which are well 
understood~\cite{Pe02,PB02}. The numerically computed bifurcation diagram associated to~\cref{eq-2} for $\eps=0$ is shown in 
\cref{fig:Lin0:a}; it contains two branches of steady-state solutions, where the 
lower branch is stable and 
the upper one is unstable. The upper branch limits on the $\Vert u \Vert_2^2$-axis in the point $B=\left( 0, \frac{2}{3} \right)$, 
which plays a crucial role in the bifurcation diagram of the regularized problem. The two branches are separated by a fold point that is located at 
$\lambda=\lambda^\ast$. For $\lambda > \lambda^\ast$, steady-state solutions of~\cref{LiLap} 
cease to exist, with the transient dynamics leading to a blow-up in finite time. 
Sample solutions along the two branches are plotted in \cref{fig:Lin0:b}; in addition, the piecewise
linear singular solution corresponding to the point $B$ is shown. That singular solution undergoes
touchdown at $x=0$.

\begin{figure}[!ht]
\centering
\hspace{-1cm}
\subfigure{ 
\includegraphics[scale=0.6]{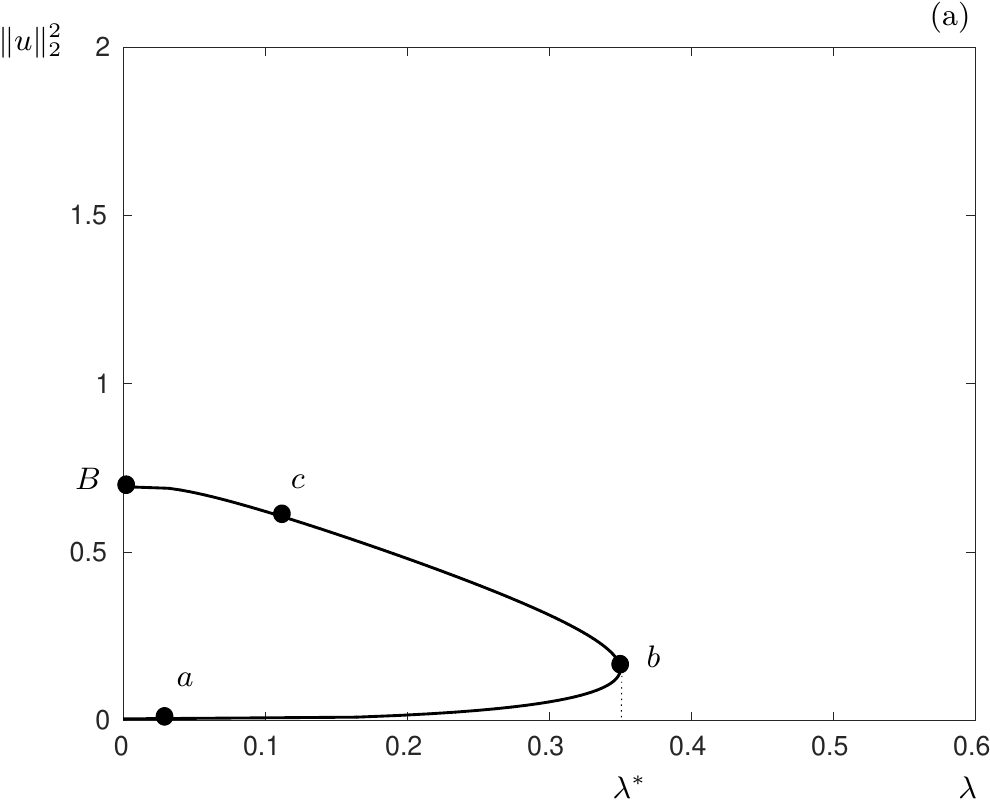}
\label{fig:Lin0:a} 
}
\hspace{.4cm}
\subfigure{ 
\includegraphics[scale=0.6]{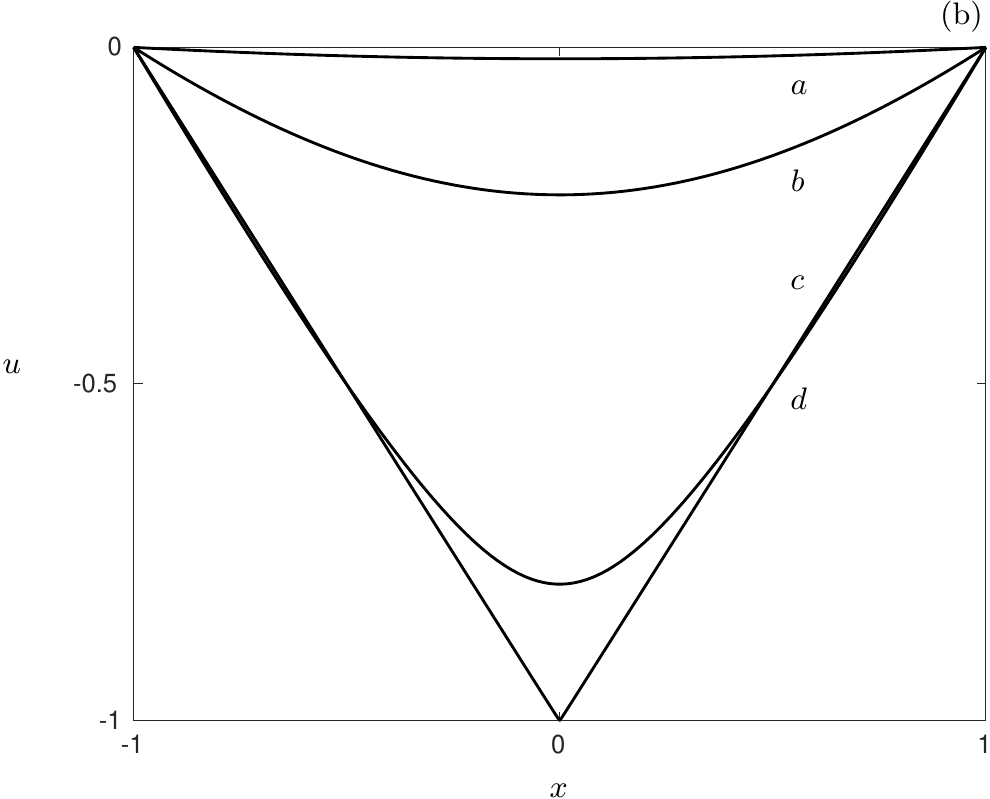}
\label{fig:Lin0:b} 
}
\caption{(a) Bifurcation diagram of the membrane model, \Cref{eq-2}, for $\eps=0$.
The lower and upper branches consist of stable and unstable steady-state solutions, respectively.
The solution labeled $d$ corresponds to the point $B$ and represents the singular solution for $\lambda=0$.
(b) Corresponding solutions in $(x,u)$-space.}
\label{fig:Lin0}      
\end{figure}

The inclusion of the $\eps$-dependent regularizing term, where $0<\eps\ll 1$, considerably alters 
the structure of the bifurcation diagram in \cref{fig:Lin0:a}. The principal new feature is the 
emergence of a third branch of stable steady-state solutions, resulting in the $S$-shaped curve 
shown in \cref{fig:Lin:a}; that diagram was established numerically and via matched asymptotics in~\cite{Li14}. 
In addition to the fact that the fold point at $\lambda^\ast$ now depends on 
$\eps$, there exists another fold point at $\lambda_\ast$ -- which is also $\eps$-dependent -- such 
that, for $\lambda_\ast < \lambda < \lambda^\ast$, there are three branches of steady states, the 
middle one of which is unstable. Solutions on that newly emergent branch are in fact 
bounded below by $u=-1+\eps$. With increasing $\lambda$, 
solutions exhibit a growing ``flat'' portion close to $u=-1+\eps$; cf.~the solution labeled $d$ 
in~\cref{fig:Lin:b}. For $\lambda < \lambda_\ast$ and $\lambda > \lambda^\ast$, there exists a unique stable
steady state; in particular, and in contrast to
the non-regularized case, numerical simulations indicate that a stable steady state exists
for every value of $\lambda>0$. 

\begin{figure}[!ht]
\centering
\subfigure{ 
\includegraphics[scale=0.8]{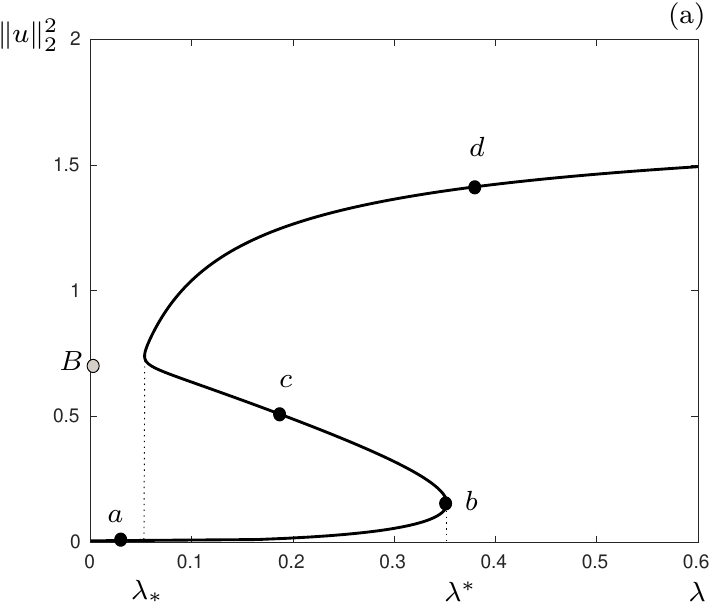}
\label{fig:Lin:a} 
}
\hspace{.5cm}
\subfigure{ 
\includegraphics[scale=0.8]{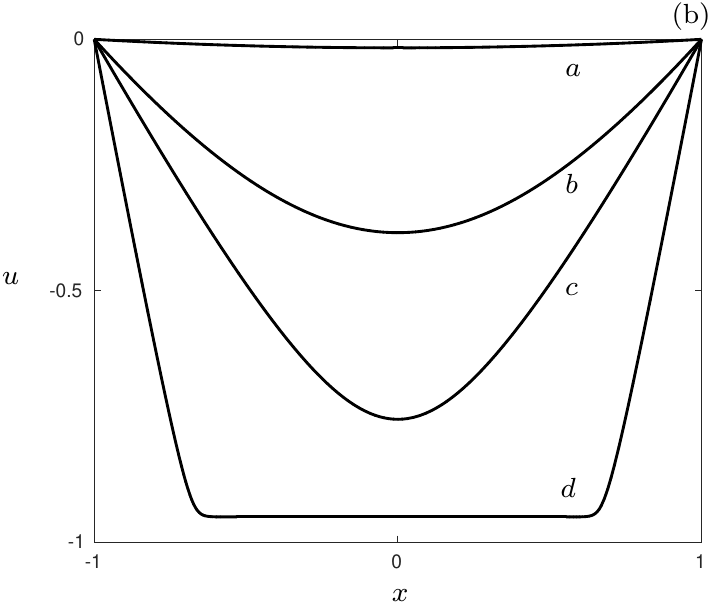}
\label{fig:Lin:b} 
}
\caption{(a) Numerically computed bifurcation diagram of the one-dimensional membrane model, \Cref{eq-2}, for $\eps=0.05$. The gray circle indicates the point $B=\left(0,\frac23\right)$.
(b) Corresponding solutions in $(x,u)$-space.}
\label{fig:Lin}      
\end{figure}


For very small values of $\eps$, the bifurcation diagram in \cref{fig:Lin:a} is difficult to resolve, even 
numerically. These difficulties are particularly prominent in the vicinity of the upper branch and the fold point at $\lambda_\ast(\eps)$; see, {\it e.g.},~\Cref{eq-normu} and \cref{rem-slope} for details. The highly singular nature of the bifurcation diagram in \cref{fig:Lin:a}, as well as the influence of the 
regularization parameter $\eps$ on the structure thereof, are the principal features of interest to us here. \\
In the present work, we will give a detailed geometric analysis of \Cref{eq-2} for small values of $\varepsilon$; in particular, we will prove that the (numerically computed) bifurcation diagram, as shown in \cref{fig:Lin:a}, is correct. Moreover, we will explain the underlying structure of that diagram. In summary, our main result can be expressed as follows:
\begin{theorem}\label{thm-1}
For $\eps\in(0,\eps_0)$, with $\eps_0 > 0$ sufficiently small, and 
\mbox{$\lambda\in[0,\Lambda]$}, with $\Lambda=\OO(1)$ positive and fixed, the bifurcation 
diagram for the boundary value problem \cref{eq-2} has the following properties:
\begin{itemize}
 \item[(i)] In the $(\lambda, \Vert u \Vert_2^2)$-plane, the set of solutions to \cref{eq-2} corresponds to an $S$-shaped curve emanating from the origin. The curve consists of three branches -- lower, middle, and upper -- that are separated by two fold points which are located at $\lambda=\lambda_\ast(\varepsilon)$ and $\lambda=\lambda^\ast(\varepsilon)$. Specifically,
there exists one steady-state solution to \cref{eq-2} for $\lambda<\lambda_\ast(\varepsilon)$ and $\lambda>\lambda^\ast(\varepsilon)$, while for $\lambda_\ast(\varepsilon)<\lambda< \lambda^\ast(\varepsilon)$, there exist three steady-state solutions.
 \item[(ii)] Along the lower and upper branches in \cref{fig:Lin:a}, $\Vert u \Vert_2^2$ is a strictly increasing function of $\lambda$, whereas $\Vert u \Vert_2^2$ is a decreasing function of $\lambda$ along the middle branch.
 \item[(iii)] The function $\lambda_\ast(\varepsilon)$ is $C^1$ in $\varepsilon$ and smooth as a function of $(\varepsilon, \ln \varepsilon)$, and admits the expansion
 \begin{align*}
  \lambda_\ast(\varepsilon) = \frac34\eps-\bigg(\sqrt{\frac32}+\frac98\bigg)\eps^2\ln\eps+ \OO(\eps^2).
 \end{align*}
  Moreover, $\lambda^\ast(\varepsilon)$ is smooth in $\varepsilon$ and admits the expansion
 \begin{align*}
  \lambda^\ast (\varepsilon) = \lambda_0^\ast + \lambda_1^\ast \varepsilon^2 + \mathcal{O}(\varepsilon^4),
 \end{align*}
 with appropriately chosen coefficients $\lambda_0^\ast$ and $\lambda_1^\ast$.
 \item[(iv)] Outside of a fixed neighborhood of the point $B$, the lower and middle branches in \cref{fig:Lin:a} are smooth perturbations of the non-regularized bifurcation curve illustrated in \cref{fig:Lin0:a}, while the upper branch has the following expansion:
 \begin{equation} \label{eq-normu}
  \Vert u \Vert_2^2 = 2 \bigg(1-\frac{\sqrt3}{3}\sqrt{\frac{\eps}{\lambda}} - 2 \eps + \OO(\eps^{\frac32} \ln\eps)\bigg).
 \end{equation}
\end{itemize}
\end{theorem}

\medskip
%

The detailed asymptotic resolution of the bifurcation diagram associated to the boundary value problem~\cref{eq-2}, carried out in the proof of~\Cref{thm-1}, is accomplished through separate investigation of three distinct, yet overlapping, regions in the diagram, both in the singular limit of $\eps=0$ and for $\eps$ positive and sufficiently small.
To that end, we first reformulate \cref{eq-2} in a dynamical systems framework; 
then, identification of two principal parameters in the resulting equations yields a two-parameter singular 
perturbation problem. Careful asymptotic analysis of that problem will allow us to identify the corresponding
limiting solutions, and to show how the third branch in the diagram found for non-zero $\eps$ emerges from the singular limit of $\eps=0$. On that basis, we will prove the 
existence and uniqueness of solutions close to these limiting solutions. While the three regions 
in the diagram share some common features, they need to be investigated 
separately for the structure of the diagram to be fully resolved.

Our analysis is based on a variety of dynamical systems techniques and, principally, on geometric 
singular perturbation theory \cite{Fe79,GSPT,Kue} and the blow-up method, or 
``geometric desingularization''~\cite{Du93,DR96,KS01}. In particular, a combination of 
these techniques will allow us to 
perform a detailed study of the saddle-node bifurcation at the fold point at $\lambda_\ast$, and to obtain an asymptotic expansion (in $\eps$) for $\lambda_\ast(\eps)$. While such an expansion has been derived 
by Lindsay via the method of matched asymptotic expansions~\cite{Li14}, 
cf.~Figure~12 therein, as well as our \cref{fig:Lin:a}, the leading-order coefficients in that expansion 
are calculated explicitly here. In the process, it is shown that the occurrence of logarithmic 
switchback terms in the steady-state asymptotics for \Cref{eq-2}, which has also been observed via asymptotic matching in~\cite{Li14}, is due to a resonance phenomenon in one 
of the coordinate charts after blow-up~\cite{Po05, PS041, PS042, SS04}; cf.~\cref{sec:logsw}.

Without loss of generality, we fix $\Lambda=1$ in~\Cref{thm-1}. The proof of \cref{thm-1} follows from a combination of \cref{prop-1,prop-2,prop-3} below; each of these 
pertains to one of the three above-mentioned regions in the bifurcation diagram.

The article is structured as follows: in \Cref{sec:dynfor}, we reformulate the boundary value problem \cref{eq-2} as a 
dynamical system. In \Cref{sec:blup}, we introduce the principal blow-up transformation on which 
our analysis of the dynamics of \cref{eq-2} close to touchdown is based. In 
\Cref{sec:bifdiag}, we describe in detail the structure of the bifurcation diagram in \cref{fig:Lin:a} by 
investigating separately three main regions therein, as illustrated in
\cref{fig:bdsegm} below. Finally, in \Cref{sec:diou}, we discuss our findings, and we present an 
outlook to future research. 

\section{Dynamical Systems Formulation}\label{sec:dynfor}

For our analysis, we reformulate \Cref{eq-2} as a boundary value problem for a corresponding first-order
system by introducing the new variable $w=u'$; here, it is useful to keep in mind that $w$ represents the
slope of the solution $u$ to \Cref{eq-2}. Moreover, we append the trivial dynamics of 
both the spatial variable $x$, which we relabel as $\xi$, and the regularizing 
parameter $\eps$, to the resulting system:
\begin{subequations}\label{eq-3}
\begin{align}
u' &=w, \\
w' &=\frac\lambda{(1+u)^2}\bigg[1-\frac{\eps^2}{(1+u)^2}\bigg], \\
\xi' &=1, \\
\eps' &=0;
\end{align}
\end{subequations}
here, the prime denotes differentiation with respect to $x$. Next, we multiply the right-hand
sides in \Cref{eq-3} with a factor of $(1+u)^4$, which allows us to desingularize the flow 
near the touchdown singularity at $u=-1$\footnote{That desingularization 
corresponds to a transformation of the independent variable which leaves the phase portrait of 
\cref{eq-3} unchanged for $u>-1$, since the factor $(1+u)^4$ is positive throughout then.}. Finally, 
we define a shift in $u$ via
\begin{equation} \label{eq-shiftu}
\tilde u=1+u,
\end{equation}
which translates that singularity to $\tilde u=0$.

Omitting the tilde and denoting differentiation with respect to the new independent variable 
by a prime, as before, we obtain the system
\begin{subequations}\label{eq-4}
\begin{align}
u' &=u^4w, \label{eq-4a} \\
w' &=\lambda(u^2-\eps^2), \label{eq-4b} \\
\xi' &=u^4, \label{eq-4c} \\
\eps' &=0
\end{align}
\end{subequations}
in $(u,w,\xi,\eps)$-space, with parameter $\lambda$ and subject to the boundary conditions
\begin{align}\label{eq-5}
u=1\qquad\text{for }\xi=\mp1.
\end{align}

Since $\eps$ is small, it seems natural to attempt a perturbative construction of solutions
to the boundary value problem \{\cref{eq-4},\cref{eq-5}\}, which turns out to be non-trivial in spite of the apparent simplicity 
of the governing equations. For $\eps=0$, \Cref{eq-4} can be solved explicitly and
admits degenerate equilibria at $u=0$, which corresponds to the touchdown singularity at $u=-1$
in the original model, \Cref{eq-2}. We denote the resulting manifold of equilibria for
\cref{eq-4} as
\begin{align}\label{S0}
\SS^0=\set{(0,w,\xi,0)}{w\in\mathbb{R},\ \xi\in\mathbb{R}}.
\end{align}

\begin{figure}[!ht]
\centering
\includegraphics[scale=0.9]{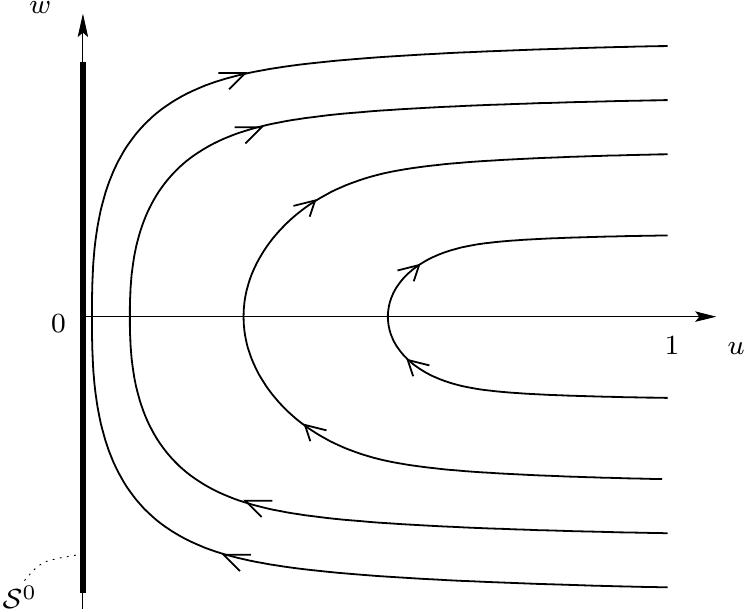}
\caption{Projection of the singular flow of \Cref{eq-4} into $(u,w)$-space
for $\eps=0$ and $\lambda \neq 0$. The solid black line represents the invariant 
manifold $\SS^0$ defined in~\cref{S0}. In view of the boundary conditions in~\cref{eq-5}, 
solutions that originate and terminate at $u=1$ are shown. All such solutions stay to the right of the manifold
$\SS^0$; those with large initial $w$-value tend arbitrarily close to $\SS^0$ without ever reaching it.
Hence, the singular flow is not transverse to $\SS^0$.
}      
\label{fig:SysNTT:a}      
\end{figure}

One complication is introduced by the fact that, for $\lambda \neq 0$, the singular flow of \cref{eq-4}
in $(u,w)$-space that is obtained for $\eps=0$ is not transverse to $\SS^0$; cf.~\cref{fig:SysNTT:a}.
As transversality is a necessary requirement of geometric singular perturbation theory~\cite{Fe79, Kue}, we need to find a way to 
remedy the lack thereof.

For $\lambda=0$ in~\cref{eq-4}, the singular flow becomes even more degenerate; see 
\cref{fig:SysNTT:b}. Furthermore, the set
\begin{align}\label{M0}
\MM^0:=\set{(u,0,\xi,0)}{u\in\mathbb{R}^+,\ \xi\in\mathbb{R}}
\end{align}
now also represents a manifold of equilibria for \Cref{eq-4a,eq-4b}.

\begin{figure}[!ht]
\centering
\includegraphics[scale=0.9]{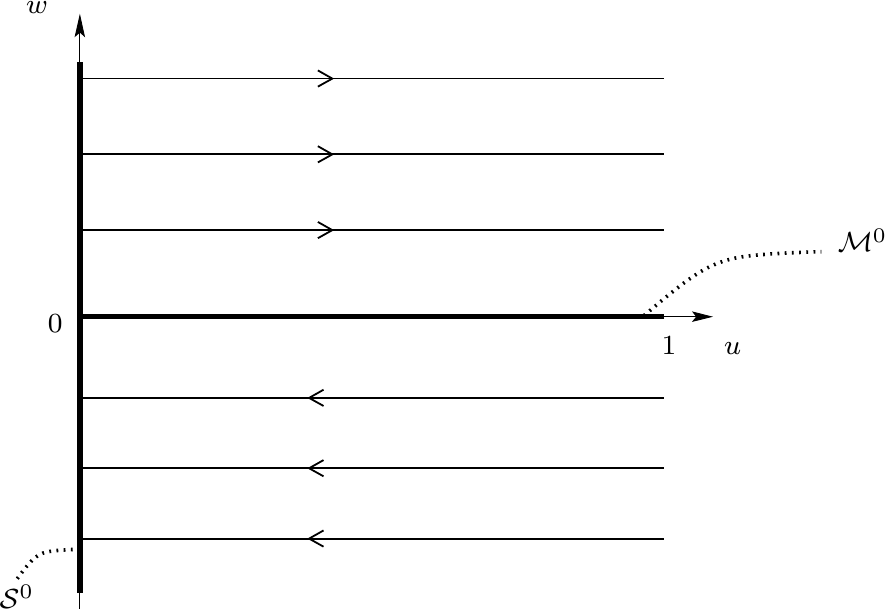}
\caption{Singular flow of \Cref{eq-4} in $(u,w)$-space for $\lambda=0$.
Solid black lines represent the invariant manifolds $\SS^0$ and $\MM^0$ that are
defined in~\cref{S0} and~\cref{M0}, respectively.
Orbits with $w \neq 0$ in $(u,w)$-space are now transverse to $\SS^0$; for $w<0$, these orbits
tend towards $\SS^0$, whereas they tend away from $\SS^0$ for $w>0$. All equilibria on
$\SS^0$ are non-hyperbolic, as the corresponding linearization
of the $(u,w)$-subsystem \{\cref{eq-4a},\cref{eq-4b}\} has a double 
zero eigenvalue.
}      
\label{fig:SysNTT:b}      
\end{figure}

As it turns out, it is beneficial to introduce the following rescaling of $w$ first:
\begin{align}\label{eq-6}
w=\frac{\tilde w}{\delta},
\end{align}
where
\begin{align}\label{delta}
\delta=\sqrt{\frac{\eps}{\lambda}}
\end{align}
is a new, non-negative parameter.
\begin{remark}
The scaling of $w$ by $\sqrt{\lambda}$ in~\cref{eq-6} shifts $\lambda$ from~\cref{eq-4b} to 
\cref{eq-4c}, the $\xi$-equation, after a rescaling of time. The scaling with $\eps^{-\frac12}$ in~
\cref{eq-6} reflects the fact that, for $\lambda=\OO(1)$, $w=\OO(\eps^{-\frac12})$, in agreement
with numerical simulations and asymptotic analysis performed in \cite{Li14}.
\end{remark}
\begin{remark}
Some parts of our analysis are conveniently carried out in the parameters $\eps$ and 
$\lambda$, while others are naturally described in terms of $\eps$ and $\delta$. Hence, we 
will alternate between these two descriptions, as needed.
\end{remark}

Substituting \cref{eq-6} into \cref{eq-4}, multiplying the right-hand sides in the resulting equations
with a factor of $\delta$, omitting the tilde and retaining the prime for differentiation
with respect to the new independent variable, as before, we find
\begin{subequations}\label{eq-7}
\begin{align}
u' &=u^4w, \label{eq-7a} \\
w' &=\eps(u^2-\eps^2), \label{eq-7b} \\
\xi' &=\delta u^4, \label{eq-7c} \\
\eps' &= 0,
\end{align}
\end{subequations}
still subject to the boundary conditions 
\begin{align}\label{eq-5b}
u=1\qquad\text{for }\xi=\mp1.
\end{align}
We remark that the fast-slow structure of 
\Cref{eq-7} is very simple, since \Cref{eq-7a,eq-7b} decouple from 
\Cref{eq-7c}; the latter induces a slow drift in $\xi$.

\Cref{eq-7,eq-5b} will form the basis for the 
subsequent analysis.
Two strategies suggest themselves for constructing solutions to the boundary value problem \{\cref{eq-7},\cref{eq-5b}\}. 
The first such strategy involves two sets of boundary conditions, corresponding 
to suitable intervals of $w$-values that are defined at $\xi=-1$ and $\xi=1$, respectively. Flowing 
these two sets of boundary conditions forward and backward, respectively, we verify the transversality 
of the intersection of the two resulting manifolds at $\xi=0$. Each initial $w$-value $w_0$ for which 
these two manifolds intersect gives a solution to the boundary value problem \{\cref{eq-7},\cref{eq-5b}\}. In particular, that strategy will be used to prove \cref{prop-1}.

Since all solutions to \{\cref{eq-7},\cref{eq-5b}\} are even, by \cref{rem:even}, another possible strategy consists of considering \Cref{eq-7} on the $\xi$-interval $[-1,0]$, with 
boundary conditions $u(-1)=1$ and $w(0)=0$. The set of initial conditions at $\xi=-1$ and $u=1$, but 
with 
arbitrary initial $w$-value $w_0$, is then tracked forward to the hyperplane $\{w=0\}$. The resulting 
manifold is naturally parametrized by $u(w,\eps,\delta,w_0)$ and $\xi(w,\eps,\delta,w_0)$; the unique 
``correct'' value $w_0(\eps,\delta)$ corresponding to a solution to the boundary value problem \{\cref{eq-7},\cref{eq-5b}\} is obtained by solving $\xi(w_0,\eps,\delta)=0$ under the constraint that $w(w_0,\eps,\delta)=0$. Details will be presented in the individual proofs below, in particular in those of \cref{prop-2} and \cref{prop-3}. Given \cref{rem:even}, any solution can be obtained via that second strategy; in fact, the intrinsic symmetry of the problem is also clearly visible in \cref{fig:Lin:b}.

\Cref{eq-7} constitutes a two-parameter fast-slow system in its fast formulation. The small
parameter $\eps$ represents the principal singular perturbation parameter here, while the limit of 
$\delta\to 0$ is also singular. For $\delta=\OO(1)$, the variables $u$ and $\xi$ are fast, while $w$ is 
slow; however, for $\delta$ small, the variable $\xi$ is slow, as well. The manifold $\SS^0$ defined in 
\cref{S0} is still invariant under the flow of \cref{eq-7}. Furthermore, for $\delta=0$, the manifold 
$\MM^0$ defined in~\cref{M0} also represents a set of equilibria for~\cref{eq-7}. 
(We remark that the same scenario occurs for $\lambda=0$ in~\cref{eq-4}.)

Setting $\eps=0$ in \Cref{eq-7}, we obtain the so-called {\it layer problem}
\begin{subequations}\label{eq-113}
\begin{align}
u' &=u^4w, \\
w' &=0, \\
\xi' &=\delta u^4, \\
\eps' &= 0;
\end{align}
\end{subequations}
see \cref{fig:SysNTT:b} for an illustration of the corresponding phase portrait in~$(u,w)$-space and, 
in particular, of the transversality of orbits of the layer problem to $\SS^0$. Rescaling the 
independent variable in \cref{eq-7} by multiplying it with $\eps$ yields the slow formulation
\begin{subequations}\label{eq-110}
\begin{align}
\eps\dot u &=u^4w, \\
\dot w &=u^2-\eps^2, \\
\eps\dot\xi &=\delta u^4, \\
\dot\eps &= 0.
\end{align}
\end{subequations}
The {\it reduced} problem, which is found by taking $\eps\to 0$ in \cref{eq-110}, reads
\begin{subequations}\label{eq-111}
\begin{align}
0 &=u^4w, \\
\dot w &=u^2, \\
0 &=\delta u^4, \\
\dot\eps &= 0.
\end{align}
\end{subequations}

For $\delta=0$, the manifolds $\SS^0$ and $\MM^0$, as defined in~\cref{S0} and \cref{M0},
respectively, now represent 
two branches of the {\it critical manifold} for \Cref{eq-7}; however, neither branch is 
normally hyperbolic, as the Jacobian of the linearization of the layer flow about both $\SS^0$ and 
$\MM^0$ is nilpotent. Moreover, as is obvious from \cref{eq-111}, the reduced flow on 
$\SS^0$ vanishes, and is hence highly degenerate. Therefore, standard geometric theory does not apply directly.

The underlying non-hyperbolicity can be remedied by means of the blow-up method~
\cite{Du93,DR96,KS01,KS01b}. A blow-up with 
respect to $\eps$ will allow us to describe the dynamics of \cref{eq-4} in a neighborhood of 
the manifold $\SS^0$; cf.~\Cref{sec:blup}. 
Our analysis relies on a number of dynamical systems techniques,
such as classical geometric singular perturbation theory \cite{Fe79}, normal form transformations \cite{Wi03}, and the 
Exchange Lemma \cite{JKK96, JK94, Kue}, the
combination of which will result in precise and rigorous asymptotics for \Cref{eq-7}.

To determine the appropriate blow-up transformation, we focus on the $(u,w)$-subsystem \{\cref{eq-7a},\cref{eq-7b}\},
which for $\eps>0$ admits two saddle equilibria at $(\pm\eps,0)$. As we restrict to 
$u\ge 0$, we consider the positive equilibrium only. The scaling $u=\eps\hat u$ transforms 
\{\cref{eq-7a},\cref{eq-7b}\} into 
\begin{align*}
\hat u' &= \eps^3 \hat u^4 w, \\
w' &= \eps^3(\hat u^2-1),
\end{align*}
which yields the integrable system
\begin{subequations} \label{xy}
\begin{align} 
 \hat u' &= \hat u^4 w, \\
 w' &= \hat u^2-1
\end{align}
\end{subequations}
after division through the common factor $\eps^3$. The saddle equilibrium at $(1,0)$, together with its 
stable and unstable manifolds, will play a crucial role in the following; the line $\hat u=0$ is 
invariant, with $w$ decreasing thereon. The corresponding phase portrait is shown in 
\cref{fig:saddle10}.

\begin{figure}[!ht]
\centering
\includegraphics[scale=0.9]{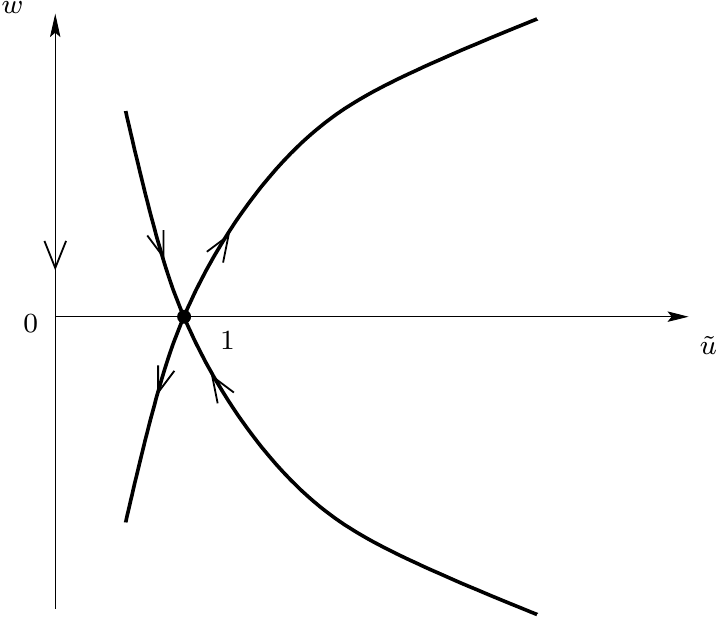}
\caption{The saddle point $(1,0)$ of \Cref{xy} and its stable and unstable manifolds.
}      
\label{fig:saddle10}      
\end{figure}

\section{Geometric Desingularization (``Blow-Up")}\label{sec:blup}

In this section, we introduce the blow-up transformation 
that will allow us to desingularize the flow of \Cref{eq-7} near the non-hyperbolic 
manifold $\SS^0$. The discussion at the end of \Cref{sec:dynfor} suggests the following blow-up:
\begin{align}\label{eq-8}
u=\bar r\bar u,\quad w=\bar w,\quad\xi=\bar \xi,\quad\text{and}\quad\eps=\bar r\bar\eps,
\end{align}
where $(\bar w, \bar \xi) \in \mathbb{R}^2$ and $(\bar u, \bar \eps) \in S^1$, {\it i.e.}, 
$\bar u^2+\bar\eps^2=1$. Moreover, $\bar{r} \in [0,r_0)$, with $r_0>0$. We note that the equilibrium at $(u,\eps)=(0,0)$ is blown up to the 
circle $\{\bar{r} = 0\}$; here, we emphasize that we do not blow up the variables $w$ and $\xi$.

The vector field that is induced by \cref{eq-7} on the cylindrical manifold
in $(\bar u,\bar w,\bar\xi,\bar\eps,\bar r)$-space is best described in coordinate charts. We require
two charts here, $K_1$ and $K_2$, which are defined by $\bar u=1$ and $\bar\eps=1$, respectively:
\begin{subequations}\label{eq-9}
\begin{align}
K_1:\ & (u,w,\xi,\eps)=(r_1,w_1,\xi_1,r_1\eps_1), \label{eq-9a} \\
K_2:\ & (u,w,\xi,\eps)=(r_2u_2,w_2,\xi_2,r_2). \label{eq-9b}
\end{align}
\end{subequations}
\begin{remark} \label{rem:K1K2}
The phase-directional chart $K_1$ describes the ``outer" regime, which corresponds to the 
transient dynamics from $u=1$ to $u=0$, while the rescaling chart $K_2$ -- 
also known as the \emph{scaling chart} -- covers the ``inner" regime where $u \approx 0$, in the 
context of \Cref{eq-7}; in particular, in chart $K_2$, we recover \Cref{xy}.
\end{remark}

\begin{figure}[!ht]
\centering
\hspace{-.5cm}
\subfigure[Flow in $(u,w,\eps)$-space; the thick gray line represents the critical manifold $\SS^0$.]{
\includegraphics[scale=0.68]{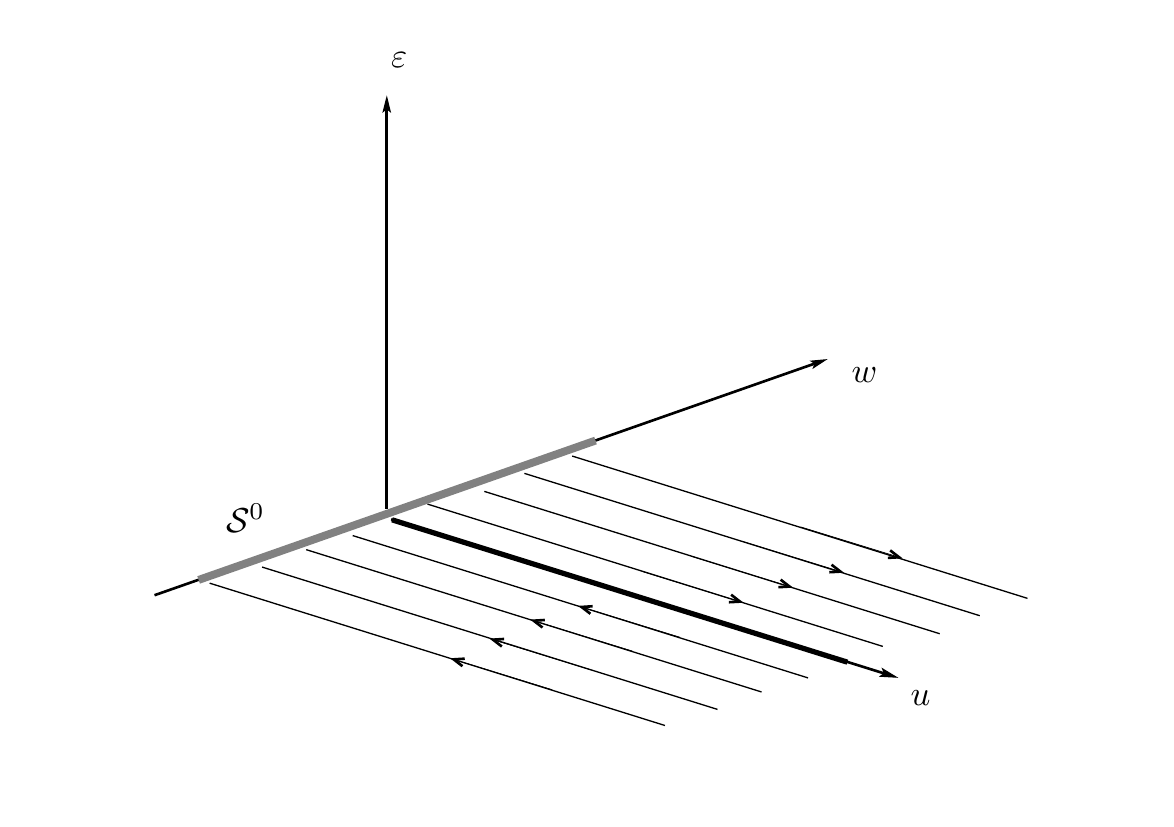} 
}
\subfigure[Geometry in blown-up $(\bar{u},\bar{w},\bar{\eps})$-space; $\SS^0$ is now represented 
by the cylinder corresponding to \mbox{$\bar{u}^2+\bar{\eps}^2=1$}.]{
\includegraphics[scale=0.68]{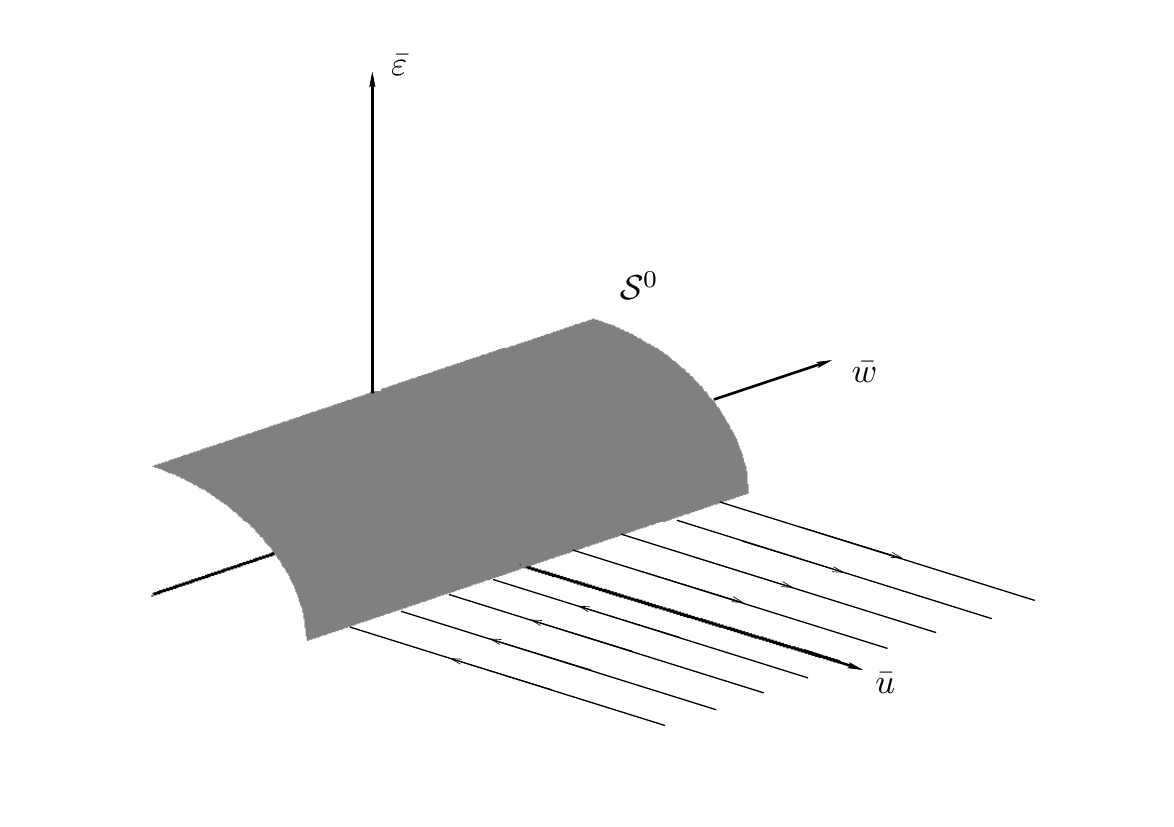}
}
\caption{Flow of \Cref{eq-7} for $\eps=0$ (a) before and (b) after the blow-up in \cref{eq-8}.}
\label{fig:blowup}      
\end{figure}

The change of coordinates between charts $K_1$ and $K_2$, which we denote by $\kappa_{12}$, can be written as
\begin{align}\label{eq-33}
\kappa_{12}:\ (u_2,w_2,\xi_2,r_2)=\big(\eps_1^{-1},w_1,\xi_1,r_1\eps_1\big),
\end{align}
while its inverse $\kappa_{21}$ is given by
\begin{align}\label{eq-34}
\kappa_{21}:\ (r_1,w_1,\xi_1,\eps_1)=\big(r_2u_2,w_2,\xi_2,u_2^{-1}\big).
\end{align}


To obtain the governing equations in $K_1$, we substitute the transformation from \cref{eq-9a} into \Cref{eq-7}; a straightforward calculation yields
\begin{subequations}\label{eq-10}
\begin{align}
r_1' &=r_1^4w_1, \\
w_1' &=r_1^3\eps_1(1-\eps_1^2), \\
\xi_1' &=\delta r_1^4, \\
\eps_1' &=-r_1^3\eps_1w_1.
\end{align}
\end{subequations}
Since $\eps=r_1 \eps_1$, the singular limit of $\eps=0$ corresponds to the restriction of the flow of 
\cref{eq-10} to one of the invariant planes $\{r_1=0\}$ or $\{\eps_1=0\}$. In order to obtain a 
non-vanishing vector field for $r_1=0$, we desingularize \Cref{eq-10} by dividing out
a factor of $r_1^3$ from the right-hand sides, which again represents a rescaling of 
the corresponding independent variable:
\begin{subequations}\label{eq-11}
\begin{align}
r_1' &=r_1w_1, \label{eq-11a} \\
w_1' &=\eps_1(1-\eps_1^2), \label{eq-11b} \\
\xi_1' &=\delta r_1, \label{eq-11c} \\
\eps_1' &=-\eps_1w_1. \label{eq-11d}
\end{align}
\end{subequations}


The governing equations in $K_2$ are obtained by substituting the transformation in \cref{eq-9b} into \cref{eq-7}, which gives
\begin{subequations}\label{eq-12}
\begin{align}
u_2' &=r_2^3u_2^4w_2, \\
w_2' &=r_2^3(u_2^2-1), \\
\xi_2' &=\delta r_2^4u_2^4. \\
r_2' &=0.
\end{align}
\end{subequations}
Desingularizing as before, by dividing out a factor of $r_2^3$ from the right-hand sides in
\cref{eq-12}, we find
\begin{subequations}\label{eq-13}
\begin{align}
u_2' &=u_2^4w_2, \label{eq-13a} \\
w_2' &=u_2^2-1, \label{eq-13b} \\
\xi_2' &=\delta r_2u_2^4, \label{eq-13c} \\
r_2' &=0.
\end{align}
\end{subequations}
Here, we remark that, by construction, the $(u_2,w_2)$-subsystem \{\cref{eq-13a},\cref{eq-13b}\} corresponds to \Cref{xy}.

Finally, we define various sections for the blown-up vector field, which will be used
throughout the following analysis: in $K_1$, we will require the entry and exit sections
\begin{subequations}\label{eq-20}
\begin{align}
\Sigma_1^{\rm in} &:=\set{(\rho,w_1,\xi_1,\eps_1)}{w_1\in[w_-,w_+],\ \xi_1\in[\xi_-,\xi_+],\text{ and }\eps_1\in[0,\sigma]}\quad\text{and} \label{eq-20a} \\
\Sigma_1^{\rm out} &:=\set{(r_1,w_1,\xi_1,\sigma)}{r_1\in[0,\rho],\ w_1\in[w_-,w_+],\text{ and }\xi_1\in[\xi_-,\xi_+]}, \label{eq-20b}
\end{align}
\end{subequations}
respectively, where $0<\rho<1$ and $0<\sigma<1$ are appropriately defined constants, while $w_\mp$ and 
$\xi_\mp$ are real constants, with $w_-<-\frac2{\sqrt3}$ and $w_+>\frac2{\sqrt3}$.
Similarly, in chart $K_2$, we will employ the section
\begin{align}\label{eq-21}
\Sigma_2^{\rm in}:=\set{(\sigma^{-1},w_2,\xi_2,r_2)}{w_2\in[w_-,w_+],\ \xi_2\in[\xi_-,\xi_+],\text{ and }r_2\in[0,\rho\sigma]}; 
\end{align}
here, we note that $\Sigma_2^{\rm in} = \kappa_{21}\big(\Sigma_1^{\rm out}\big)$.

\Cref{eq-11,eq-13} will allow us to construct solutions of~\{\cref{eq-7},\cref{eq-5b}\}.
Following the strategy outlined in~\Cref{sec:dynfor}, we will focus our attention on the $\xi$-interval $[-1,0]$ with boundary conditions $u(-1)=1$ and $w(0)=0$; in particular, and as indicated in \Cref{rem:K1K2}, the ``outer" regime will be realized in terms of the flow between the sections $\Sigma_1^{\rm in}$ and $\Sigma_1^{\rm out}$ in chart $K_1$. Translating the resulting asymptotics into chart $K_2$ via the transformation in \Cref{eq-33}, we will then construct solutions in the ``inner" regime between the section $\Sigma_2^{\rm in}$ and the hyperplane corresponding to $\{w=0\}$.

\begin{remark} \label{not}
In the following, we will denote a given general variable $z$ in blown-up space with $\bar{z}$.
In charts $K_i$, $i=1,2$ that variable will instead be labeled with the corresponding subscript, as 
$z_i$. 
\end{remark}

\section{Analysis of Bifurcation Diagram -- Proof of \cref{thm-1}} \label{sec:bifdiag}

In this section, we establish the bifurcation diagram in \cref{fig:Lin:a} for $\eps$ positive and 
sufficiently small, proving \cref{thm-1}. To that end, we investigate the existence and uniqueness of solutions to 
\Cref{eq-7}, subject to the boundary conditions in \cref{eq-5b}.

All such solutions arise as perturbations of certain limiting solutions that are obtained in the limit
of $\eps=0$. We denote these limiting solutions as singular solutions, as is usual in geometric singular perturbation theory.
The approach adopted thereby is the following: first, singular solutions are constructed 
by analyzing the dynamics in charts $K_1$ and $K_2$ separately in the limit as $\eps\to0$.
Then, the persistence of singular solutions for non-zero $\eps$ is shown via the 
shooting argument outlined in \Cref{sec:dynfor}, which relies on the transversality of the 
geometric objects involved. That transversality translates into the existence of solutions to the boundary value problem \{\cref{eq-7},\cref{eq-5b}\} along the branches depicted in the bifurcation diagram in \cref{fig:Lin:a}.

\begin{definition}\label{def:soltyp}
\hspace{-.2cm}We distinguish three types of singular solutions to the boundary value problem \{\cref{eq-7},\cref{eq-5b}\}; see \cref{fig:soltyp}:
\begin{description}
\item[\emph{Type I.}] Solutions of type I satisfy $u=0$ for $x \in I$, where $I$ is an interval centered at 
$x=0$. Consequently, the slope of such solutions must initially satisfy $|w|>1$, in terms of the original $w$-variable. Type I-solutions, which will henceforth be illustrated in blue, occur in two subtypes: 
the ones 
corresponding to $\lambda=\OO(\eps)$ have constant finite slope $w$ outside of $I$, while the ones 
corresponding to $\lambda=\OO(1)$ vanish on $I=(-1,1)$.
\item[\emph{Type II.}] Solutions of type II are those of slope $w\equiv\mp1$, in terms of the original 
$w$-variable. These solutions exhibit ``touchdown'', reaching $\{u=0\}$ at one point only, namely at $\xi=0$. 
Type II-solutions will be indicated in green in all subsequent figures.	
\item[\emph{Type III.}] Solutions of type III never reach $\{u=0\}$; hence, no touchdown phenomena 
occur. These solutions correspond to solutions of the non-regularized model, with $\eps=0$ in 
\Cref{eq-2} \cite{Pe02, PB02}.
\end{description}
\end{definition}

\begin{remark} \label{rem:soltyp}
The usage of the plural in the definition of type II-solutions requires additional clarification. 
For \Cref{eq-4}, there exists just one singular solution of type II for $\lambda=0$ with slope 
$w=\mp 1$; see the solution labeled~$d$ in Figure~\cref{fig:Lin0}. However, in our blow-up analysis, 
that singular solution corresponds to a one-parameter family of type II-solutions.
\end{remark}

\begin{figure}[H]
\centering
\subfigure[]{
\includegraphics[scale=0.75]{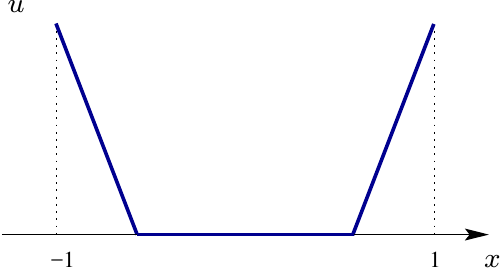} \hspace{.5cm}
\includegraphics[scale=0.75]{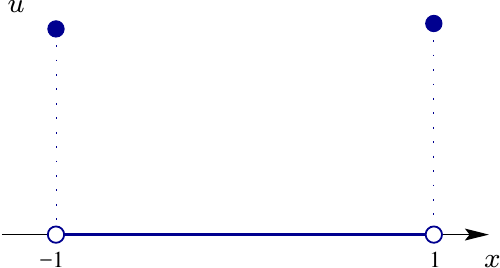} \label{fig:soltypI}
}\\
\subfigure[]{
\includegraphics[scale=0.75]{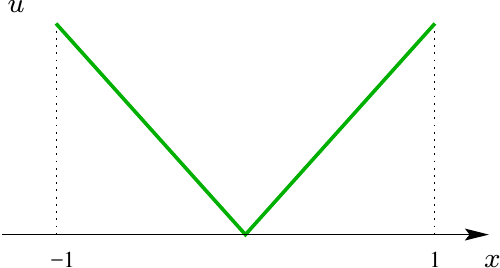} \label{fig:soltypII}
}\\
\subfigure[]{
\includegraphics[scale=0.75]{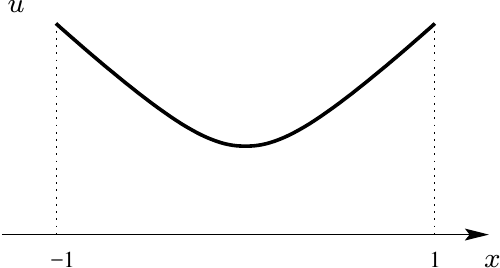}
}
\caption{Singular solutions to \Cref{eq-7}, as specified in \cref{def:soltyp}: (a) type I-solutions for $\lambda=\OO(\eps)$ (left panel) and $\lambda = \OO(1)$ (right panel), (b) type II-solutions, and (c) type III-solutions.}
\label{fig:soltyp}      
\end{figure}

For $\eps>0$, we divide the bifurcation diagram in \cref{fig:Lin:a} into three overlapping regions, as 
shown in \cref{fig:bdsegm}. 
\begin{remark}\label{normu}
Henceforth, we will refer to the norm $\Vert u\Vert_2^2$ in terms of the original variable $u$ in order to be able to compare our analysis with that in \cite{Li14}; see \cref{fig:Lin0,fig:Lin}. 
\end{remark}
Region $\mathcal{R}_1$ is defined as
\begin{align}\label{eq:R1}
\mathcal{R}_1:=[0,1]\times\bigg[\frac23+\nu_1,2\bigg],\qquad\text{with }\nu_1 >0;
\end{align}
that region covers the upper part of the bifurcation diagram, where we find the newly emergent 
branch of solutions for $\eps>0$ in \cref{eq-2} by perturbing from singular solutions of type I. 
Region $\mathcal{R}_2$, which is defined as
\begin{align}\label{eq:R2}
\mathcal{R}_2:=[0,\eps\lambda_2]\times\bigg[\frac23-\nu_2,\frac23+\nu_2\bigg],\qquad
\text{with }\lambda_2>0\text{ and }\nu_2>0
\end{align}
for $\nu_2>\nu_1$ and $\lambda_2$ large, but fixed, represents a small neighborhood of the point 
$B$ that is depicted as a rectangle in \cref{fig:bdsegm}. That region
shrinks with decreasing $\eps$, collapsing to the segment 
$\{0\}\times\big[\frac23-\nu_2,\frac23+\nu_2\big]$
as $\eps \to 0$. The branch of solutions contained in this ``transition'' region is constructed 
by perturbation from singular solutions of types I and II. Finally,
region $\mathcal{R}_3$ is defined as
\begin{align}\label{eq:R3}
\mathcal{R}_3:=[0,1]\times\bigg[0,\frac23+\nu_2\bigg]\setminus[0,\eps\lambda_3]
\times\bigg[\frac23-\nu_3,\frac23+\nu_2\bigg],\qquad\text{with }\lambda_3>0\text{ and }\nu_3 >0,
\end{align}
where $\nu_3<\nu_2$ and $\lambda_3$ is again large, but fixed, with $\lambda_3 < \lambda_2$.
Region $\mathcal{R}_3$ covers the lower part of the bifurcation diagram in \cref{fig:Lin:a}, and 
contains the branch of solutions which is obtained by perturbing from solutions of types II and III. 

\begin{figure}[!h]
\centering
\includegraphics[scale=.7]{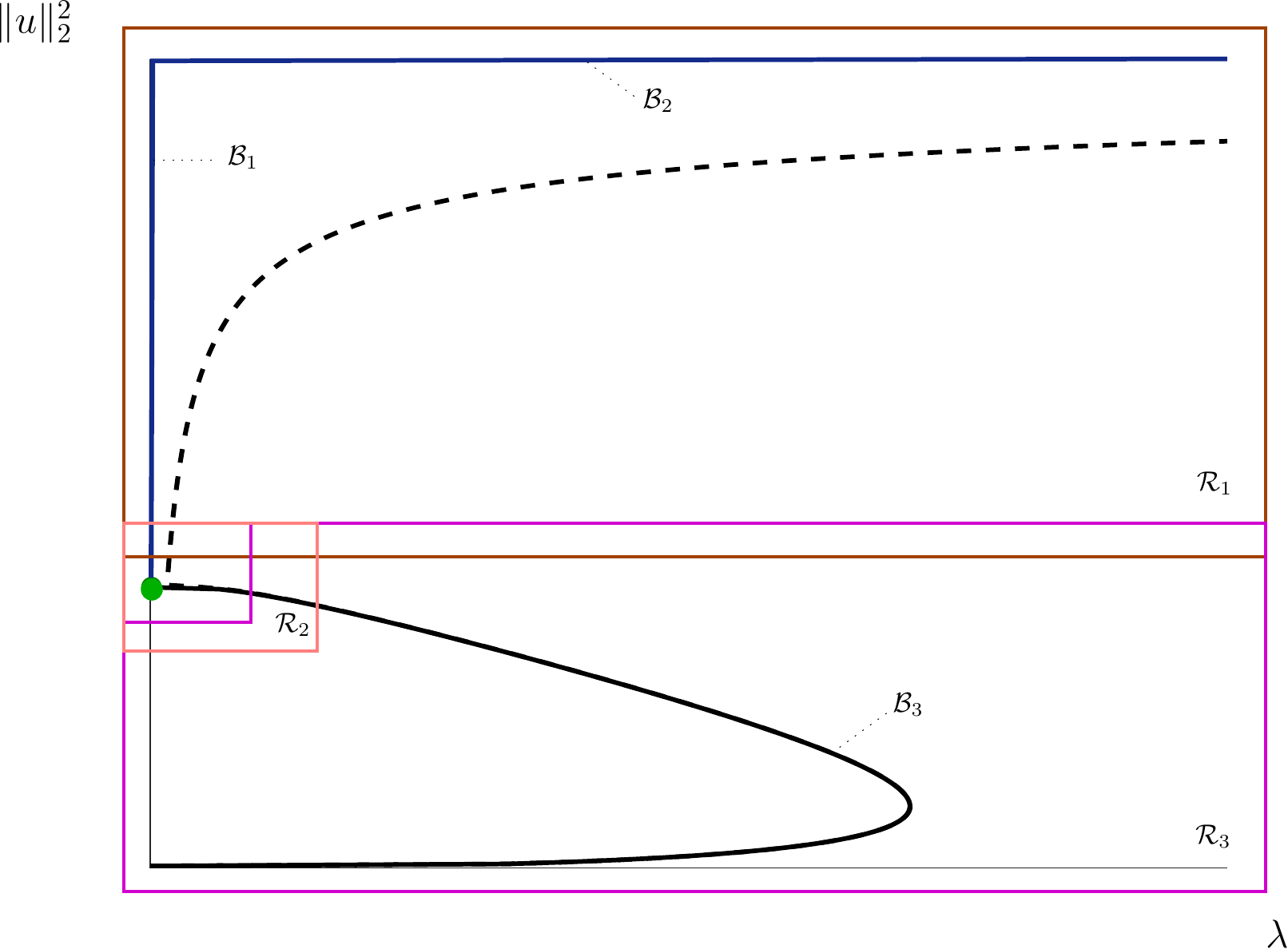}
\caption{Covering of the bifurcation diagram for the boundary value problem \{\cref{eq-7},\cref{eq-5b}\}
by three overlapping regions $\mathcal{R}_1$ (brown), $\mathcal{R}_2$ (pink), and 
$\mathcal{R}_3$ (magenta), for $\eps$ positive and small. (For improved visibility,
the regions have been extended slightly below $\Vert u \Vert_2^2=0$, above $\Vert u \Vert_2^2=2$, 
and to negative $\lambda$, respectively; also, we recall \cref{normu} with regard to the interpretation 
of $\Vert u \Vert_2^2$ in this context.) The branches of solutions to 
the boundary value problem \{\cref{eq-7},\cref{eq-5b}\} for $\eps=0.01$ (dotted curve) and $\eps=0$ (solid curve)
are also displayed. In $\mathcal{R}_3$, these branches overlap almost entirely. 
For $\eps=0$, the blue branch reduces to the union of a vertical part $\mathcal{B}_1$, 
corresponding to $\lambda = \OO(\eps)$, and a horizontal part
$\mathcal{B}_2$ which corresponds to $\lambda = \OO(1)$.
The green dot at $B$ represents the singular solution of type II
for $\lambda=0$ that is labeled~$d$ in \cref{fig:Lin0:a}. The black curve, corresponding to 
the branch of solutions to the non-regularized model, is labeled $\mathcal{B}_3$.
In the limit as $\eps\to 0$, $\mathcal{R}_2$ shrinks to a segment on the $\Vert u \Vert_2^2$-axis
that contains the point $B$, cf.~\cref{eq:R2}, while region $\mathcal{R}_3$ grows 
to a rectangle minus a smaller segment on the $\Vert u \Vert_2^2$-axis
containing the point $B$; recall~\cref{eq:R3}.}
\label{fig:bdsegm}      
\end{figure}

The true meaning of these regions becomes clearer when we consider a blow-up of the bifurcation 
diagram in parameter space, {\it i.e.}, with respect to $\lambda$ and $\eps$, as illustrated in 
\cref{fig:bdblup}. (That same point of view will also prove useful in parts of the following analysis.)
We first embed the diagram, which depends on $(\lambda, \Vert u \Vert_2^2)$, into 
$\mathbb{R}^3$ by including the third variable $\eps$. Then, we blow up the line $\{(0,0)\} \times \mathbb{R}$ by introducing $\bar{r}, \bar{\lambda}$, and $\bar{\eps}$ such that 
\begin{align*}
\lambda=\bar{r}\bar{\lambda}\qquad\text{and}\qquad\eps=\bar{r}\bar{\eps} 
\end{align*}
with $\bar{\lambda}^2+\bar{\eps}^2=1$, {\it i.e.}, for $(\bar{\lambda},\bar{\eps})\in S^1$, and $\bar{r} \in [0, r_0)$, where $r_0>0$.
In the blown-up space $S^1\times\mathbb{R}^2$, the line
$\{(0,0)\}\times\mathbb{R}$ is hence blown up to a cylinder $S^1\times\{0\}\times\mathbb{R}$. 

After blow-up, the curve of singular solutions obtained for $\eps=0$ consists of three portions which
correspond to singular solutions of types I, II, and III, cf. \cref{fig:soltyp}, and which are shown in 
blue, green, and black, respectively. The black curve (type III) is located in $\bar{\eps}=0$, while
the green curve (type II) lies on the cylinder, {\it i.e.}, in $\{\bar r=0\}$, with $\Vert u \Vert_2^2=\frac23$ 
constant. Finally, the blue curve (type I) consists of a branch on the cylinder, corresponding to 
$\lambda = \OO(\eps)$, and of another branch in the plane $\{\bar{\eps}=0\}$ that corresponds to 
$\lambda=\OO(1)$. In the former case, type I-solutions
resemble the one shown in the left panel of \cref{fig:soltypI}; in the second case, 
type I-solutions are as in the right panel of \cref{fig:soltypI}. These two branches correspond to 
$\mathcal{B}_1$ and $\mathcal{B}_2$, respectively, as defined 
in \cref{fig:bdsegm}. 

Loosely speaking, in blown-up space, a neighborhood of the green curve is hence covered by 
region $\mathcal{R}_2$ and part of $\mathcal{R}_3$. The blue curve is mostly covered by region 
$\mathcal{R}_1$, with a small portion close to $\delta=\frac2{\sqrt3}$ covered by $\mathcal{R}_2$.
Finally, region $\mathcal{R}_3$ covers the remainder of the green curve close to $\delta=0$, 
and the black curve. The curve obtained for $0<\eps\ll 1$, which is depicted in red in \cref{fig:bdblup}, 
lifts off from the singular curve corresponding to the limit of $\eps = 0$. 
\begin{remark}
When referring to regions $\mathcal{R}_i$, $i=1,2,3$, in blown-up space, we need to consider the 
preimages of $\mathcal{R}_i\times [0,\eps_0]$ under the blow-up transformation defined above, 
strictly speaking. However, for the sake of simplicity, we will use the two notations interchangeably.
\end{remark}

\begin{figure}[H]
\centering
\includegraphics[scale=.9]{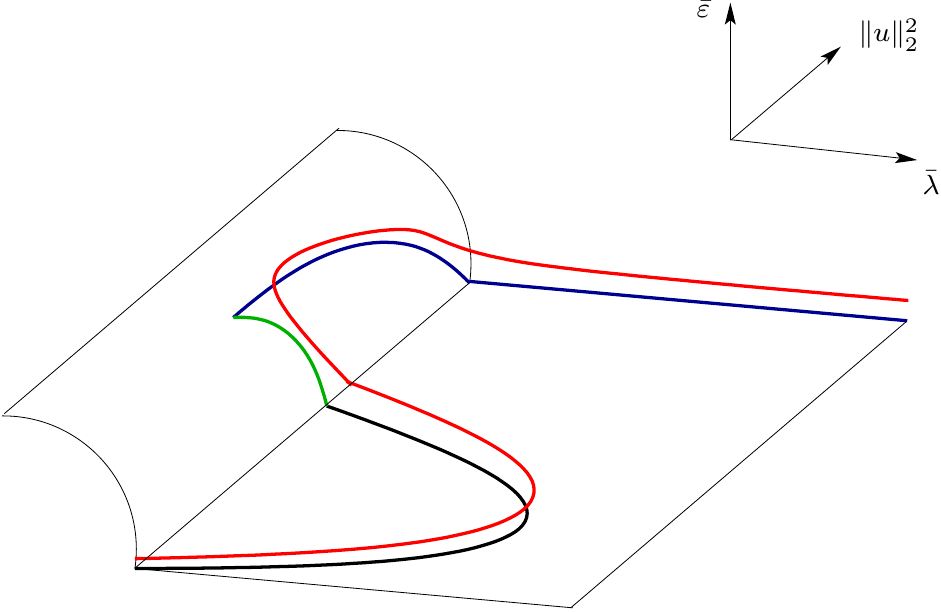}
\caption{Bifurcation diagram for the boundary value problem \{\cref{eq-7},\cref{eq-5b}\} in blown-up
parameter space. In the singular limit corresponding to $\eps=0$, the diagram consists of the union of the blue, green, and black solid curves, which are covered
by $\mathcal{R}_1$, $\mathcal{R}_2$, and $\mathcal{R}_3$, respectively.
The portion of the blue curve which lies on the cylinder corresponds to the line $\mathcal{B}_1$, while
the portion contained in the $\{\bar{\eps}=0\}$-plane corresponds to $\mathcal{B}_2$.
The red curve which lifts off from the \{$\eps=0$\}-curve represents
solutions to the boundary value problem \{\cref{eq-7},\cref{eq-5b}\} for $\eps$ positive, but small.}
\label{fig:bdblup}      
\end{figure}

As stated in \cref{thm-1}, we consider $\lambda \in [0,\Lambda]$, where we take $\Lambda=1$ 
for the sake of simplicity. In region $\mathcal{R}_3$, away from the point $B$, 
the perturbation with $\eps$ is regular. As will be shown below, singular solutions in regions 
$\mathcal{R}_1$ and $\mathcal{R}_2$ exist only for $\lambda\geq\frac34 \eps$ or, equivalently, for
$\delta\leq\frac2{\sqrt3}$; cf.~\cref{ssec:reg1,ssec:reg2}. Hence, in these 
regions, we need to take $\lambda\in\big[\frac34\eps,1\big]$, {\it i.e.},
\begin{align} \label{delrange}
\delta\in\bigg[\sqrt{\eps},\frac2{\sqrt3}\bigg],
\end{align}
which corresponds to the region shaded in gray in \cref{fig:deleps}.

\begin{figure}[H]
\centering
\includegraphics[scale=1]{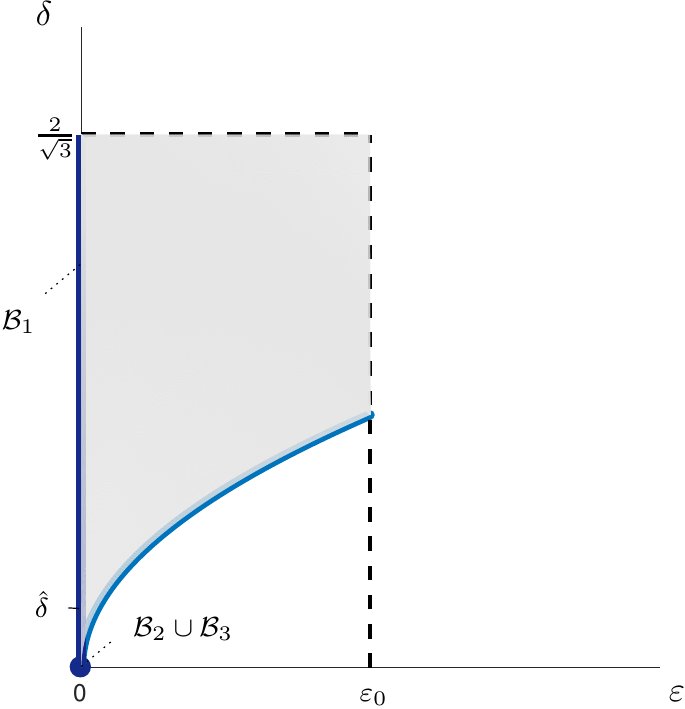}
\caption{Region in $(\eps,\delta)$-space, as considered in our analysis. The region, which is 
shaded in gray, is bounded from below by $\{\delta=\sqrt{\eps}\}$ (light blue curve) and from above 
by $\{\delta=\frac2{\sqrt3}\}$ (dashed horizontal line); cf.~\cref{delrange}. The dark blue segment of the $\delta$-axis corresponds to the union of the two portions of the blue curve shown in \cref{fig:bdsegm,fig:bdblup}. The vertical line corresponds to the line $\mathcal{B}_1$,
which is divided into two segments at some $\hat{\delta} > 0$ small, while the dot corresponds to the curve $\mathcal{B}_2\cup\mathcal{B}_3$. 
}
\label{fig:deleps}      
\end{figure}

As evidenced in \cref{fig:deleps}, $\delta=0$ occurs only when $\eps=0$, which is the point 
represented by the blue dot therein. The corresponding, highly degenerate limit gives 
a singular orbit of type I with very singular structure, as shown in the right panel in \cref{fig:soltypI}. 
Hence, the whole line $\mathcal{B}_2$ in the bifurcation diagram for $\eps=0$ shown in 
\cref{fig:bdsegm} corresponds to that one singular solution. 

\subsection{Region $\mathcal{R}_1$}\label[subsection]{ssec:reg1}

Region $\mathcal{R}_1$ in the bifurcation diagram in \cref{fig:bdsegm} corresponds to solutions 
that reduce to those of type I in the singular limit; cf.~\cref{def:soltyp}.
For $\eps$ positive and sufficiently small, solutions on that branch come very close to $\{u=\eps\}$;
moreover, the length of the interval $I$ where $u\approx\eps$ grows with $\lambda$. 
In the singular limit of $\eps = 0$, the slope of the respective solutions is moderate
for $\lambda=\OO(\eps)$, corresponding to $0<\delta<\frac{2}{\sqrt3}$, while it
tends to infinity for $\lambda=\OO(1)$ -- {\it i.e.}, as $\delta\to 0$ --
along the two segments where $u$ changes from $u = 0$ to $u=1$.
These observations are confirmed by the rescaling of $w$ in \cref{eq-6}: for 
$\lambda=\OO(\eps)$, that rescaling translates into $w=\OO(1)$, while it gives $w\to\infty$
for $\lambda=\OO(1)$; cf.~\cref{fig:soltypI}. Interestingly, the proof of our main result in this section,
which is stated below, is very similar for these two $\lambda$-regimes:
\begin{proposition}\label{prop-1}
Given $\delta_1$ fixed, with $0 <\delta_1<\frac2{\sqrt3}$ and $\delta_1\approx\frac2{\sqrt3}$, 
there exists $\eps_0>0$ sufficiently small such that in region $\mathcal{R}_1$,  
the boundary value problem \{\cref{eq-7},\cref{eq-5b}\} has a unique branch of solutions for $\eps\in(0,\eps_0)$ and 
$\lambda\in\big[\frac\eps{\delta_1^2},1\big]$. As $\eps\to 0$, these solutions limit on a singular solution $\Gamma$ of type I.
\end{proposition}
\begin{remark}
The singular solution $\Gamma$ depends on $\lambda$ or, equivalently, on $\delta$. Interpreted 
in terms of $\delta$, the range for which singular solutions exist corresponds to 
$\delta\in\big[\sqrt\eps,\delta_1\big]$; recall~\cref{delta}.
\end{remark}

To prove \cref{prop-1}, we construct solutions corresponding to the
branch that is contained in region $\mathcal{R}_1$ for fixed $\lambda$ in the regime
considered here. For $\delta$ fixed, a unique singular orbit $\Gamma$ is determined in 
blown-up phase space by investigating the dynamics of the boundary value problem \{\cref{eq-7},\cref{eq-5b}\}
separately in charts $K_1$ and $K_2$, and by then combining the results obtained in these charts. 
Finally, the singular orbit $\Gamma$, which is essentially determined by the dynamics in chart $K_2$, is shown to persist for $\eps$ positive and sufficiently small.

\subsubsection{Dynamics in chart $K_2$}

The flow of \Cref{eq-7} from the section $\Sigma_2^{\rm in}$ back to itself,
whereby the sign of $w$ changes from negative to positive, is naturally described in chart $K_2$;
cf.~\cref{fig:singorb}.

Recalling that $r_2=\eps$, we observe that \Cref{eq-13} constitutes a fast-slow system in 
the standard form of geometric singular perturbation theory \cite{Fe79, GSPT, Kue}, with $(u_2,w_2)$ the fast variables and
$\xi_2$ the slow variable. The fast system is given by \cref{eq-13}, whence the corresponding 
slow system is obtained by a rescaling of the independent variable with $r_2$:
\begin{subequations}\label{eq-47}
\begin{align}
r_2\dot{u}_2 &=u_2^4w_2, \\
r_2\dot{w}_2 &=u_2^2-1, \\
\dot{\xi}_2 &=\delta u_2^4, \\
\dot{r}_2 &=0.
\end{align}
\end{subequations}
The associated layer and reduced problems, which are obtained by setting $r_2=0$ in~\cref{eq-13} 
and~\cref{eq-47}, respectively, read
\begin{subequations}\label{eq-35}
\begin{align}
u_2' &=u_2^4w_2, \label{eq-35a}\\
w_2' &=u_2^2-1, \label{eq-35b}\\
\xi_2' &=0, \\
r_2' &=0
\end{align}
\end{subequations}
and
\begin{subequations}\label{eq-36}
\begin{align}
0 &=u_2^4w_2, \\
0 &=u_2^2-1, \\
\dot{\xi}_2 &=\delta u_2^4, \label{eq-36c} \\
\dot{r}_2 &=0,
\end{align}
\end{subequations}
respectively. (We note that the $(u_2,w_2)$-subsystem \{\cref{eq-35a},\cref{eq-35b}\} is 
precisely equal to \Cref{xy}.) The critical manifold for \Cref{eq-36} 
is given by the line
\begin{align}\label{eq-24}
\SS_2^0:=\set{(1,0,\xi_2,0)}{\xi_2\in[\xi_-,\xi_+]},
\end{align}
where the constants $\xi_\mp$ are defined as before.
\begin{remark}
While steady states are also found for $u_2=-1$ in \cref{eq-13}, these states are irrelevant,
since $u_2$ and $r_2$ are both non-negative and since $\{u_2=0\}$ is an invariant hyperplane for 
\cref{eq-13} which the flow cannot cross.
\end{remark}
Linearization of \cref{eq-35} about the critical manifold $\SS_2^0$
shows that any point $Q_2=(1,0,\xi_2,0)\in\SS_2^0$ is a saddle, with Jacobian
\begin{align*}
\Bigg[\begin{array}{cc}
4u_2^3w_2 & u_2^4 \\
2u_2 & 0
\end{array}\Bigg]\Bigg|_{(u_2,w_2)=(1,0)}=\Bigg[\begin{array}{cc}
0 & 1 \\
2 & 0
\end{array}\Bigg]
\end{align*}
and eigenvalues $\pm\sqrt{2}$. Hence, the manifold $\SS_2^0$ is normally hyperbolic. The
reduced flow thereon is described by $\dot{\xi}_2=\delta$, which corresponds
to a constant drift in the positive \mbox{$u_2$-direction} with speed $\delta$.

To describe the integrable layer flow away from $\SS_2^0$, we introduce $u_2$ as the independent 
variable, dividing \cref{eq-35b} formally by \cref{eq-35a}:
\begin{align*}
\frac{{\rm d} w_2}{{\rm d} u_2}=\frac{u_2^2-1}{u_2^4w_2(u_2)}.
\end{align*}
Solving the above equation with $w_2(1)=0$, we find
\begin{align}\label{eq-27}
w_2^\mp(u_2)=\mp\sqrt{\frac43-\frac2{u_2}+\frac2{3u_2^3}}.
\end{align}
In particular, it follows from \cref{eq-27} that, for any fixed choice of $\xi_2$, 
the stable and unstable manifolds of $Q_2$ can be written as graphs over $u_2$:
\begin{subequations}\label{eq-28}
\begin{align}
\WW_2^{\rm s}(Q_2) &=\set{(u_2,w_2^-(u_2),\xi_2,0)}{u_2\in[1,\infty)}, \\
\WW_2^{\rm u}(Q_2) &=\set{(u_2,w_2^+(u_2),\xi_2,0)}{u_2\in[1,\infty)}.
\end{align}
\end{subequations}
We have the following result.
\begin{lemma}\label{lem-1}
Let $r_2\in(0,r_0)$, with $r_0$ positive and sufficiently small. Then, the following
statements hold for \Cref{eq-47}:
\begin{enumerate}
\item The normally hyperbolic critical manifold $\SS_2^0$ perturbs to a slow manifold 
\begin{align*}
\SS_2^{r_2}=\set{(1,0,\xi_2,r_2)}{\xi_2\in[\xi_-,\xi_+]},
\end{align*}
where $\xi_\mp$ are appropriately chosen constants. In particular, we emphasize that
$(u_2,w_2)=(1,0)\in\SS_2^{r_2}$. 
\item The corresponding stable and unstable foliations $\FF_2^{\rm s}(\SS_2^{r_2})$ and
$\FF_2^{\rm u}(\SS_2^{r_2})$ are identical to $\FF_2^{\rm s}(\SS_2^0)$ and
$\FF_2^{\rm u}(\SS_2^0)$, except for their constant $r_2$-component. For $r_2\in[0,r_0)$ fixed, 
these foliations may be written as
\begin{subequations}\label{eq-29}
\begin{align}
\FF_2^{\rm s}(\SS_2^{r_2})
&=\set{(u_2,w_2^-(u_2),\xi_2,r_2)}{u_2\in[1,\infty),\ \xi_2\in[\xi_-,\xi_+]}\quad\text{and} \label{eq-29a} \\
\FF_2^{\rm u}(\SS_2^{r_2})
&=\set{(u_2,w_2^+(u_2),\xi_2,r_2)}{u_2\in[1,\infty),\ \xi_2\in[\xi_-,\xi_+]}. \label{eq-29b}
\end{align}
\end{subequations}
\end{enumerate}
\end{lemma}
\begin{proof}
Both statements follow immediately from standard geometric singular perturbation theory~\cite{Fe79}, in
combination with the preceding analysis; in particular, the fact that the plane 
$\{(u_2,w_2)=(1,0)\}$ is invariant for \Cref{eq-13} 
irrespective of the choice of 
$r_2$ implies that the restrictions of $\SS_2^{r_2}$ and $\SS_2^0$ to $(u_2,w_2,\xi_2)$-space 
do not depend on $r_2$.
\end{proof}
\begin{remark}
The fast-slow structure of \Cref{eq-13} is very simple, since the $(u_2,w_2)$-subsystem
\{\cref{eq-13a},\cref{eq-13b}\} decouples from \Cref{eq-13c}. Even for $\eps>0$, the fast dynamics is determined by that integrable planar system, and organized by the saddle point at $(1,0)$
and the stable and unstable manifolds thereof. The slow flow on the slow manifold $\SS_2^{r_2}$
is just the drift given by $\dot{\xi}=\delta$. 
\end{remark}

In the limit as $u_2 \to \infty$, $w_2^\mp(u_2)$ converges to $w_2^\mp(\infty)=\mp \frac2{\sqrt3}$; recall~\cref{eq-27}.
Transforming the stable manifold $\WW_2^{\rm s}(Q_2)$ and the unstable manifold 
$\WW_2^{\rm u}(Q_2)$ to chart $K_1$, via the coordinate change $\kappa_{21}$ defined 
in~\cref{eq-34}, we see that these manifolds limit on the points 
$\big(0,\mp\frac2{\sqrt3},\xi_1,0\big)$, respectively, for $\xi_1$ fixed; see \cref{fig:singorb}.

\subsubsection{Dynamics in chart $K_1$}

The portions of the singular orbit $\Gamma$ corresponding to the flow
between two sets of boundary conditions that are located at $\xi=\mp1$ and 
the section $\Sigma_1^{\rm out}$ are studied in chart $K_1$.
A simple calculation reveals that \Cref{eq-11} admits a line of steady states at
\begin{align}\label{eq-14}
\SS_1^0:=\set{(0,0,\xi_1,1)}{\xi_1\in[\xi_-,\xi_+]},
\end{align}
as well as the plane of steady states
\begin{align}\label{eq-15}
\pi_1:=\set{(0,w_1,\xi_1,0)}{w_1\in[w_-,w_+]\text{ and }\xi_1\in[\xi_-,\xi_+]};
\end{align}
here, $w_\mp$ and $\xi_\mp$ are defined as in \cref{eq-20}. (Another 
set of equilibria, with $\eps_1=-1$, is irrelevant to us due to our assumption that 
$r_1$ and $\eps_1$ are both non-negative.)
The line $\SS_1^0$ corresponds to the saddle equilibrium at $(\hat u,w)=(1,0)$ of 
\Cref{xy}, and coincides with the critical manifold $\SS_2^0$ introduced in chart $K_2$;
cf.~\Cref{eq-24}.

In chart $K_1$, the singular limit of $\eps=0$ corresponds to either $r_1=0$ or $\eps_1=0$
in \Cref{eq-11}, which yields the following two limiting systems in the corresponding
invariant hyperplanes:
\begin{subequations}\label{eq-16}
\begin{align}
r_1' &=0, \\
w_1' &=\eps_1(1-\eps_1^2), \label{eq-16b} \\
\xi_1' &= 0, \\
\eps_1' &=-\eps_1w_1 \label{eq-16d}
\end{align}
\end{subequations}
and
\begin{subequations}\label{eq-17}
\begin{align}
r_1' &=r_1w_1, \label{eq-17a} \\
w_1' &=0, \\
\xi_1' &=\delta r_1, \label{eq-17c} \\
\eps_1' &=0,
\end{align}
\end{subequations}
respectively. 
\Cref{eq-16} is equivalent to \Cref{eq-35} 
in chart $K_2$ under the coordinate change $\kappa_{21}$ defined in~\cref{eq-34}; these equations describe the portion of the singular orbit $\Gamma$ in chart $K_1$ that is located between 
$\Sigma_1^{\rm out}$ and the hyperplane $\{\eps_1=0\}$.
\Cref{eq-17}, on the other hand, determines the
portion of the singular orbit which connects the hyperplane $\{r_1=0\}$ with the boundary conditions imposed at
$r_1=1$. Hence, we first focus our attention on that limiting system.

The value of $w_1$ in \Cref{eq-17} is constant: $w_1\equiv w_0$, for some constant 
$w_0$. Since $w_0$ must match the $w_2$-value obtained in the limit $u_2\to\infty$
in~\cref{eq-27} in chart $K_2$, see \cref{fig:singorb},
$w_1\equiv\mp\frac{2}{\sqrt3}$ must hold in the hyperplane $\{\eps_1=0\}$.
The corresponding orbits of~\cref{eq-17} are then easily found by dividing \cref{eq-17c} 
formally by \cref{eq-17a}: $\frac{{\rm d}\xi_1}{{\rm d}r_1}=\frac{\delta}{w_0}$.
For any initial condition $\xi_1(1)=\xi_0$, the solution to that equation reads 
\begin{align}\label{eq-44}
\xi_1(r_1)=\frac{\delta}{w_0}(r_1-1)+\xi_0.
\end{align}
The boundary conditions in~\cref{eq-5b} imply $\xi_0=\mp1$; hence, and since 
$w_0=\mp\frac{2}{\sqrt3}$, we obtain
\begin{align}\label{eq-44b}
\xi_1^\mp(r_1)=\mp\frac{\sqrt3}2\delta(r_1-1)\mp1.
\end{align}
Any orbit of \cref{eq-17} can then be written as
\begin{align}\label{eq-45}
\bigg\{\bigg(r_1,\mp \frac2{\sqrt3},\xi_1^\mp(r_1),0\bigg)\, \bigg|\, r_1\in[0,1]\bigg\}.
\end{align} 

Orbits of the integrable \Cref{eq-16} can be found by
introducing $\eps_1$ as the independent variable: dividing \cref{eq-16b} formally by \cref{eq-16d},
we obtain \mbox{$\frac{{\rm d}w_1}{{\rm d}\eps_1}=-\frac{1-\eps_1^2}{w_1(\eps_1)}$}, which can be 
solved explicitly with $w_1(0)=\mp \frac2{\sqrt3}$ to yield
\begin{align}\label{eq-19}
w_1^\mp(\eps_1)=\mp\sqrt{\frac43-2\eps_1+\frac23\eps_1^3},
\end{align}
where the sign in \cref{eq-19} equals that of the initial $w_1$-value.
(We remark that \cref{eq-19} corresponds to \Cref{eq-27}, after transformation to 
$K_1$-coordinates.) 
The corresponding values of $\xi_1$ are constant, and must equal the respective values
of $\xi_1^\mp(r_1)$ in~\cref{eq-44b} at $r_1=0$, {\it i.e.},
\begin{equation}
 \xi_1^\mp(0)= \pm \frac{\sqrt3}2 \delta \mp 1.
\end{equation}
\begin{remark} \label{rem:delta}
For $\delta=\frac2{\sqrt3}$, it follows that $\xi_1^\mp(0)=0$, {\it i.e.}, we obtain a singular orbit of 
type II; see \cref{fig:soltypII,fig:3lam}. Hence, we must assume
$\delta<\frac2{\sqrt3}$ in the statement of \cref{prop-1}.
\end{remark}
Any orbit of \cref{eq-16} can thus be represented as
\begin{align}\label{eq-32}
\set{(0,w_1^\mp(\eps_1),\xi_1^\mp(0),\eps_1)}{\eps_1\in[0,\sigma]},
\end{align}
where $\sigma$ is as in the definition of the section $\Sigma_1^{\rm out}$; recall~\cref{eq-20}.

Concatenation of the two orbit segments defined in \Cref{eq-45,eq-32} 
with the respective signs will yield the singular orbits $\Gamma_1^-$ and $\Gamma_1^+$, which are located between the sections $\mathcal{V}_{1_0}^-$ and $\Sigma_1^{\rm out}$ and $\Sigma_1^{\rm out}$ and $\mathcal{V}_{1_0}^+$, respectively. Here, 
\begin{align}\label{eq-60bis}
\VV_{1_0}^\mp:=\set{(1,w,\mp1,0)}{w\in I^\mp},
\end{align}
with $I^\mp$ being appropriately defined neighborhoods of the points $w_0^-=-\frac2{\sqrt3}$ and \mbox{$w_0^+=\frac2{\sqrt3}$}, respectively; see~\cref{fig:singorb}.

\subsubsection{Singular orbit $\Gamma$}\label{sec:soO1}

A singular orbit $\Gamma$ for \Cref{eq-7} can now be constructed on the basis of
the dynamics in charts $K_1$ and $K_2$, by 
taking into account the corresponding boundary conditions in \Cref{eq-5b}.

After transformation to $K_1$, the manifolds $\WW_2^{\rm s}(Q_2)$ and $\WW_2^{\rm u}(Q_2)$ 
meet the portions of the orbits $\Gamma_1^-$ and $\Gamma_1^+$, respectively, as given 
by~\cref{eq-32}, in the points
\begin{align}\label{eq:p1mp}
P_1^\mp=\bigg(0,\mp\frac2{\sqrt3},\pm\frac{\sqrt3}{2}\delta\mp1,0\bigg).
\end{align}
These points are contained in the two lines
\begin{subequations}
\begin{align}
\ell_1^-&=\set{(0,-\tfrac2{\sqrt3},\xi_1,0)}{\xi_1\in[\xi_-,\xi_+]}\quad\text{and} \\
\ell_1^+&=\set{(0,\tfrac2{\sqrt3},\xi_1,0)}{\xi_1\in[\xi_-,\xi_+]},
\end{align}
\end{subequations}
respectively,
in the hyperplane $\{\eps_1=0\}$, which are both located in the plane of steady states 
$\pi_1$; cf.~\cref{eq-15}.
The portions of the singular orbit $\Gamma$ that lie in chart $K_1$ can hence finally be written as
\begin{subequations}\label{eq-31}
\begin{align}
\begin{split}
\Gamma_1^- &=\set{(r_1,-\tfrac2{\sqrt3},-\tfrac{\sqrt3}{2} \delta (r_1-1)-1,0)}{r_1\in(0,1]}\cup P_1^- \\
& \cup\set{(0,-\sqrt{\tfrac43-2\eps_1+\tfrac23\eps_1^3},\tfrac{\sqrt3}{2} \delta-1,\eps_1)}{\eps_1\in(0,\sigma]}\quad\text{and}
\end{split} \label{eq-31a} \\
\begin{split}
\Gamma_1^+ &=\set{(r_1,\tfrac2{\sqrt3},\tfrac{\sqrt3}{2} \delta (r_1-1)+1,0)}{r_1\in(0,1]}\cup P_1^+ \\
& \cup\set{(0,\sqrt{\tfrac43-2\eps_1+\tfrac23\eps_1^3},-\tfrac{\sqrt3}{2} \delta +1,\eps_1)}{\eps_1\in(0,\sigma]}.
\end{split} \label{eq-31b}
\end{align}
\end{subequations} 
It remains to identify the portion of $\Gamma$ that is located in chart $K_2$; we denote the
corresponding singular orbit by $\Gamma_2$. We note that, for $r_2=0$, \Cref{eq-35} 
implies $\xi_2\equiv {\rm constant}$ on $\Gamma_2$. Given the definition of $\Gamma_1^\mp$ 
and the fact that $\xi_2=\xi_1$, we define the points
\begin{align}\label{Q2}
Q_2^\mp=\bigg(1,0,\pm \frac{\sqrt3}{2} \delta \mp1,0\bigg)\in\SS_2^0;
\end{align}
therefore, we may write
\begin{align}\label{eq-37}
\Gamma_2=\WW_2^{\rm s}(Q_2^-)\cup Q_2^-\cup\set{(1,0,\xi_2,0)}{\xi_2\in(\tfrac{\sqrt3}{2} \delta-1,
-\tfrac{\sqrt3}{2} \delta+1)}\cup Q_2^+\cup\WW_2^{\rm u}(Q_2^+),
\end{align}
recall \Cref{eq-28}, where $u_2$ now varies in the range $[1,\sigma^{-1}]$.
The orbit $\Gamma_2$ is hence defined as the union of three segments, with the first being 
the stable manifold of $Q_2^-$, the second corresponding to the slow drift in $\xi_2$ from 
$Q_2^-$ to $Q_2^+$, as shown in the inset of \cref{fig:singorb}, and the third being the 
unstable manifold of $Q_2^+$. 

The sought-after singular orbit $\Gamma$, which represents the singular solution to the boundary value problem \{\cref{eq-7},\cref{eq-5b}\}, can then be written as the union of $\Gamma_1^-$, $\Gamma_2$, 
and $\Gamma_1^+$ in blown-up space:
\begin{align*}
\Gamma := \Gamma_1^- \cup \Gamma_2 \cup \Gamma_1^+.
\end{align*}
A visualization of the orbit $\Gamma$ is given in \cref{fig:singorb}.

\begin{figure}[!h]
\centering
\includegraphics[scale=0.65]{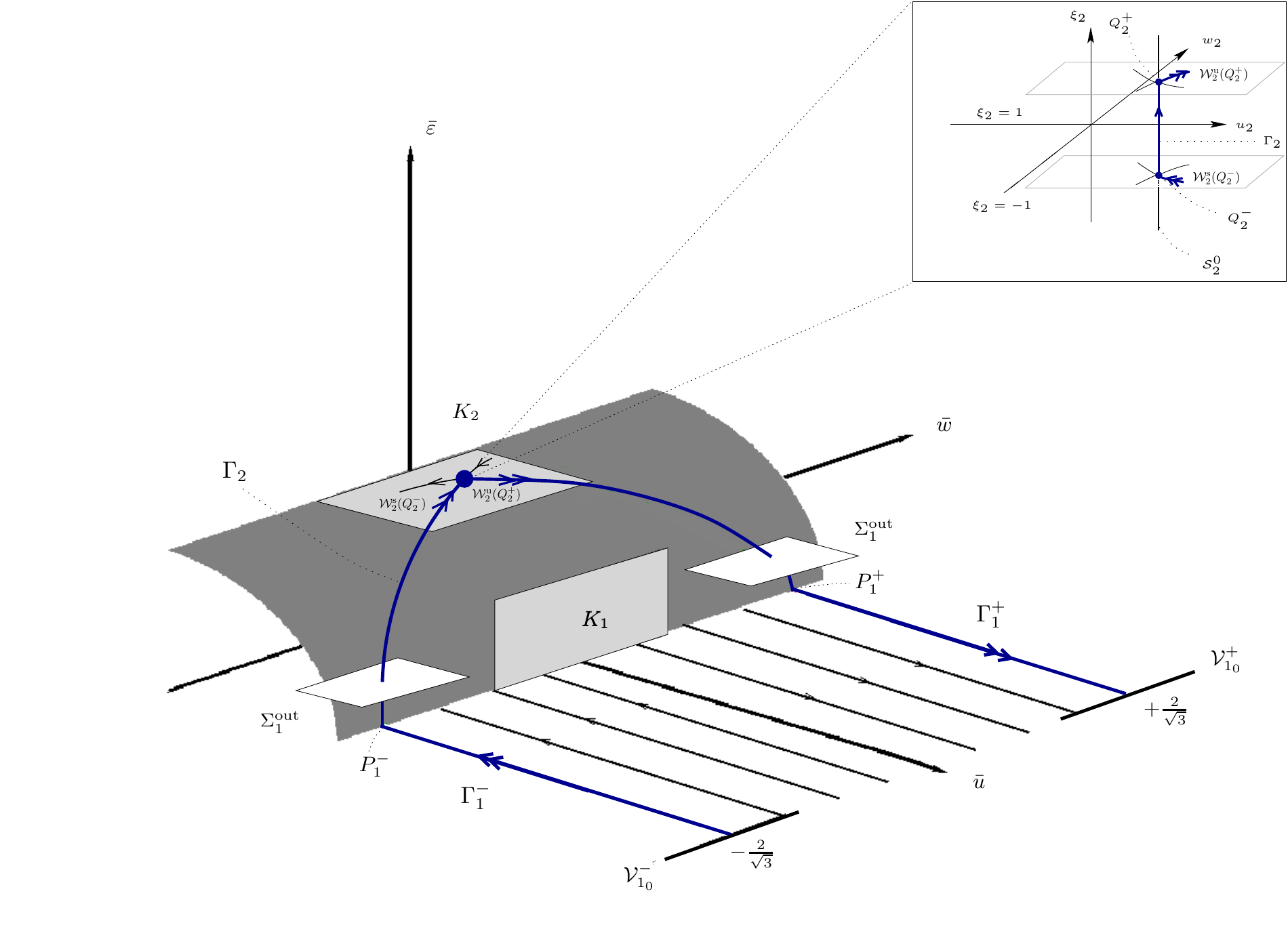}
\caption{Geometry of the singular orbit $\Gamma=\Gamma_1^-\cup\Gamma_2\cup\Gamma_1^+$ for
\Cref{eq-7} in blown-up space. The orbit $\Gamma$ is shown in blue, corresponding to a 
singular solution of type I. The inset resolves in detail the fast-slow structure
in $(u_2, w_2, \xi_2)$-space in chart $K_2$; in particular, the critical manifold $\SS_2^0$ and the
resulting singular connection between $Q_2^-$ and $Q_2^+$ is shown.}
\label{fig:singorb}
\end{figure}

\subsubsection{Persistence of $\Gamma$ -- Proof of \cref{prop-1}}

The proof of \cref{prop-1} is based on the shooting argument outlined in \Cref{sec:dynfor},
which is implemented by approximating the dynamics of \Cref{eq-7} for $\eps$ small
in the two coordinate charts $K_1$ and $K_2$. We begin by defining the two manifolds
\begin{align}\label{eq-60}
\VV_{1_\eps}^\mp:=\set{(1,w,\mp1,\eps)}{w\in I^\mp}\qquad\text{for }\eps\in[0,\eps_0),
\end{align}
which represent the boundary conditions in \cref{eq-5b} in chart $K_1$, with $r_1=1$
for $\xi_1=\mp1$; hence, it also follows that $\eps_1=\frac\eps{r_1}=\eps$ there. 
(We note that, for $\eps=0$, the manifolds $\VV_{1_\eps}^\mp$
in \cref{eq-60} reduce to $\VV_{1_0}^\mp$, respectively, as defined in \cref{eq-60bis}.)
The intervals $I^-$ and $I^+$ are defined as neighborhoods of the points 
$w_0^-=-\frac2{\sqrt3}$ and $w_0^+=\frac2{\sqrt3}$, respectively, as before.

We note that the manifolds $\VV_{1_\eps}^\mp$ are mapped onto each other by the transformation 
$(r_1,w_1,\xi_1,\eps_1)\mapsto(r_1,-w_1,-\xi_1,\eps_1)$, in accordance with the symmetry properties
of the boundary value problem \{\cref{eq-7},\cref{eq-5b}\}, as discussed in \Cref{sec:intro}. It is hence sufficient to consider 
the transition from $\VV_{1_\eps}^-$ to 
$\Sigma_1^{\rm out}$ under the flow of \cref{eq-11}, as its counterpart, the transition between 
$\Sigma_1^{\rm out}$ and $\VV_{1_\eps}^+$, can be obtained in a symmetric fashion.

We now introduce $\eps_1$ as the independent variable in
\Cref{eq-11}, whence
\begin{subequations}\label{eq-25}
\begin{align}
\frac{{\rm d}r_1}{{\rm d}\eps_1} &=-\frac{r_1}{\eps_1}, \label{eq-25a} \\
\frac{{\rm d}w_1}{{\rm d}\eps_1} &=-\frac{1-\eps_1^2}{w_1(\eps_1)}, \label{eq-25b} \\
\frac{{\rm d}\xi_1}{{\rm d}\eps_1} &=-\delta \frac{r_1(\eps_1)}{\eps_1w_1(\eps_1)}. \label{eq-25c}
\end{align}
\end{subequations}
Here, we remark that $w_1(\eps_1)$ remains non-zero for $\eps$ sufficiently small, as we know that
$w_1=\mp\frac2{\sqrt3}+\OO(\eps_1)\ne 0$ in the singular limit, {\it i.e.}, for $\eps=0$. Solving 
\Cref{eq-25a,eq-25b}, with initial condition $(1,w,-1,\eps)\in\VV_{1_\eps}^-$, 
we find
\begin{align}\label{eq-26}
r_1(\eps_1)=\frac\eps{\eps_1}\quad\text{and}\quad 
w_1^-(\eps_1)=-\sqrt{w^2+2(\eps-\eps_1)-\frac23(\eps^3-\eps_1^3)}.
\end{align}
Substituting the expressions 
in \cref{eq-26} into \cref{eq-25c} and expanding the result for $\eps_1$ small, we obtain
\begin{align*}
\frac{{\rm d}\xi_1}{{\rm d}\eps_1} &=\delta \frac\eps{\eps_1^2}
\frac1{\sqrt{w^2+2(\eps-\eps_1)-\frac23(\eps^3-\eps_1^3)}} \\
&=\delta \frac\eps{\eps_1}\frac1{\sqrt{w^2+2\eps-\frac23\eps^3}}
\bigg[\frac1{\eps_1}+\frac1{w^2+2\eps-\frac23\eps^3}\bigg]+\OO(1),
\end{align*}
which can be solved to the order considered here and evaluated in $\Sigma_1^{\rm out}$ --
{\it i.e.}, for $\eps_1=\sigma$ -- to yield
\begin{equation} \label{eq-112}
\xi_1^{{\rm out}-}=-1-\frac{\delta}{w} + \frac{\delta}{w^3}\eps\ln\eps+\OO(\eps).
\end{equation}
Similarly, evaluating \cref{eq-26} in $\Sigma_1^{\rm out}$, we find
\begin{multline*}
\big(r_1^{{\rm out}-},w_1^{{\rm out}-},\xi_1^{{\rm out}-},\eps_1^{{\rm out}-}\big) \\
=\bigg(\frac\eps\sigma,-\sqrt{w^2+2(\eps-\sigma)-\frac23(\eps^3-\sigma^3)},
-1-\frac{\delta}{w}+\OO(\eps\ln\eps),\sigma\bigg),
\end{multline*}
which defines a curve $(w_1^{{\rm out}-},\xi_1^{{\rm out}-})(w)$ that is parametrized by the initial
$w_1$-value $w$ in $\VV_{1_\eps}^-$. That curve, which we denote by $\VV_{1_\eps}^{{\rm out}-}$, 
is located in a two-dimensional subset of $\Sigma_1^{\rm out}$ and, specifically, in the 
$(w_1,\xi_1)$-plane, with $(r_1,\eps_1)$ fixed:
\begin{equation}\label{eq-40}
\VV_{1_\eps}^{{\rm out}-}:=\bigg\{\bigg(-\sqrt{w^2+2(\eps-\sigma)-\frac23(\eps^3-\sigma^3)},
-1-\frac{\delta}{w}+\OO(\eps\ln\eps)\bigg)\, \bigg| \, w\in I^-\bigg\}.
\end{equation}

It remains to study the stable foliation $\FF_2^{\rm s}(\SS_2^{r_2})$ in coordinate chart $K_2$, and to 
show that the intersection thereof with $\VV_{1_\eps}^{{\rm out}-}$ is transverse for $\eps$ 
sufficiently small. To that end, we recall the definition of $\FF_2^{\rm s}(\SS_2^{r_2})$ in
\cref{eq-29a}, which we restrict to the section $\Sigma_2^{\rm in}=\kappa_{12}(\Sigma_1^{\rm out})$:
taking $r_2(=\eps)$ fixed, as before, and evaluating $\FF_2^{\rm s}(\SS_2^{r_2})$ at 
$u_2=\sigma^{-1}$ defines a curve $\FF_2^{{\rm in}-}$ in $\Sigma_2^{\rm in}$ which is 
parametrized by $\xi_2\in[\xi_-,\xi_+]$ via 
\begin{align*}
(u_2^{\rm in},w_2^{\rm in},\xi_2^{\rm in},r_2^{\rm in})=\Big(\sigma^{-1},-\sqrt{\tfrac43-2\sigma
+\tfrac23\sigma^3},\xi_2,r_2\Big),
\end{align*}
for any $r_2\in[0,\rho\sigma]$; cf.~\cref{eq-21}. 
Transforming $\FF_2^{{\rm in}-}$ to chart $K_1$, we obtain the corresponding curve 
$\FF_1^{{\rm in}-}$:
\begin{equation}\label{eq-41}
\FF_1^{{\rm in}-}:=(w_1^{\rm out},\xi_1^{\rm out})=\bigg\{\bigg(-\sqrt{\frac43-2\sigma+\frac23\sigma^3},\xi_1\bigg)\, \bigg| \, \xi_1\in[\xi_-,\xi_+]\bigg\}.
\end{equation}
Comparing \Cref{eq-40,eq-41} and expanding
\begin{align*}
-\sqrt{w^2+2(\eps-\sigma)-\tfrac23(\eps^3-\sigma^3)}=w+\frac{\eps-\sigma}{w}+\OO[(\eps-\sigma)^2]
\end{align*}
and 
\begin{align*}
-\sqrt{\tfrac43-2\sigma+\tfrac23\sigma^3}=-\frac2{\sqrt3}+\frac{\sqrt3}2\sigma+\OO(\sigma^2),
\end{align*}
we conclude that $\VV_{1_\eps}^{{\rm out}-}$ and $\FF_1^{{\rm in}-}$ intersect in
some point
\begin{align*}
P_1^{{\rm out}-}=\bigg(-\frac2{\sqrt3}+\OO(\eps),-1-\frac{\delta}{w}+\OO(\eps\ln\eps)\bigg).
\end{align*}
As the corresponding tangent vectors in the $(w_1,\xi_1)$-plane are given by $(1,\frac{\delta}{w^2})$ 
and $(0,1)$ to leading order, that intersection is transverse for any $\eps$ small. More precisely,
transversality between $\VV_{1_\eps}^{{\rm out}-}$ and $\FF_1^{{\rm in}-}$ occurs already for 
$\eps=0$, {\it i.e.}, in $\{r_1=0\}$, which is sufficient for the Exchange Lemma to apply in chart 
$K_2$; cf.~\cref{fig:trint}. As these two curves perturb smoothly, the transversality of their 
intersection persists for $\eps \neq 0$, as well.

\begin{figure}[H]
\centering
\includegraphics[scale=0.9]{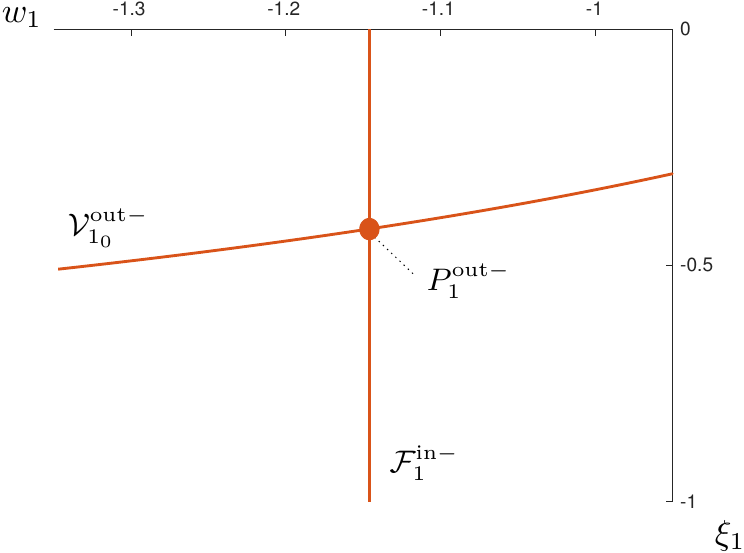}
\caption{Transverse intersection of the sets $\VV_{1_0}^{{\rm out}-}$ and $\FF_1^{{\rm in}-}$ in
$(w_1,\xi_1)$-space.}
\label{fig:trint}
\end{figure}

Next, and as stated above, the symmetry of \Cref{eq-11} implies the existence of a point
$P_1^{{\rm out}+}=\big(\frac2{\sqrt3}+\OO(\eps),1-\frac{\delta}{w}+\OO(\eps\ln\eps)\big)$ in 
$\Sigma_1^{\rm out}$ in which the curves 
\begin{align*}
\VV_{1_\eps}^{{\rm out}+}=\bigg\{\bigg(\sqrt{w^2+2(\eps-\sigma)-\frac23(\eps^3-\sigma^3)},
1-\frac{\delta}{w}+\OO(\eps\ln\eps)\bigg)\, \bigg| \, w \in I^+\bigg\}
\end{align*}
and 
\begin{align*}
\FF_1^{{\rm out}+}=\bigg\{\bigg(\sqrt{\frac43-2\sigma+\frac23\sigma^3},\xi_1\bigg)\, \bigg| \,
\xi_1\in[\xi_-,\xi_+]\bigg\}
\end{align*}
intersect transversely.

In summary, we have hence constructed a connection between the two manifolds of boundary conditions 
$\VV_{1_\eps}^-$ and $\VV_{1_\eps}^+$, as follows: in the singular limit of $\eps=0$, the image $\VV_{1_0}^{{\rm out}-}$ in $\Sigma_1^{\rm out}$ of $\VV_{1_0}^-$ under the forward 
flow intersects transversely the equivalent of the stable manifold $\WW_2^{\rm s}(Q_2)$ under 
the change of coordinates to chart $K_1$, namely, $\FF_1^{{\rm out}-}$.
Then, a slow drift occurs along the critical manifold $\SS_2^0$ until the flow leaves along the 
unstable manifold $\WW_2^{\rm u}(Q_2)$. In $\Sigma_1^{\rm out}$,
that manifold -- which corresponds to $\FF_1^{{\rm out}+}$ after transformation to $K_1$-coordinates 
-- again intersects transversely the image $\VV_{1_0}^{{\rm out}+}$ of the boundary manifold 
$\VV_{1_0}^+$ under the backward flow. The construction persists for $\eps \neq 0$ sufficiently 
small; in fact, it guarantees a transverse intersection between $\VV_{1_\eps}^{{\rm out}\mp}$ and 
$\FF_1^{{\rm out}\mp}$ when $0<\eps\ll1$.
Finally, the fact that the perturbed orbit approaching the stable foliation of the slow manifold 
$\SS_2^{r_2}$ will leave along the unstable foliation thereof is guaranteed by the Exchange Lemma.

The above argument allows us to obtain the portion of the branch of solutions in the bifurcation 
diagram which perturbs from $\mathcal{B}_1$ for $\delta \geq \hat{\delta}$, with $\hat{\delta}>0$ small; see \cref{fig:bdblup,fig:deleps}. The portion of
the branch perturbing from the part of $\mathcal{B}_1$ for $0 \leq \delta < \hat{\delta}$, as well as from $\mathcal{B}_2$, 
can be obtained in a similar spirit. However, as that regime involves the limit
as $\delta\to 0$, it requires further consideration. Setting $\delta=0$ does not affect our 
construction in chart $K_1$; however, it destroys the slow drift on $\SS_2^0$ in chart $K_2$ for 
$\eps=0$, cf.~\Cref{eq-36}. 

The segment $\mathcal{B}_2$ is associated to the regime where $\lambda=\OO(1)$. Singular 
solutions in that regime are of type I; see the right panel of
\cref{fig:soltypI}. We recall that $\delta=0$ occurs only for $\eps=0$, cf.~\cref{fig:deleps}, and that
$\delta$ is bounded below by $\sqrt\eps$. Hence, it is convenient to introduce the rescaling
\begin{equation}\label{delres}
\delta=\sqrt\eps \tilde{\delta},
\end{equation}
with $\tilde\delta \geq 1$, which we substitute into the governing \Cref{eq-11,eq-13} in charts $K_1$ and $K_2$, respectively.
In chart $K_1$, the rescaling in \cref{delres} yields the same dynamics as is obtained by setting $\delta=0$ in~\cref{eq-11}: the singular
limit of $\eps=0$ implies $\xi_1\equiv\mp1$ in the invariant hyperplane~$\{\eps_1=0\}$; 
cf.~\Cref{eq-17}. It follows that the value of $\xi$ in the transition from $u=1$ to $u=0$
does not change, as can also be seen in the corresponding type I-solution; see again the right panel of
\cref{fig:soltypI}. In chart $K_2$, introduction of the rescaling in~\cref{delres} again yields a 
fast-slow system,
\begin{subequations}\label{eq-13tilde}
\begin{align}
u_2' &=u_2^4w_2, \\
w_2' &=u_2^2-1, \\
\xi_2' &=\tilde{\delta} r_2^{\frac32}u_2^4, \\
r_2' &=0.
\end{align}
\end{subequations}
The only difference to the previous case of $\delta\neq 0$ is that the slow dynamics is now even
slower, as the small perturbation parameter in~\cref{eq-13} is given by $r_2^{3/2}$, instead of
by $r_2$. The global construction illustrated in this section is unaffected by that difference, though, 
as the techniques we have relied on -- such as, {\it e.g.}, the Exchange Lemma --
still apply. As $\tilde{\delta}$ grows to $\OO(r_2^{-1/2})$, the transition between the 
two regimes occurs.
\begin{remark}\label{rem:lam34}
We emphasize that the restriction on $\delta$ in the statement of \cref{prop-1} is due to 
the fact that we require $Q_2^-\ne Q_2^+$; cf.~also \cref{rem:delta}. Specifically, for the 
Exchange Lemma to apply, $\frac{\sqrt3}{2}\delta-1<-\frac{\sqrt3}{2}\delta+1$ must hold, which is 
equivalent to $\delta<\frac{2}{\sqrt{3}}=:\delta_\ast$. The case where that condition is violated is
studied in \cref{ssec:reg2} below, which covers region $\mathcal{R}_2$. In particular, 
it is shown there that \Cref{eq-7} then locally admits a pair of solutions which limit on 
a solution of type I and one of type II, respectively; these two singular solutions meet in a saddle-node 
bifurcation at $\delta=\delta_\ast$.
\end{remark}

\subsubsection{Logarithmic switchback}\label{sec:logsw}

In Lindsay's work~\cite{Li14}, logarithmic terms in $\eps$, as well as fractional powers of $\eps$, 
arise in the asymptotic expansions of solutions to \Cref{eq-2} as ``switchback'' terms
that need to be included during matching in order to ensure the consistency of these 
expansions~\cite{Li11}. In this subsection, we show that these terms are due to a resonance phenomenon in the blown-up vector field, see \cite{Po05}, hence establishing a connection between our dynamical systems approach and the method of matched asymptotic expansions. That connection has already been 
observed in various classical singular perturbation problems; examples include Lagerstrom's model 
equation for low Reynolds number flow~\cite{La84,La88,PS041}, front propagation in the Fisher-Kolmogorov-Petrowskii-Piscounov equation with cut-off \cite{Dumortier_2014,Dumortier_2007}, and the generalized 
Evans function for degenerate shock waves derived in~\cite{SS04}. \\
The occurrence of logarithmic switchback is necessarily studied in chart $K_1$, as 
the small parameter $\eps$ has to appear as a dynamic variable for resonances to be possible
between eigenvalues of the linearization about an appropriately chosen steady state, namely
$P_1^\mp$; recall \Cref{eq:p1mp}. Due to the symmetry properties of the 
corresponding vector field, it again suffices to restrict to the transition under the flow of \cref{eq-11} past $P_1^-=\big(0,-\frac2{\sqrt3},\frac{\sqrt3}{2}\delta-1,0\big)$ only.
\begin{proposition}
Let $\eps\in(0,\eps_0)$, with $\eps_0$ positive and sufficiently small. Then, \Cref{eq-11} admits the normal form
\begin{subequations} \label{eq:K1sysf}
\begin{align}
r_1' &=-r_1, \label{eq:K1sysfa} \\
W_1' &=\frac{3\sqrt3}8W_1^2\eps_1+\frac{27}{16}W_1\eps_1^2-\frac{5\sqrt3}{64}\eps_1^3+\OO(4),
\label{eq:K1sysfb} \\
\Xi_1' &=\frac{3\sqrt3}8\delta \eps+\frac{27}{16}\delta \eps W_1
+\frac{3\sqrt3}8\delta r_1W_1^2+\frac{27\sqrt3}{64}\delta \eps\eps_1
+\OO(4), \label{eq:K1sysfc} \\
\eps_1' &=\eps_1 \label{eq:K1sysfd}
\end{align}
\end{subequations}
in an appropriately chosen neighborhood of $P_1^-$. (Here, $\OO(4)$ denotes terms of order $4$ and upwards in $(r_1,W_1,\Xi_1,\eps_1)$.)
\end{proposition}
\begin{proof}
The proof is based on a sequence of near-identity transformations in a neighborhood 
of $P_1^-$ which reduces \Cref{eq-11} to the system of equations in \cref{eq:K1sysf}. In a first step, 
we shift $P_1^-$ to the origin, introducing the new variables $\tilde w_1$ and $\tilde\xi_1$
via $w_1=-\frac2{\sqrt3}+\tilde w_1$ and $\xi_1=\xi_1^-+\tilde\xi_1$. (Here and in the following, we write $\xi_1^-=\frac{\sqrt3}{2}\delta-1$.) Then, we divide out
a positive factor of $\frac2{\sqrt3}-\tilde w_1(=-w_1)$ from the right-hand sides in the resulting 
equations, which corresponds to a transformation of the independent variable that leaves the
phase portrait unchanged:
\begin{subequations}\label{eq-11n}
\begin{align}
r_1' &=-r_1, \\
\tilde w_1' &=\frac{\eps_1(1-\eps_1^2)}{\frac2{\sqrt3}-\tilde w_1}, \label{eq-11nb} \\
\tilde\xi_1' &=\delta \frac{r_1}{\frac2{\sqrt3}-\tilde w_1}, \label{eq-11nc} \\
\eps_1' &=\eps_1.
\end{align}
\end{subequations}
Next, we expand $\big(\tfrac2{\sqrt3}-\tilde{w_1}\big)^{-1}=\tfrac{\sqrt3}2\big(1-\tfrac{\sqrt3}2\tilde w_1\big)^{-1}=\tfrac{\sqrt3}2\big(1+\tfrac{\sqrt3}2\tilde w_1+\tfrac34\tilde w_1^2+\OO(w_1^3)\big)$ in~\Cref{eq-11nb,eq-11nc}, whence
\begin{align*}
\tilde w_1' &=\frac{\sqrt3}2\eps_1\bigg(1+\frac{\sqrt3}2\tilde w_1+\frac34\tilde w_1^2-\eps_1^2\bigg)+\OO(4), \\
\tilde\xi_1' &=\frac{\sqrt3}2\delta r_1\bigg(1+\frac{\sqrt3}2\tilde w_1+\frac34\tilde w_1^2\bigg)+\OO(4).
\end{align*}
Since none of the terms in the $\tilde w_1$-equation above are resonant, they can be removed by a sequence of near-identity transformations. For instance, setting $\tilde w_1=\hat w_1+\frac{\sqrt3}2\eps_1$, we may eliminate the linear $\eps_1$-term from 
that equation, whence
\begin{align*}
\hat w_1'=\frac34\eps_1\hat{w_1}+\frac{3\sqrt3}8\eps_1^2+\frac{3\sqrt3}8\hat w_1^2\eps_1
+\frac98\hat w_1\eps_1^2-\frac{7\sqrt3}{32}\eps_1^3+\OO(4).
\end{align*}
Similarly, we can eliminate the linear $r_1$-terms in the $\tilde\xi_1$-equation by introducing $\tilde\xi_1=\hat\xi_1-\frac{\sqrt3}2\delta r_1$; the equation for $\hat\xi_1$ then reads
\begin{align*}
\hat\xi_1'=\frac34\delta r_1\hat w_1+\frac{3\sqrt3}8 \delta \eps+
\frac{3\sqrt3}8\delta r_1\hat w_1^2+\frac98\delta \eps\hat w_1
+\frac{9\sqrt3}{32}\delta \eps\eps_1+\OO(4).
\end{align*}
The term $\frac{3\sqrt3}8\delta\eps=\frac{3\sqrt3}8\delta r_1\eps_1$ in the above equation is now resonant of order $2$, as
$(-1)+0+0+1=0$ for the eigenvalues corresponding to the monomial $r_1\eps_1$ therein; 
hence, that term cannot be eliminated in general. (Here, we note that any factor of $\eps$ 
contributes a quadratic term to the asymptotics when considered in $(r_1,\hat w_1,\hat\xi_1,\eps_1)$-coordinates.)\\
A final sequence of near-identity transformations allows us to eliminate any non-resonant second-order terms from
\cref{eq-11n}. Specifically, introducing $W_1$ and $\Xi_1$ such that 
\begin{align*}
 \hat w_1&=W_1+\frac34W_1\eps_1+\frac{3\sqrt3}{16}\eps_1^2, \\
 \hat\xi_1&=\Xi_1-\frac34\delta r_1W_1,
\end{align*}
we obtain \Cref{eq:K1sysf}, as required.
\end{proof}
Next, we outline how the normal form in \Cref{eq:K1sysf} gives rise to logarithmic (``switchback") terms in the expansion for $\xi_1$ -- or, rather, for the value $\xi_1^{\rm out}$ thereof in the section $\Sigma_1^{\rm out}$, as defined in~\cref{eq-20b}; see also Section~\ref{ssec:reg1}. In the process, we refine the approximation for $\xi_1^{\rm out}$ that was derived in the proof of Propositions~\ref{prop-1}; recall \Cref{eq-40}.
\begin{lemma}\label{lem:xi1out}
Let $\VV_{1_\eps}^-$ be defined as in \Cref{eq-60}, and consider the point $(1,w,-1,\eps)\in\VV_{1_\eps}^-$, with $w$ in a small neighbourhood of $w_0^-=-\frac2{\sqrt3}$. Then, the orbit 
of \Cref{eq-11} that is initiated in that point intersects the section $\Sigma_1^{\rm out}$ 
in a point $\big(\frac\eps\sigma,w_1^{\rm out},\xi_1^{\rm out},\delta\big)$, with
\begin{align}\label{eq-70}
\xi_1^{\rm out}=-1+\frac{\sqrt3}2\delta -\frac{3\sqrt3}8\delta \eps
\ln\eps+\OO(\delta \eps).
\end{align}
\end{lemma}
\begin{proof}
\Cref{eq:K1sysfa,eq:K1sysfd} can be solved explicitly for $r_1$ and 
$\eps_1$, which gives
\begin{align}\label{eq:r1e1}
r_1(\tilde x)=\rho\mathrm{e}^{-\tilde x}\quad\text{and}\quad
\eps_1(\tilde x)=\frac\eps\rho\mathrm{e}^{\tilde x};
\end{align}
here, $\tilde x$ denotes the rescaled independent variable that was introduced in the derivation
of \cref{eq:K1sysf}. Hence, the transition ``time" $\widetilde X$ between the 
sections $\Sigma_1^{\rm in}$ and $\Sigma_1^{\rm out}$ under the flow of \Cref{eq:K1sysf}
is given by
\begin{align}\label{eq:xel}
\widetilde X=\ln{\frac{\rho\delta}{\eps}}.
\end{align}
For the sake of simplicity, we will henceforth only consider terms of up to order $2$ in \Cref{eq:K1sysfb,eq:K1sysfc}, which gives
\begin{align}\label{eq-80a}
W_1'=0\quad\text{and}\quad\Xi_1'=\frac{3\sqrt3}8\delta 
\end{align}
to that order. Hence, solving \Cref{eq-80a} for $W_1$ and $\Xi_1$ in forward time gives
\begin{align}\label{eq-81a}
W_1\equiv W_0\quad\text{and}\quad\Xi_1=\Xi_0+\frac{3\sqrt{3}}8\delta \eps\tilde x,
\end{align}
where $W_0=W_1(0)$ and $\Xi_0=\Xi_1(0)$ are constants that remain to be determined. \\
Undoing the above sequence of near-identity transformations -- i.e., reverting to the shifted variable $\tilde\xi_1$ -- we obtain
\begin{align}\label{eq:xi1}
\tilde\xi_1=\Xi_1-\frac{\sqrt3}2\delta r_1-\frac34\delta r_1W_1
=\Xi_0+\frac{3\sqrt{3}}8\delta \eps\tilde x
-\frac{\sqrt3}2\delta r_1-\frac34\delta r_1W_1.
\end{align}
Hence, we also need to undo the transformation for $W_1\equiv W_0$;
inverting the successive transformations for the variable $w_1$, we have
\begin{align}\label{eq-82a}
 w_1=-\frac2{\sqrt3}+\Big(1+\frac34\eps_1\Big)W_0+\frac{\sqrt3}2\eps_1+\frac{3\sqrt3}{16} \eps_1^2.
\end{align}
Since $w_1\to-\frac2{\sqrt3}$ in the singular limit as $\eps_1 \to 0$, it follows that $W_0=0$ to the
order considered here. In fact, expanding the expression for $w_1(\eps_1)$ in \Cref{eq-26}
and retracing the above sequence of normal form transformations 
$w_1\mapsto\tilde w_1\mapsto\hat w_1\mapsto W_1$, we may infer from \cref{eq-81a} that 
$W_0=\tilde w_0+\OO(\eps)$, where we have written $w_0=-\frac2{\sqrt3}+\tilde w_0$ in \cref{eq-26}.
As $\tilde w_0=\OO(\eps)$, by the proof of Proposition~\ref{prop-1},
we may conclude that $W_0=\OO(\eps)$.\\
Next, substituting into~\cref{eq:xi1} and noting that 
$\Xi_0=\tilde\xi_0+\frac{\sqrt3}2\delta\rho+\OO\big(\delta \eps\big)$ 
due to $r_1=\rho$ in $\Sigma_1^{\rm in}$, we obtain
\begin{align}\label{eq:xi12}
\tilde\xi_1=\tilde\xi_0+\frac{\sqrt3}2\delta (\rho-r_1)+\frac{3\sqrt3}8
\delta\eps\tilde x+\OO(\delta \eps).
\end{align}
Reverting to the original variable $\xi=\xi_1^-+\tilde\xi_1$, we then conclude that in 
$\Sigma_1^{\rm out}$,
\begin{align}\label{eq:xi1out}
\xi_1^{\rm out}=\xi_1(\widetilde X)=\xi_0+\frac{\sqrt3}2\delta \rho
-\frac{3\sqrt3}8\delta\eps\ln\eps+\OO(\delta \eps).
\end{align}
We emphasize that the resonant term $\frac{3\sqrt3}8\delta\eps$ in \cref{eq-81a} gives rise to
$\frac{3\sqrt3}8\delta\eps\tilde x$ in \cref{eq:xi12} after integration which, for $\tilde x=\widetilde X$, yields an $\eps\ln\eps$-term in the expansion for $\xi_1^{\rm out}$.
(Here, the error estimate
in \cref{eq:xi1out} is again due to the fact that $W_1=\OO(\eps)$ throughout.)\\
It remains to approximate $\xi_0$. To that end, we consider \Cref{eq-11nc}, rewritten
with $r_1$ as the independent variable: solving
\begin{align*}
\frac{{\rm d}\tilde\xi_1}{{\rm d}r_1}=\frac{{\rm d}\xi_1}{{\rm d}r_1}=-\frac{\delta }{\frac2{\sqrt3}-\tilde w_1}
=-\frac{\sqrt3}2\delta \big(1+\OO(\tilde w_1)\big)
\end{align*}
with $\xi_1(1)=-1$ and noting that $\tilde w_1=\OO(\eps)$, by \cref{eq-81a},  we find 
\begin{align*}
\xi_1(r_1)=-1-\frac{\sqrt3}2\delta (r_1-1)+\OO\big(\delta \eps\big)
\end{align*}
and, hence,
$\xi_0=\xi_1(\rho)=-1-\frac{\sqrt3}2\delta (\rho-1)+\OO\big(\delta \eps\big)$ which, in combination
with \cref{eq:xi1out}, yields \Cref{eq-70}, as claimed.
\end{proof}
\begin{remark}
The fact that \Cref{eq-11nc} is decoupled, in combination with the structure
of the above sequence of normal form transformations $\tilde w_1\mapsto\hat w_1\mapsto W_1$
-- which depends on $\eps_1$ only -- implies that no resonances will occur in the 
corresponding expansion for $w_1^{\rm out}$. In fact, such an expansion can immediately be
derived from \cref{eq-82a}.
\end{remark}
\begin{remark}\label{rem:xc}
One can show that \Cref{lem:xi1out} is consistent with Lindsay's results \cite[Section~3]{Li14}; in fact, up to a transformation of variables, the quantity $\xi_1^{\rm out}$ corresponds to the point $-x_c$ introduced there, with $\lambda_{0c}=\frac{m-1}{2(m-2)}=\frac34$ due to $m=4$ in our case:
\begin{equation} \label{eq:xi1outli}
 -x_c = -1+ \eps^{\frac12}\bar{x_c} = -1+\frac{\sqrt3}2\delta -\frac{3\sqrt3}8\sqrt{\frac\eps\lambda}\eps
\ln\eps+\OO(\delta \eps).
\end{equation}
\end{remark}

\subsection{Region $\mathcal{R}_2$}\label[subsection]{ssec:reg2}

For $\eps >0$, region $\mathcal{R}_2$ covers a small neighborhood of the point 
$B$ in $(\lambda,\Vert u\Vert_2^2)$-space; recall \cref{fig:bdsegm}. That region 
contains the portion of the branch of solutions in the bifurcation diagram which limit on 
solutions of types I and II as $\eps\to 0$; moreover, $\mathcal{R}_2$ establishes the connection 
with the branches of solutions that are contained in regions $\mathcal{R}_1$ and $\mathcal{R}_3$. 

According to the definition in~\cref{eq:R2}, the size of $\mathcal{R}_2$ is $\eps$-dependent; in 
particular, that region collapses onto a line as $\eps \to 0$. We will show that, for $0<\eps\ll 1$, a 
saddle-node bifurcation occurs in $\mathcal{R}_2$ at $\lambda=\lambda_\ast$,
as defined in \cite{Li14}; see \cref{fig:bdfolds}.

\begin{figure}[!ht]
\centering
\includegraphics[scale=1.0]{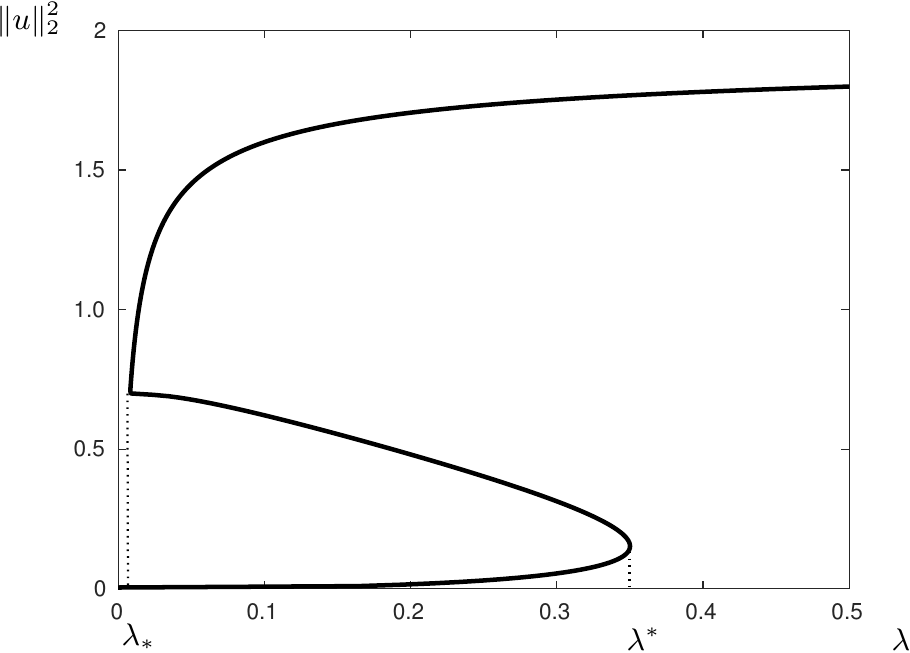}
\caption{Numerical bifurcation diagram showing solutions of~\cref{eq-2} for $\eps=0.01$. 
Saddle-node bifurcations occur at $\lambda_\ast$ and $\lambda^\ast$. }
\label{fig:bdfolds}
\end{figure}

Due to the singular dependence of $\lambda_\ast$ on the regularization parameter $\eps$, an 
accurate numerical approximation is difficult to obtain for small values of $\eps$. Using matched 
asymptotics, it was shown in~\cite{Li14} that $\lambda_{\ast}=\OO(\eps)$, 
with an expansion of the form
\begin{align*}
\lambda_{\ast}(\eps)=\lambda_{\ast 0}\eps+\lambda_{\ast 1}\eps^2\ln\eps+\lambda_{\ast 2}\eps^2 
+\OO(\eps^3).
\end{align*}
However, the coefficients $\lambda_{\ast i}$ remained undetermined there. Here, we 
confirm rigorously the structure of the above expansion, and we determine explicitly the values of the
coefficients $\lambda_{\ast i}$ therein for $i=0,1$. 
Moreover, we indicate how higher-order coefficients may be found 
systematically, and we identify the source of the logarithmic (``switchback") terms (in $\eps$) in the 
expansion for $\lambda_{\ast}$; cf.~\cref{prop-2} below.
\begin{remark}
While a saddle-node bifurcation is equally observed in the \linebreak{bi-Laplacian} case, recall~\Cref{LiBiLap}, 
Lindsay's work \cite{Li16} shows that the asymptotics of the associated $\lambda$-value 
$\lambda_\ast$ is far less singular in that case, allowing for a straightforward and 
explicit calculation of the corresponding coefficients.
\end{remark}

To leading order, $\lambda_\ast$ equals the abovementioned critical value $\frac34\eps$, 
which corresponds to $\delta_\ast=\frac2{\sqrt3}$ in terms of $\delta$. That critical $\delta$-value 
was not covered in our discussion of region $\mathcal{R}_1$ in the previous section, as the argument 
applied in that region failed there; cf.~\cref{rem:lam34}. Hence, a different argument is required
for analysing the local dynamics in a neighborhood of the saddle-node bifurcation point at 
$\delta_\ast$.

In a first step, we consider the existence of singular solutions for varying $\delta$; in particular, the 
existence of type II-solutions in region $\mathcal{R}_2$ is guaranteed by the following
\vspace{-.8cm}
\begin{lemma}\label{lem-II}
Let $\frac1{\sqrt{\lambda_2}}\leq\delta\leq \delta_1$, with $\delta_1<\frac2{\sqrt3}$. Then, a singular 
solution of type II exists if and only if $w_1=\mp \delta$ at $\xi_1=\mp 1$.
\end{lemma}
\begin{proof}
In the original model, \Cref{eq-3}, the ``touchdown'' solution of type II satisfies 
$w=\mp1$ at $x=\mp1$; cf.~\cref{def:soltyp}. After the $w$-rescaling in~\cref{eq-6}, these boundary 
conditions are equivalent to $\tilde w=\mp\delta$ at $\xi=\mp1$. The dynamics close to the 
boundary is naturally studied in chart $K_1$, which implies that $w_1=\mp\delta$ must hold 
at $\xi_1=\mp1$; cf.~\cref{eq-9a}. For $\frac1{\sqrt{\lambda_2}} \leq\delta\leq\delta_1$, and in 
contrast to the solutions of type I considered in \cref{ssec:reg1}, the
corresponding orbits can be fully studied in chart $K_1$, as they stay away from the critical manifold $\SS_2^0$ in $K_2$; recall~\Cref{eq-24}. The existence of a connecting orbit on
the blow-up cylinder between $w_1=-\delta$ and $w_1=\delta$ then follows automatically; see
the upper panel of \cref{fig:3lam}.
\end{proof}

\begin{remark}
For $\delta=0$, the type II-solution constructed in \cref{lem-II} collapses onto the line 
$\{w_1=0\}$. That case, which requires further consideration, is studied in region $\mathcal{R}_3$; 
cf.~\cref{ssec:reg3}. In fact, and as mentioned previously, both $\mathcal{R}_2$ and 
$\mathcal{R}_3$ are required to cover the green curve in \cref{fig:bdblup}.
\end{remark}

\cref{lem-II} guarantees the existence of a type II-solution for every 
$\delta\in\big[ \frac1{\sqrt{\lambda_2}},\delta_1\big]$, with $\delta_1<\delta_\ast$.
For the same range of $\delta$, {\it i.e.}, in the overlap between regions $\mathcal{R}_1$ and 
$\mathcal{R}_2$, \cref{prop-1} implies the local existence of type I-solutions. Hence, we can conclude that the boundary value problem \{\cref{eq-7},\cref{eq-5b}\} admits a pair of singular solutions for $\delta<\delta_\ast$;
one of these is of type I, while the other is of type II.  At $\delta=\delta_\ast$, the two singular 
solutions coalesce in a type II-solution. Finally, for $\delta>\delta_\ast$, no singular solution exists. 
The resulting three scenarios are illustrated in \cref{fig:3lam}. In particular, we note that solutions of 
\mbox{type I} satisfy $w_1=\mp\frac{2}{\sqrt{3}}$ -- or, equivalently,  $w=\mp\frac{2}{\sqrt{3}\delta}$ in 
the original formulation -- for $\xi_1=\mp 1$, while those of type II are
characterized by $w_1=\mp\delta$ at $\xi_1=\mp 1$, as proven in \cref{lem-II}; see again \cref{fig:3lam}.

\begin{figure}
\centering
\subfigure[]{
\includegraphics[scale=0.7]{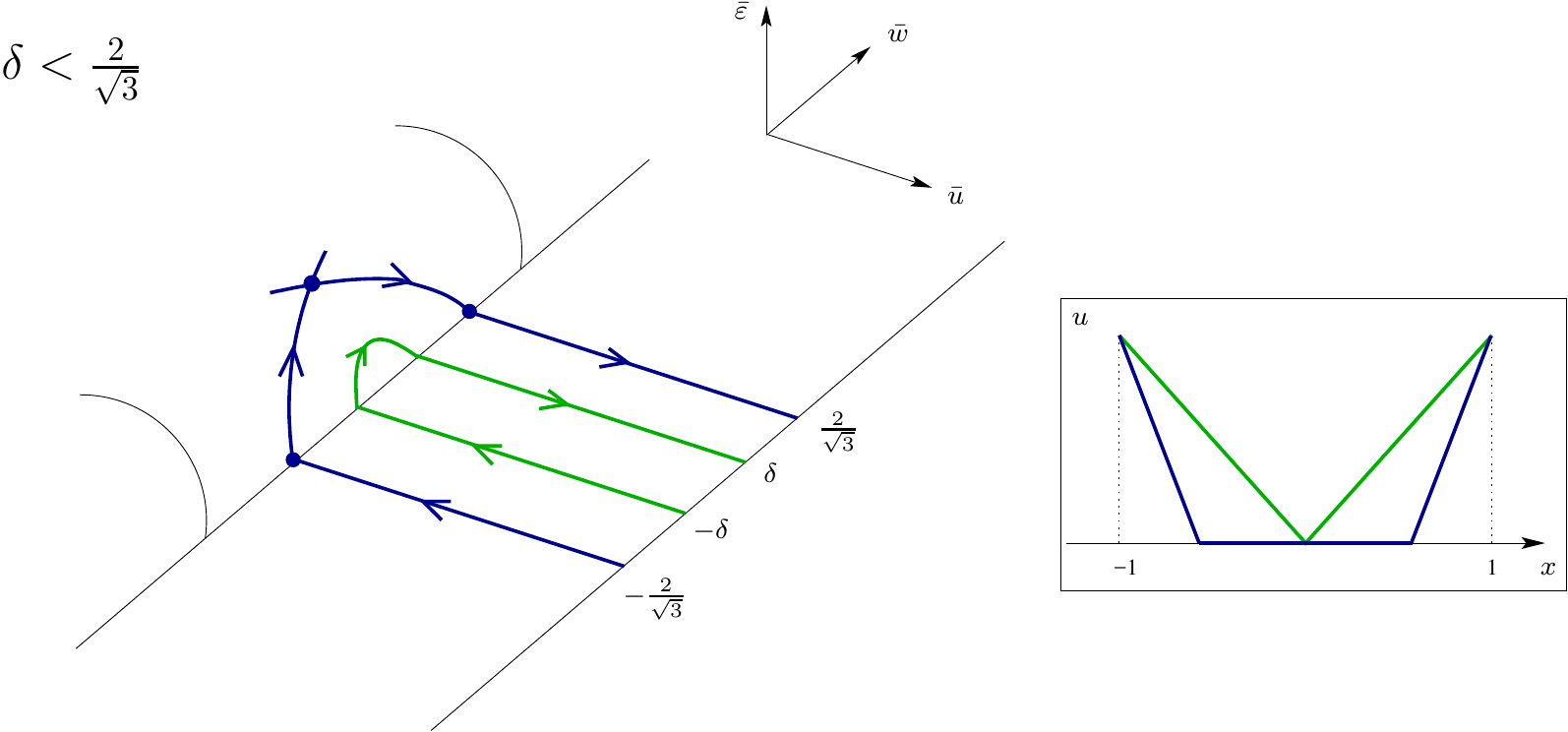}
} \\
\subfigure[]{
\includegraphics[scale=0.7]{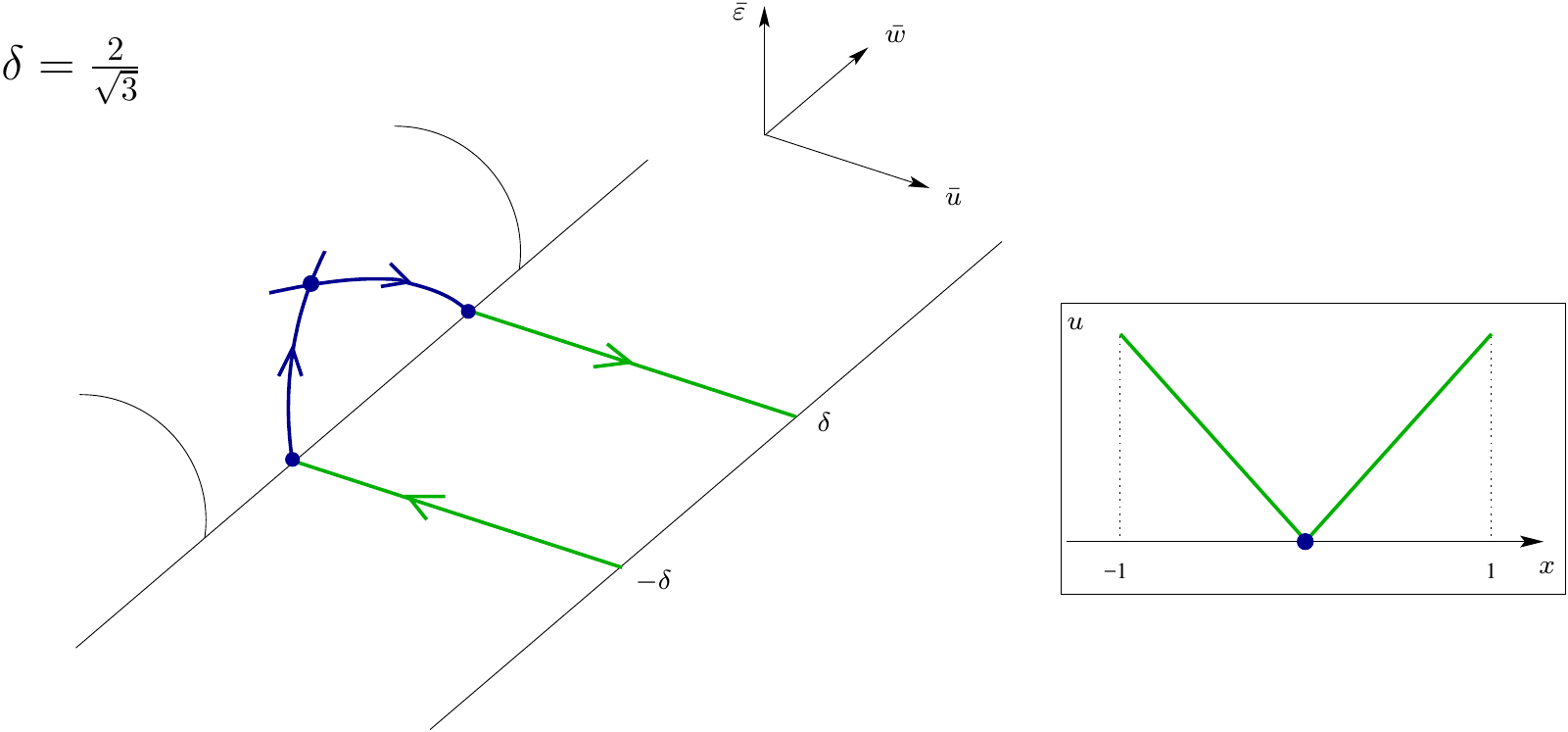}
} \\
\subfigure[]{
\includegraphics[scale=0.7]{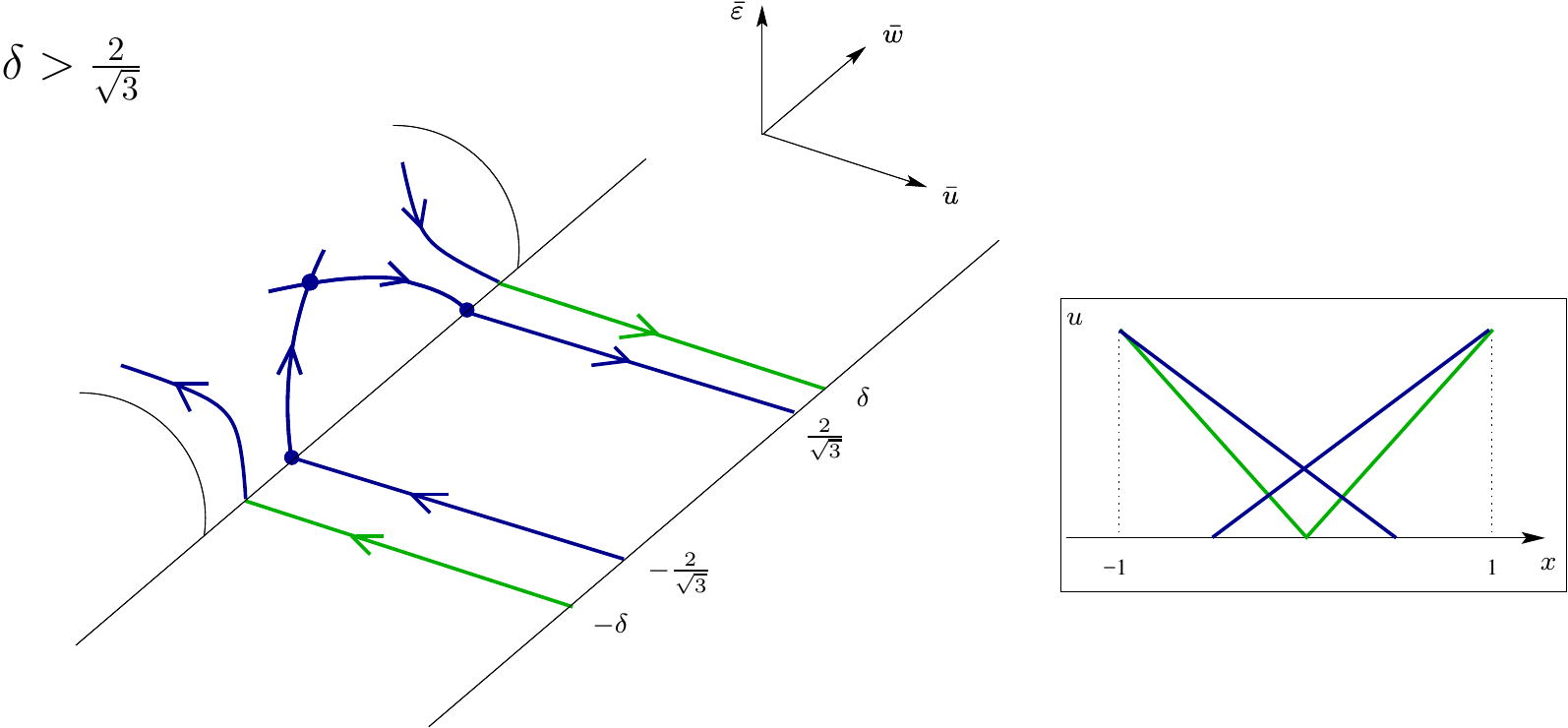}
}
\caption{Saddle-node bifurcation in the singular limit of $\eps=0$ in \Cref{eq-7} upon variation of $\delta$. In the respective insets, the corresponding singular solutions of types I (blue) and II (green) are shown. In particular, for $\delta>\frac{2}{\sqrt{3}}$, no solutions of types I or II exist.}
\label{fig:3lam}      
\end{figure}

The main result of this section is the following
\begin{proposition}\label{prop-2}
There exists $\eps_0>0$ sufficiently small such that in region $\mathcal{R}_2$, the boundary value problem \{\cref{eq-7},\cref{eq-5b}\} admits a unique branch of solutions for $\eps\in(0,\eps_0)$. 
That branch consists of two sub-branches which limit on singular solutions of types I and II, 
respectively, as $\eps\to0$.

The two sub-branches meet in a saddle-node bifurcation point at $\lambda_\ast(\eps)$, 
where two solutions exist for $\lambda>\lambda_\ast$ and 
$|\lambda-\lambda_\ast|$ small, whereas no solution exists for $\lambda < \lambda_\ast$.
Moreover, for $\eps\in(0,\eps_0)$, $\lambda_\ast$ has the asymptotic expansion
\begin{align}\label{eq-42}
\lambda_{\ast}(\eps)=\frac34\eps-\bigg(\sqrt{\frac32}+\frac98\bigg)\eps^2\ln\eps+
\OO(\eps^2).
\end{align}
The transition between regions $\mathcal{R}_2$ and $\mathcal{R}_1$
occurs as the branch of solutions limiting on solutions of type I connects to the branch already constructed in \cref{prop-1}.
\end{proposition}
\begin{proof}
The proof consists of two parts: we first consider a small neighborhood
of $\delta_\ast=\frac{2}{\sqrt{3}}$ -- {\it i.e.}, of $\lambda=\frac34 \eps$ -- where the saddle-node 
bifurcation occurs. We define a suitable bifurcation equation, which describes the transition from solutions which limit on type I-solutions to those which limit on solutions of type II.
Based on that equation, we infer the presence of the saddle-node bifurcation, and we calculate
the expansion for the corresponding $\lambda$-value $\lambda_\ast$. 

In a second step, we consider the branch of solutions that limit on type II-solutions for the remaining values of $\lambda$ in $\mathcal{R}_2$. Later, that branch will be shown
to connect to solutions that are covered by region $\mathcal{R}_3$.

We begin by constructing the requisite bifurcation equation for the first step in our proof. 
Since $w\approx-\frac2{\sqrt{3}}$ and $\delta\approx\frac{2}{\sqrt{3}}$, we write
\begin{align} \label{w0ll}
w_0=-\frac{2}{\sqrt{3}}+\Delta w\quad\text{and}\quad\delta=\frac{2}{\sqrt{3}}+\Delta\delta.
\end{align}
in chart $K_1$.

Applying the shooting argument outlined in \Cref{sec:dynfor}, we track the corresponding 
orbit from the initial manifold $\VV_{1_\eps}^-$ defined in \cref{eq-60} through $K_1$ and into the 
section $\Sigma_1^{\rm out}$; we denote that orbit by $\gamma_1^-$. In chart $K_2$, the point 
of intersection of the equivalent orbit $\gamma_2^-$ with the section $\Sigma_2^{\rm in}$ is then 
given by $(\sigma^{-1},w_2^{\rm in}, \xi_2^{\rm in},\eps)$,
for appropriately defined $w_2^{\rm in}$ and $\xi_2^{\rm in}$. 

\begin{figure}[!ht]
\centering
\includegraphics[scale=1]{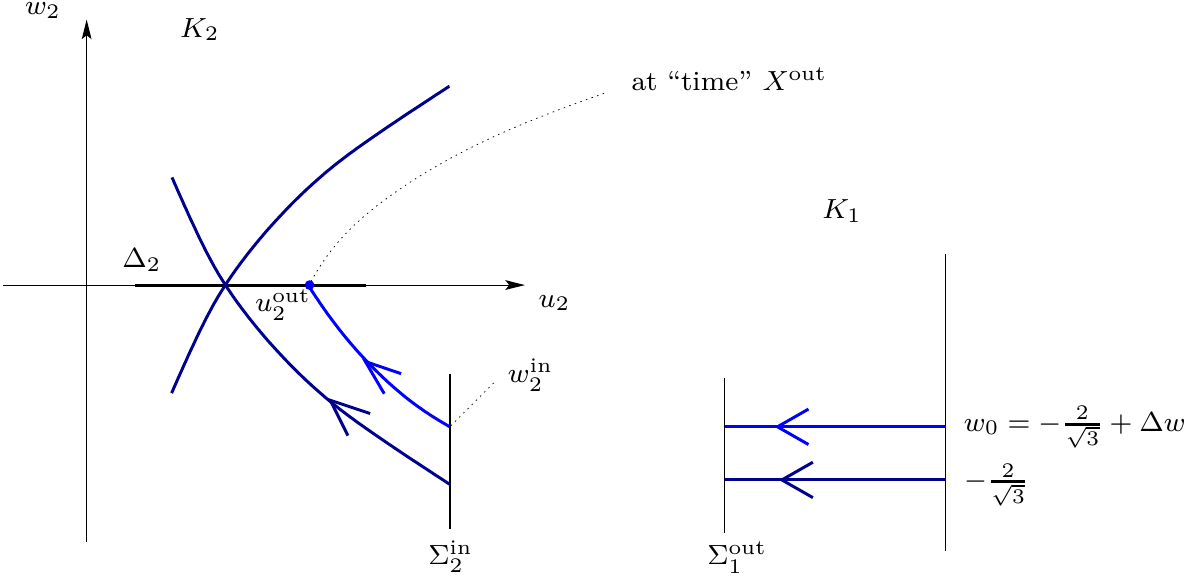}
\caption{Sketch of the shooting argument underlying the proof of \cref{prop-2}.}
\label{fig:prop3}
\end{figure}

Next, we consider the evolution of the orbit $\gamma_2^-$ through $K_2$. Let $X^{\rm out}$ denote 
the ``time'' at which $\gamma_2^-$ reaches the hyperplane $\Delta_2=\{w_2=0\}$, viz. 
$w_2(X^{\rm out})=0$. (By symmetry, it then follows that the reflection $\gamma_2^+$ of 
$\gamma_2^-$ under the map $(u_2,w_2,\xi_2,r_2)\mapsto (u_2,-w_2,-\xi_2,r_2)$ will
satisfy the boundary condition at $\VV_{1_\eps}^+$, with $w_0=\frac{2}{\sqrt{3}}-\Delta w$, after 
transformation to $K_1$.) Clearly, $X^{\rm out}$ depends on $w_2^{\rm in}$ and, in particular, on 
$\Delta w$, {\it i.e.}, on the initial deviation of the orbit from its singular limit $\Gamma_1^-$ in 
chart $K_1$.

As per our shooting argument, we need to impose the constraint that $\xi_2(X^{\rm out})=0$. Dividing 
\Cref{eq-13c} by 
\Cref{eq-13a} and recalling that $r_2=\eps$ in chart $K_2$, we find
$\frac{\mathrm{d}\xi_2}{\mathrm{d} u_2}=\frac{\delta\eps}{w_2}$ and, therefore,
\begin{align}\label{eq:csi2int}
\xi_2(u_2)=\xi_2^{\rm in}+\delta\eps\int_{u_2^{\rm in}}^{u_2^{\rm out}}\frac{1}{w_2(u_2)}\,
\mathrm{d} u_2.
\end{align}
Here, $u_2^{\rm in}=\sigma^{-1}$ as in the definition of $\Sigma_2^{\rm in}$ in \cref{eq-21}, while 
$u_2^{\rm out}$ denotes the value of $u_2$ such that $w_2(u_2^{\rm out})=0$; cf.~again 
\cref{fig:prop3}.

The sought-after bifurcation equation now corresponds to a relation between $\Delta w$, $\eps$, and 
$\delta$ that is satisfied for any solution to the boundary value problem \{\cref{eq-7},\cref{eq-5b}\} close to the 
saddle-node bifurcation in \Cref{eq-7}. To derive such a relation, we must first approximate 
$u_2^{\rm out}$: recalling the explicit expression for $w_1(\eps_1)$ on $\gamma_1^-$, as 
given in \cref{eq-26}, substituting the Ansatz made in \Cref{w0ll}, and rewriting the result in the coordinates of chart $K_2$, we find
\begin{align}\label{w2u2}
w_2(u_2) = -\sqrt{\Big(-\frac{2}{\sqrt{3}}+\Delta w\Big)^2+2\Big(\varepsilon-\frac{1}{u_2}\Big)-\frac23\Big(\varepsilon^3-\frac{1}{u_2^3}\Big)}
\end{align}
on $\gamma_2^-$. Next, we write $u_2^{\rm out}=1+\Delta u$ in \cref{w2u2}, where $\Delta u$ is 
assumed to be sufficiently small due to the fact that we stay close to the equilibrium
at $(u_2,w_2)=(1,0)$ in $K_2$. Then, we solve the resulting expression for $\Delta u$ to find three 
roots; two of these are complex conjugates, and are hence irrelevant due to the real nature of our 
problem. Expanding the third root, which is real irrespective of the value of $\Delta w$, in a series 
with respect to $\Delta w$ and $\varepsilon$, we find
\begin{align}\label{u2out}
u_2^{\rm out}=1+\frac{13\sqrt{3}}{9}\Delta w-\frac{13}{6}\varepsilon+\OO(2),
\end{align}
to first order in $\Delta w$ and $\varepsilon$.

It remains to determine the leading-order asymptotics of the integral in~\cref{eq:csi2int}. To that end, 
we expand the integrand therein as
\begin{align}\label{expintd}
\frac{1}{w_2(u_2)}=-\sqrt{\frac{3 u_2^3}{2 (u_2-1)^2 (2u_2+1)}}+\OO(\Delta w,\eps),
\end{align}
which can be shown to be sufficient to the order of accuracy considered here. (The inclusion of 
higher-order terms in \cref{eq:csi2int} would yield a refined bifurcation equation, and would 
hence allow us to take the expansion for $\lambda_\ast$ in \cref{eq-42} to higher order in $\eps$.)

Combining~\cref{expintd} and~\cref{u2out} and noting that $u_2^{\rm in}$ only enters through 
higher-order terms in $\Delta w$, which are neglected here, we finally obtain the expansion
\begin{align}\label{expint}
\int_{u_2^{\rm in}}^{u_2^{\rm out}} \frac{1}{w_2(u_2)}\,\mathrm{d}u_2=-\frac{\sqrt{2}}{2}\ln\Delta w+C
+\mathcal{O}(\Delta w) 
\end{align}
where $C$ is a computable constant.
(The above expansion reflects the fact that, as $\Delta w\to 0$, {\it i.e.}, as the point 
$(\sigma^{-1},w_2^{\rm in},\xi_2^{\rm in},\eps)$ tends to the stable manifold $\WW_2^{\rm s}(Q_2)$,
the ``time'' required for reaching $\Delta_2$ tends to infinity. Moreover, it is consistent 
with the observation 
that expansions of solutions passing close to equilibria or slow manifolds of saddle type frequently 
involve logarithmic terms.)

Next, we substitute $\xi_2^{\rm in}(=\xi_1^{\rm out})=-1-\frac{\delta}{w_0}+
\frac{\delta}{w_0^3}\eps\ln\eps+\OO(\eps)$ from~\cref{eq-112} into \cref{expint} to obtain
\begin{align} \label{csi2}
\xi_2(u_2^{\rm out})=-1-\frac{\delta}{w_0}-\frac{\sqrt{2}}{2}\delta\eps\ln\Delta w+\frac{\delta}{w_0^3}\eps\ln\eps+\OO(\eps)\stackrel{!}{=}0.
\end{align}
Shifting $w_0$ and $\delta$ by $\Delta w$ and $\Delta\delta$, cf.~\cref{w0ll}, and solving 
\cref{csi2} for $\Delta \delta$, we obtain the following bifurcation equation in 
$(\Delta w, \Delta\delta,\eps)$:
\begin{align} \label{bifeq}
\Delta\delta=-\Delta w+\frac{2\sqrt{2}}{3}\eps\ln\Delta w+\frac{\sqrt{3}}{2}\eps\ln\eps+\OO(\eps).
\end{align}

The last step consists in finding the $\Delta w$-value $\Delta w_\ast$ at which the bifurcation equation 
in~\cref{bifeq} attains its minimum, corresponding to the approximate location of the saddle-node 
bifurcation in \Cref{eq-7}, and in reverting to the original scalings. To that aim, we differentiate 
\Cref{bifeq} and solve $\frac{\mathrm{d}\Delta\delta}{\mathrm{d}\Delta w}=0$ to leading order, 
which yields $\Delta w_\ast=\frac{2\sqrt{2}}{3}\varepsilon$; see~\cref{fig:bdsn}.

Substituting into~\cref{bifeq}, we obtain the corresponding value of 
$\Delta\delta_\ast$, which implies 
$\lambda_\ast=\frac{\eps}{\delta_\ast^2}=\eps\big(\frac{2}{\sqrt{3}}+\Delta\delta_\ast\big)^{-2}$
by \Cref{w0ll}. Hence, we find the desired asymptotic expansion for $\lambda_\ast$, viz.
\begin{align} \label{snexp}
\lambda_{\ast}(\eps)=\frac34\eps-\bigg(\sqrt{\frac32}+\frac98\bigg)\eps^2\ln\eps+\OO(\eps^2),
\end{align}
as claimed. Finally, since 
\begin{align*}
\frac{\mathrm{d}^2\Delta\delta}{\mathrm{d}(\Delta w)^2}\bigg|_{\Delta w=\Delta w_\ast}=
-\frac{2 \sqrt{2}}{3} \frac{\eps}{(\Delta w_\ast)^2}
\end{align*}
is negative, the function $\Delta\delta(\Delta w)$ is locally concave, which implies that the unfolding of 
solutions to \Cref{eq-7} for $|\lambda-\lambda_\ast|$ small is as given in the statement of 
the proposition; see~\cref{fig:bdsna}.
In particular, the branch of solutions which limits on solutions of type I overlaps with the one
contained in region $\mathcal{R}_1$, as $\delta_1$ can be chosen arbitrarily close to
$\frac{2}{\sqrt{3}}$ in the statement of \cref{prop-1}.

The last part of the proof concerns the existence of solutions to the boundary value problem \{\cref{eq-7},\cref{eq-5b}\}
which limit on type II-solutions as $\eps\to 0$ for the remaining values of $\lambda$ in 
$\mathcal{R}_2$, {\it i.e.}, for $\delta\in\big(\frac1{\sqrt{\lambda_2}},\delta_1\big)$. The existence of 
singular solutions of type II in that range is ensured by \cref{lem-II}.
In the singular limit, {\it i.e.}, for $\eps=0$, we have transversality at $\xi_1=0$ with respect to 
variation of $w_1$ at $\xi_1=\mp 1$ around $\mp\delta$. Hence, the corresponding singular solutions 
perturb to solutions of the boundary value problem \{\cref{eq-7},\cref{eq-5b}\} for $0<\eps \ll 1$, which completes the proof.
\end{proof}
\begin{remark}
The branch of solutions derived in the last part of the proof is still described by the bifurcation equation in
\cref{bifeq}, the difference being that the $\eps \ln\Delta w$-term
is now regular, \emph{i.e.}, $\OO(\eps)$, due to $\Delta w=\OO(1)$. The above proof also implies that $\Delta\delta$ must be larger than $\OO(\eps)$; in fact,
Lindsay's work \cite{Li14} shows that $\Delta\delta=\OO(\eps\ln\eps)$.
\end{remark}
\begin{remark}
The presence of an $\eps \ln \eps$-term in the bifurcation equation
\cref{bifeq} implies that the convergence to the singular limit of $\eps=0$ cannot be 
smooth in $\eps$; rather, it will be regular in $(\eps, \ln \eps)$. A similar situation was
encountered in \cref{prop-1} above, where the presence of logarithmic switchback terms 
in $\eps$ was observed; recall \cref{sec:logsw}. Here, we emphasize that the source of these terms in \cref{bifeq} is two-fold: in addition to switchback due to a resonance in chart $K_1$, logarithmic terms are also introduced through the passage of the flow past the saddle point at $(1,0)$ in $K_2$, as is evident from the $\eps\ln\Delta w$-term in \Cref{expint}. In particular, both contributions manifest in the expansion for $\lambda_\ast(\eps)$ in \Cref{snexp}.
\end{remark}

\begin{figure}[H]
\centering
\subfigure[Local view.]{      
\includegraphics[scale=0.8]{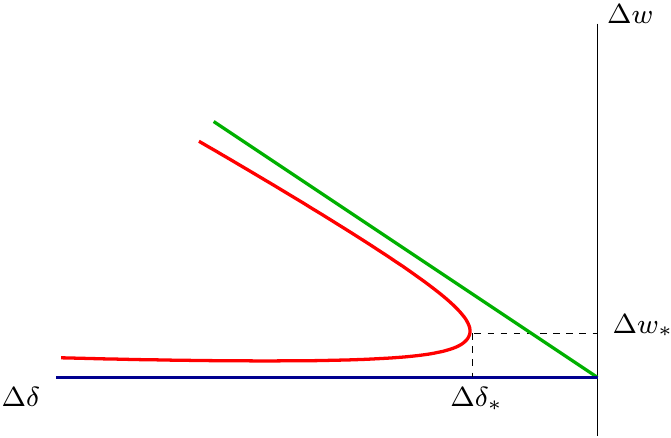} \label{fig:bdsna}   
}
\subfigure[Global view.]{ 
\includegraphics[scale=0.7]{sn4b2.pdf} \label{fig:bdsnb}   
}
\caption{Illustration of the saddle-node bifurcation at $\lambda_\ast$ in \Cref{eq-7}. The red curve corresponds to the 
case of $\eps\neq 0$, while the singular limit of $\eps=0$ is represented in blue (type I),
green (type II), and black (type III). The point of intersection of the green and blue curves 
corresponds to the critical $\delta$-value $\delta_\ast$; a small neighborhood of that point where 
the transition between these two curves occurs is considered in the first part of \cref{prop-2}, while the 
remainder of the green curve -- up to an arbitrarily small, but fixed distance from the 
intersection with the black curve -- is studied in the second part of \cref{prop-2}. Finally, the transition 
between the green and black curves is described in \cref{ssec:reg3} below.}
\label{fig:bdsn}      
\end{figure}

The asymptotic expansion for $\lambda_\ast$ in~\cref{eq-42} shows excellent agreement with 
numerical values that were obtained using the continuation software 
package~\texttt{AUTO}~\cite{AU}; see \cref{fig:errcom}. In particular, the distance between 
the two curves is  $\OO(\eps^2)$, {\it i.e.}, of higher order in $\eps$, as postulated. 

\begin{figure}[!h]
\centering
\includegraphics[scale=1.0]{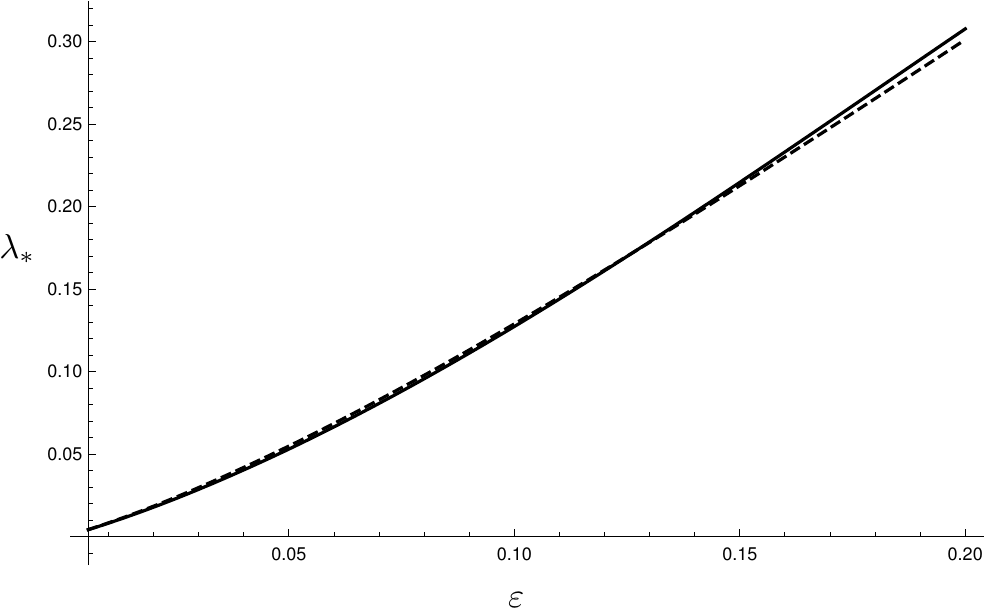}
\caption{Comparison between the asymptotic expansion for $\lambda_\ast(\eps)$ in~\cref{eq-42} 
(solid curve) and numerical values obtained with~\texttt{AUTO} (dashed curve).}
\label{fig:errcom}
\end{figure}

\subsection{Region $\mathcal{R}_3$} \label[subsection]{ssec:reg3}

It remains to analyse region $\mathcal{R}_3$, which contains the branch of solutions in the bifurcation
diagram that perturb from type III-solutions, corresponding to the non-regularized problem
\begin{align}\label{eq:nrm}
u''=\frac{\lambda}{u^2},\qquad\text{ for }x\in[-1,1],\text{ with }u=1\text{ when }x=\mp1.
\end{align}
By \cref{def:soltyp}, solutions of type III differ from those of types I and II, in that they do not exhibit 
touchdown phenomena. Regularization affects them only weakly, {\it i.e.}, in a regular fashion, with
the effect becoming slightly more pronounced as $\lambda \to 0$; cf.~\cref{fig:segm3}. 
Thus, most of the solutions contained in region $\mathcal{R}_3$ perturb from 
$\mathcal{B}_3$ in a regular way, and are hence easy to obtain.
The limit of $\lambda \to 0$, {\it i.e.}, the transition from $\mathcal{R}_3$ to $\mathcal{R}_2$, 
needs to be treated more carefully.
\begin{remark}
It is easy to see that \Cref{eq:nrm} -- or, rather, the corresponding first-order system -- is
Hamiltonian; the level curves of the associated Hamiltonian are given precisely 
by the singular solutions in panel (b) of \cref{fig:segm3}.
\end{remark}

\begin{figure}[!h]
\centering
\includegraphics[scale=0.815]{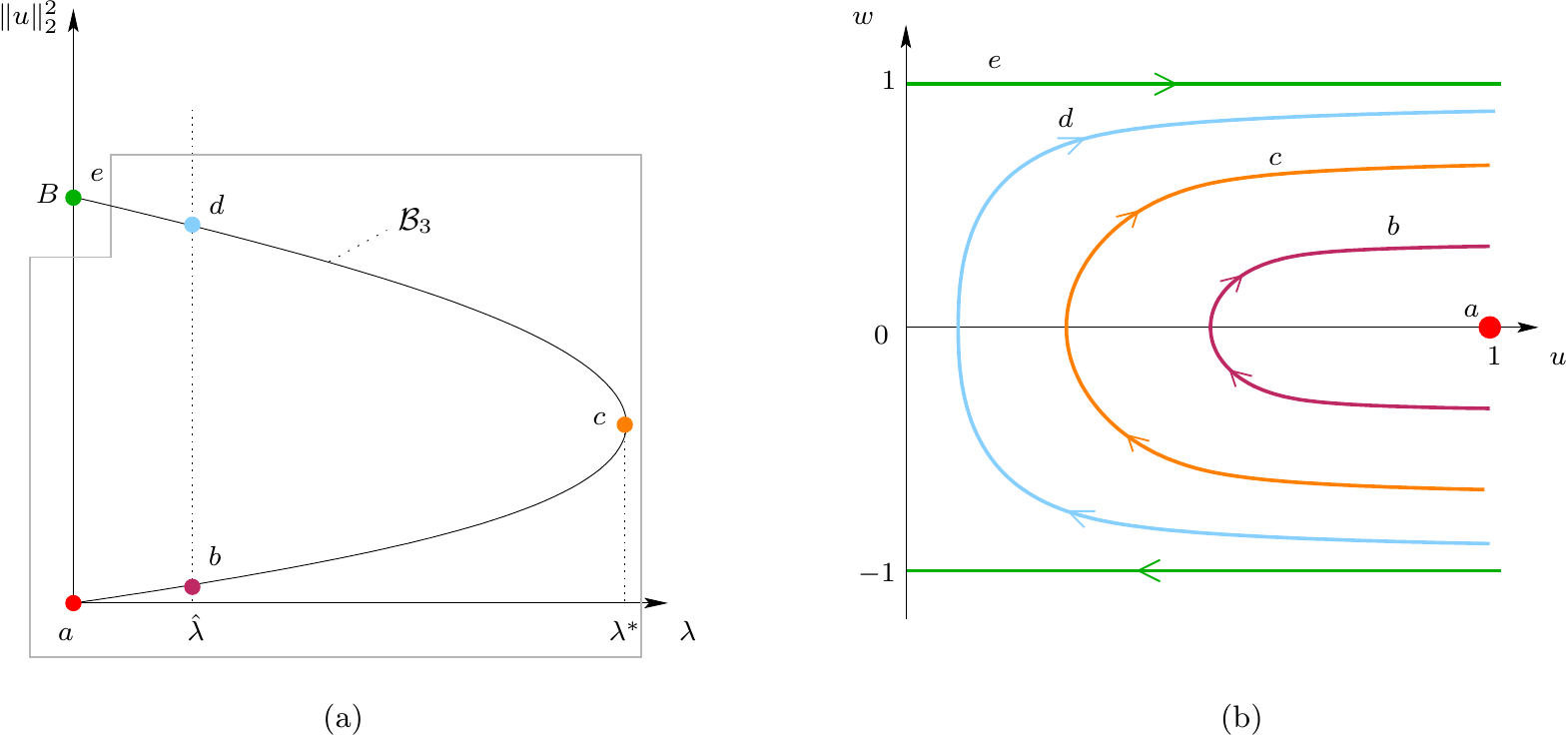}
\caption{(a) Covering of the curve $\mathcal{B}_3$ by region $\mathcal{R}_3$ for $\eps > 0$;
(b) the corresponding singular solutions in the original $(u,w)$-space. The green solution is of 
type II, and corresponds to the limit as $\lambda \to 0$. Recall that, for $\eps \to 0$, $\mathcal{R}_3$ 
approaches the line $\{\lambda=0\}$ near the point $B$. 
Singular solutions corresponding to $\lambda>0$ are of type III and do not exhibit touchdown 
phenomena. The orange solution realized at the fold point $\lambda=\lambda^\ast$ is the one 
where the two parts of the curve $\mathcal{B}_3$ meet.}
\label{fig:segm3}
\end{figure}

Type III-solutions are contained in the curve $\mathcal{B}_3$ in the limit of $\delta=0$; see 
\cref{fig:deleps}. That limit was not covered in region $\mathcal{R}_2$, as the approach used 
there required the assumption that $\delta\geq \frac1{\sqrt{\lambda_2}}$. The limit as $\delta\to 0$, 
however, results in singular dynamics in chart $K_1$, as the type II-solution (green)
-- corresponding to $w_1=\mp \delta$ at $\xi_1=\mp1$ -- collapses onto the line
\begin{align}\label{eq:M10}
\MM_1^0:=\set{(r_1,0,\xi_1,0)}{r_1\in\mathbb{R}^+,\ \xi_1\in\mathbb{R}};
\end{align}
see \cref{fig:SysNTT:b} and the upper panel of \cref{fig:3lam}. Clearly, $\MM_1^0$ constitutes a 
line of non-hyperbolic equilibria for \Cref{eq-11} which corresponds to the manifold 
$\MM^0$ in~\cref{M0}, after blow-down. The singular nature of $\MM_1^0$ is related 
to the rescaling of $w$ introduced in~\cref{eq-6}. That rescaling, which corresponded to a ``zooming 
out'', turned out to be particularly useful for our analysis in regions $\mathcal{R}_1$ and 
$\mathcal{R}_2$. However, it cannot provide a good description of region $\mathcal{R}_3$. To 
study the dynamics in $\mathcal{R}_3$, we would have to perform another blow-up involving 
$\delta$, $w_1$, and $\eps_1$ in chart $K_1$ in order to basically undo the $w$-rescaling 
in~\cref{eq-6}. 
It is much simpler to consider the $\delta$-range covered by $\mathcal{R}_3$ by returning to the 
original system without any rescaling of $w$; cf.~\Cref{eq-4}.

The main result of this section is the following
\begin{proposition}\label{prop-3}
There exists $\eps_0>0$ sufficiently small such that in region $\mathcal{R}_3$,
the boundary value problem \{\cref{eq-4},\cref{eq-5}\} admits a unique branch of solutions for $\eps\in(0,\eps_0)$. 
Outside of a fixed neighborhood of the point $B$, that branch converges smoothly
as $\eps\to 0$ to the curve $\mathcal{B}_3$ along which solutions of the 
non-regularized boundary value problem, \Cref{eq:nrm}, exist. In the $\eps$-dependent region overlapping with $\mathcal{R}_2$, the branch of solutions limiting on solutions of type II described in \cref{prop-2} is recovered. There, the transition
from solutions that limit on type-III solutions to those limiting on singular solutions of type II occurs.
\end{proposition}

\begin{proof}
We recall the original first-order system, \Cref{eq-4}:
\begin{align*}
u' &=u^4w, \\
w' &=\lambda(u^2-\eps^2), \\
\xi' &=u^4, \\
\eps' &=0;
\end{align*}
given \Cref{delta}, we write $\eps=\delta^2 \lambda$ and obtain the equivalent system
\begin{subequations}\label{sysdel}
\begin{align}
u' &=u^4w, \\
w' &=\lambda(u^2-\delta^4 \lambda^2), \\
\xi' &=u^4, \\
\delta' &=0.
\end{align}
\end{subequations}
Here, the parameter $\delta$ plays the role of the small perturbation parameter, with the 
$\delta$-range corresponding to region $\mathcal{R}_3$ given by
\begin{align*}
\delta\in\Big[0,\frac1{\sqrt{\lambda_3}}\Big];
\end{align*}
cf.~\cref{eq:R3}. In summary, it is hence more convenient to consider $\lambda$ 
and $\delta$, rather than $\lambda$ and $\eps$, as the relevant parameters in this regime.

For $\delta=0$ and $\lambda>0$, the projection of the flow of \Cref{sysdel} is as illustrated
in \cref{fig:SysNTT:a}. In region $\mathcal{R}_3$, however, we are also interested in covering a 
small neighborhood of $\lambda=0$, which again gives the singular dynamics shown in 
\cref{fig:SysNTT:b}. In $(u,w)$-space, the singular solution found for $\lambda=0$ consists of 
$[0,1]\times\{-1\}$ and $[0,1]\times\{1\}$, {\it i.e.}, it approaches the degenerate line of equilibria 
for \cref{sysdel} at $\{(0,w)\}$ under the forward and backward flow in $x$, respectively; 
see \cref{fig:flg}.

\begin{figure}[H]
\centering
\includegraphics[scale=0.9]{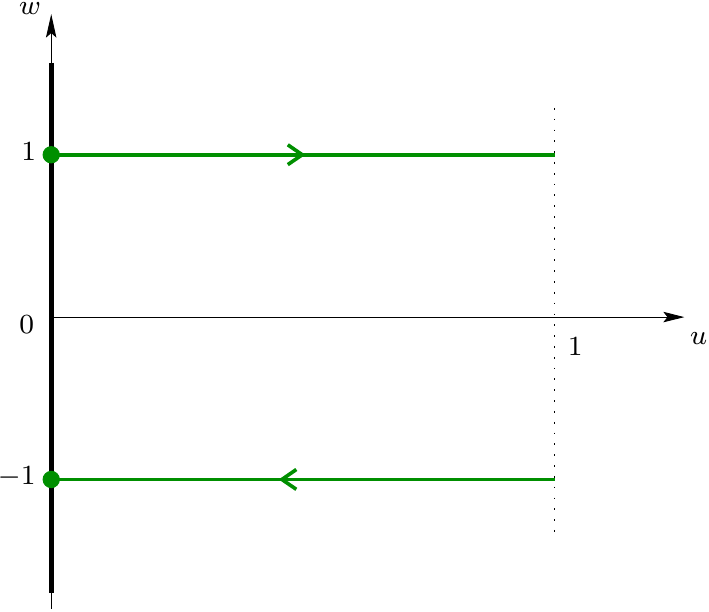}
\caption{Singular solution of \Cref{sysdel} for $(\delta,\lambda)=(0,0)$ in $(u,w)$-space. 
That solution, which is of type II, cf.~\cref{fig:soltypII}, is shown in green. The solid black 
line represents the degenerate line of equilibria at $\{(0,w)\}$.}
\label{fig:flg}
\end{figure}
To analyze the dynamics close to that line, we have to introduce a blow-up of $(u,\lambda)=(0,0)$.
As the blow-up involves $\lambda$, we append the trivial equation $\lambda'=0$ to
\cref{sysdel}:
\begin{subequations} \label{sysdele}
\begin{align}
u' &=u^4w, \\
w' &=\lambda(u^2-\delta^4 \lambda^2), \\
\xi' &=u^4, \\
\lambda' &=0, \\
\delta'&=0.
\end{align}
\end{subequations}
The requisite blow-up transformation is then given by
\begin{align}\label{buR3}
u=\bar{r}\bar{u}\qquad\text{and}\qquad\lambda=\bar{r} \bar{\lambda},
\end{align}
where $(\bar{u},\bar{\lambda}) \in S^1$, {\it i.e.}, $\bar{u}^2+\bar{\lambda}^2=1$, 
and $\bar{r} \in [0, r_0)$, with $r_0>0$. 
We denote the chart corresponding to $\bar{u}=1$ by $\kappa_1$. The analysis in that 
chart turns out to be sufficient for proving \cref{prop-3}.
In chart $\kappa_1$, the blow-up transformation in~\cref{buR3} reads
\begin{align}\label{kappa1c}
u=r_1\qquad\text{and}\qquad\lambda=r_1 \lambda_1.
\end{align}
which gives
\begin{subequations} \label{sysdel1}
\begin{align}
r_1' &=r_1 w, \\
w' &=\lambda_1(1-\delta^4 \lambda_1^2),\\
\xi' &=r_1, \\
\lambda_1' &= -\lambda_1 w, \\
\delta' &=0
\end{align}
\end{subequations}
for \Cref{sysdele}; here, $\delta$ is the small (regular) perturbation parameter. For any
$\lambda \in [0,1]$, the existence of solutions to \cref{sysdel1} can be studied via the symmetric 
shooting argument outlined in \Cref{sec:dynfor}. To that end, we define a set of initial conditions at 
$(r_1,\xi)=(1,-1)$, as follows:
\begin{align}\label{Vl}
\mathcal{V}_\lambda=\set{(1,w_0,-1,\lambda,\delta)}{w_0 \in I},
\end{align}
where $I$ is a neighborhood of $w=-1$. We remark that the initial value $\lambda$ for 
$\lambda_1$ follows from $\lambda=r_1 \lambda_1$, cf.~\cref{kappa1c}, as $r_1=1$
initially. Next, we introduce $w$ as the independent variable in \cref{sysdel1}, whence
\begin{subequations}\label{sysdel1w}
\begin{align}
\frac{{\rm d} r_1}{{\rm d} w} &=\frac{r_1 w}{\lambda_1(1-\delta^4 \lambda_1^2)}, \\
\frac{{\rm d} \xi}{{\rm d} w} &=\frac{r_1}{\lambda_1(1-\delta^4 \lambda_1^2)}, \\
\frac{{\rm d} \lambda_1}{{\rm d} w} &= -\frac{w}{1-\delta^4 \lambda_1^2}, \\
\frac{{\rm d} \delta}{{\rm d} w} &=0,
\end{align}
\end{subequations}
with initial conditions
\begin{align}\label{sysdelbc}
r_1(w_0)=1,\quad\xi(w_0)=-1,\quad\lambda_1(w_0)=\lambda,\quad\text{and}\quad\delta(w_0)=0.
\end{align}
We track $\mathcal{V}_\lambda$ under the flow of \cref{sysdel1w} up to the hyperplane $\{w=0\}$; see \cref{fig:Vl}. There, we obtain a point 
$(r_1^{\rm out},0,\xi^{\rm out},\lambda_1^{\rm out},\delta)$ in $(r_1,w,\xi,\lambda_1,\delta)$-space.
Our shooting argument implies that we have to solve the equation
\begin{align}\label{xiout11}
\xi^{\rm out}(w_0,\lambda,\delta)=0.
\end{align}
\begin{figure}[H]
\centering
\includegraphics[scale=0.7]{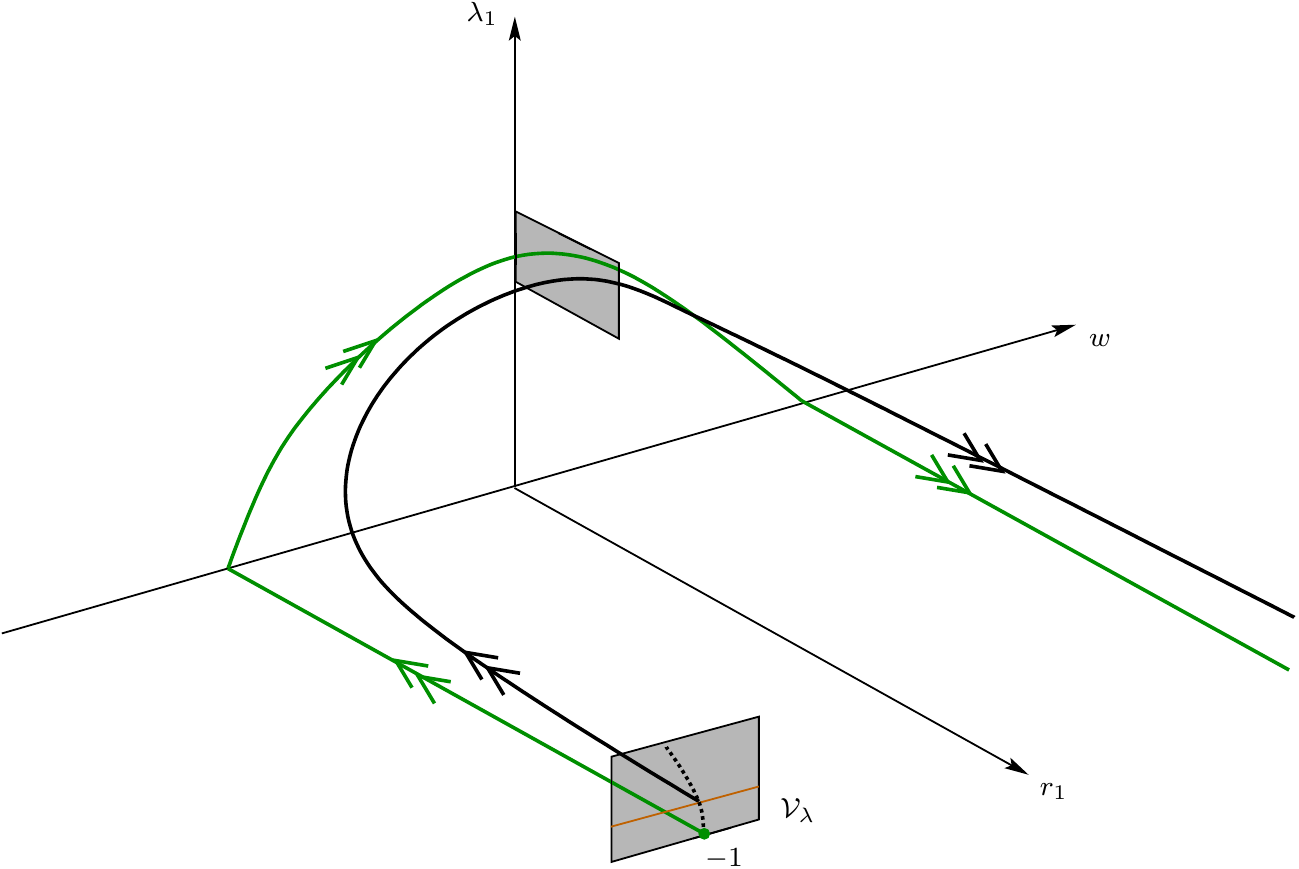}
\caption{Dynamics of \Cref{sysdel1} in $(r_1,w,\lambda_1)$-space. The gray 
section at $r_1=1$ corresponds to $\mathcal{V}_\lambda$, cf.~\cref{Vl}, which is flown forward to 
$\{w=0\}$. The green orbit represents the singular solution for $\lambda=0$, {\it i.e.}, a singular 
solution of type II, which satisfies $w=-1$ at $\xi=-1$. The black orbit corresponds to a 
solution of \cref{sysdel1} with initial conditions in $\mathcal{V}_{\lambda}$ for a fixed value of 
$\lambda >0$ and $\delta=0$, which is a solution of type III. The dashed curve contained in 
$\mathcal{V}_\lambda$ corresponds to the set $\{w_0=w_0(\lambda)\}$ that solves 
\Cref{w0A}. That set is defined by \Cref{w0A2} for $\delta=0$. 
The orange line indicates a section through $\mathcal{V}_{\lambda}$ for fixed $\lambda > 0$. 
}
\label{fig:Vl}
\end{figure}
At this point, we split $\mathcal{R}_3$ into two subregions in which we apply separate arguments
to prove the existence of a unique branch of solutions, as claimed in the statement of the proposition.
For $\lambda\geq\tilde{\lambda}$, with $\tilde{\lambda}$ fixed and positive, and $\delta=0$, 
\Cref{sysdel1w} can be solved explicitly subject to \cref{sysdelbc}; moreover, a solution 
$w_0=w_0(\lambda)$ of \Cref{xiout11} can be proven to exist for 
$\lambda\leq\lambda^\ast$. At $\lambda=\lambda^\ast$, transversality
breaks down, as \Cref{xiout11} does not admit a solution for $\lambda>\lambda^\ast$. The 
corresponding singular solutions are of type III; cf.~\cref{def:soltyp}.
Due to the regularity of \cref{xiout11} with respect to $\delta$, these 
solutions perturb in a regular fashion to solutions of \{\cref{sysdel1w},\cref{sysdelbc}\} for $\delta$ positive and small; in particular, we consider
$\delta\leq\frac1{\sqrt{\lambda_3}}$ with $\lambda_3$ large, in accordance with \cref{eq:R3}.
For $\lambda$ close to $\lambda^\ast$, individual solutions do not perturb regularly; however, the
structurally stable saddle-node bifurcation at $\lambda^\ast$ as a whole will persist as a regular 
perturbation, giving rise to a slightly perturbed value $\lambda^\ast(\delta)$ for the perturbed 
saddle-node point. Since the resulting asymptotics of $\lambda^\ast(\delta)$ is not our main 
concern, we do not consider it further here.

The second subregion of $\mathcal{R}_3$, which includes the overlap with region $\mathcal{R}_2$, 
corresponds to a small neighborhood of $(\lambda,\delta)=(0,0)$ that is given by
\begin{align} \label{p3p2}
(\lambda,\delta)\in\big[0,\tilde{\lambda}\big]\times\Big[0,\frac1{\sqrt{\lambda_3}}\Big].
\end{align}
To study the branch of solutions in this subregion, we solve \Cref{sysdel1w} with initial 
conditions as in~\cref{sysdelbc} by expanding around $(w_0,\lambda,\delta)=(-1,0,0)$, and 
by making use of the fact that the equations can be solved explicitly for $\delta=0$.
Linearizing \Cref{xiout11} around $\delta=0$, we obtain a regular perturbation problem in $\delta$ for $\xi^{\rm out}$, which gives the following expanded form of \Cref{xiout11}, up to 
higher-order terms in $(w_0,\lambda,\delta)$:
\begin{align}\label{w0A}
w_0+1-(4+3 w_0)\lambda\ln\lambda+\frac1{288}(1+w_0)\delta^8\ln\lambda=0.
\end{align}
\Cref{w0A} again contains logarithmic terms due to resonance between the 
eigenvalues $-1$, $0$ (double), and $1$ of the linearization of \Cref{sysdel1} about
the steady state at $(0,-1,-1,0)$ in chart $\kappa_1$. 
These terms arise in the passage of orbits through a neighborhood of $\{r_1=0\}$, as was observed in chart $K_1$; see \cref{sec:logsw}. Solving \Cref{w0A} for $w_0$ 
gives
\begin{align}\label{w0A2}
w_0=-1+\lambda\ln\lambda+C(\delta)\lambda+\OO\big[\lambda^2(\ln\lambda)^2\big]
\end{align}
with $C(\delta)=\OO(\delta^8)$, which is regular in $\delta$, as expected. We note that, for 
$\lambda=0$, \cref{w0A} reduces to the trivial equation $w_0+1=0$, which is solved by
$w_0=-1$, irrespective of $\delta$. The resulting singular solutions are type II-solutions, 
which are shown as the part of the green curve in the blown-up bifurcation diagram in \cref{fig:bdblup} 
that corresponds to $\bar{\eps}$ small. In line with these observations, \Cref{w0A} is 
identical to \Cref{csi2} up to terms of order $\OO(\delta^2\lambda)$ after the rescaling 
of $w$ in~\cref{eq-6}. For $\delta=0$ and $\lambda>0$, on the other hand, we match with the 
branch obtained in the part of region $\mathcal{R}_3$ that corresponds to 
$\lambda\geq\tilde{\lambda}$. 

The results obtained in the above two subregions prove the existence and 
uniqueness of a curve of solutions to the boundary value problem \{\cref{eq-7},\cref{eq-5b}\} in $\mathcal{R}_3$, as 
stated in \cref{prop-3}. It remains to consider the overlap between regions $\mathcal{R}_3$ and
$\mathcal{R}_2$: in $(\lambda,\delta)$-space, $\mathcal{R}_3$ corresponds to
\begin{align}\label{R3ps}
[0,1]\times\bigg[0,\frac1{\sqrt{\lambda_3}}\bigg]\setminus [0,\eps\lambda_3]\times 
\bigg[\frac1{\sqrt{\lambda_2}},\frac1{\sqrt{\lambda_3}}\bigg],
\end{align}
while $\mathcal{R}_2$ covers the area
\begin{align}\label{R2ps}
[0,\eps\lambda_2]\times\bigg[\frac1{\sqrt{\lambda_2}},\delta_1\bigg],
\end{align}
where $\delta_1<\frac2{\sqrt3}$ is defined as in \cref{prop-1}.
Hence, in $(\lambda,\delta)$-space, regions $\mathcal{R}_3$ and $\mathcal{R}_2$ overlap in 
the rectangle
\begin{align}\label{R2R3}
[\eps\lambda_3,\eps\lambda_2]\times\bigg[\frac1{\sqrt{\lambda_2}},\frac1{\sqrt{\lambda_3}}\bigg],
\end{align}
see \cref{fig:ld}, which is the area where the transition between the two regions occurs.
This concludes the proof of \cref{prop-3}.
\end{proof}
\begin{figure}[!h]
\centering
\includegraphics[scale=0.8]{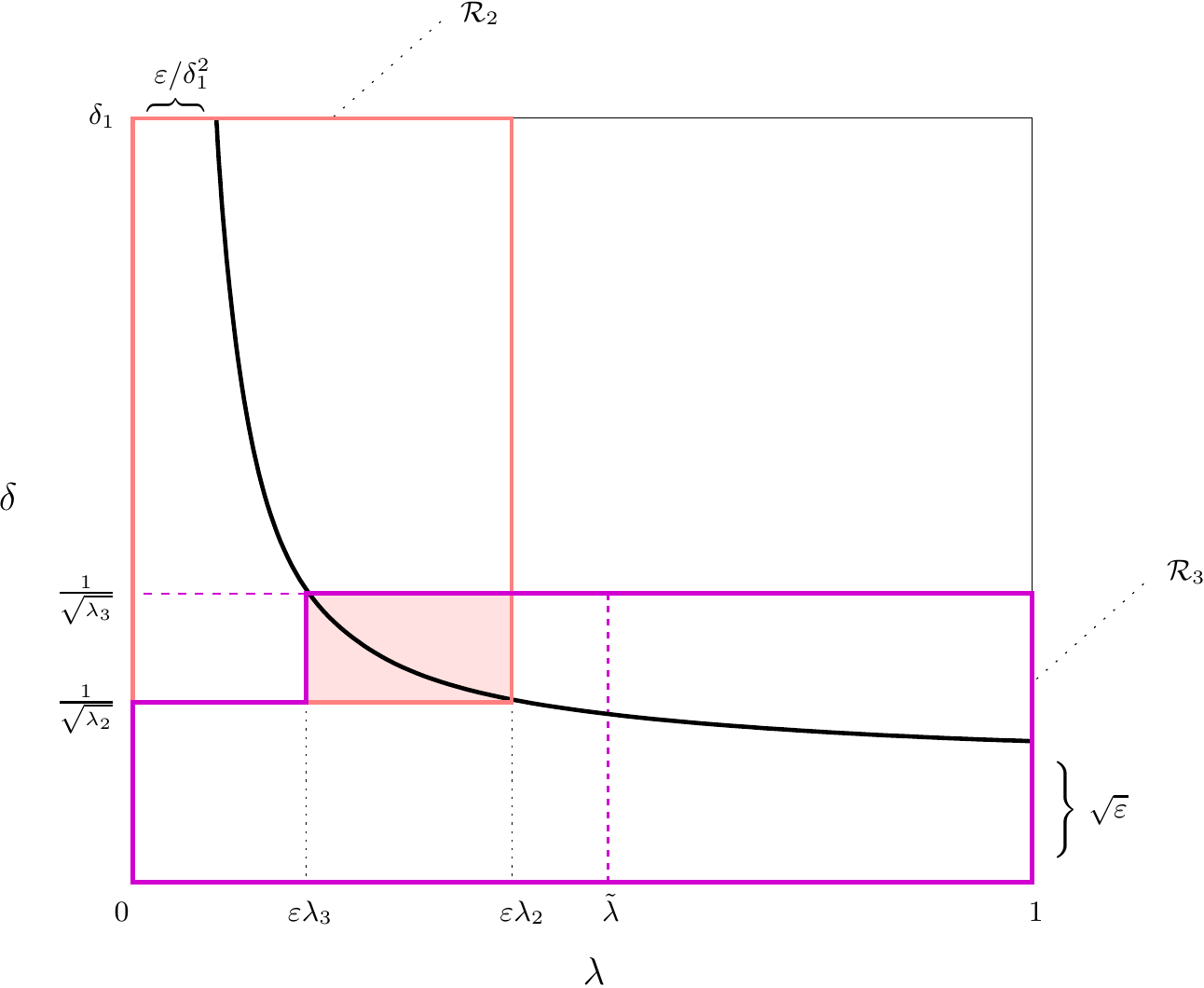}
\caption{Regions $\mathcal{R}_2$ (pink) and $\mathcal{R}_3$ (magenta) in 
$(\lambda,\delta)$-parameter space; cf.~\cref{R2ps} and~\cref{R3ps}, respectively. 
The dashed vertical magenta line at $\lambda=\tilde{\lambda}$ delimits the two subregions 
considered in the proof of \cref{prop-3}. The dashed horizontal magenta line at 
$\delta=\frac1{\sqrt{\lambda_3}}$ indicates that the argument employed in the second part of the 
proof of \cref{prop-3} is valid in the entire rectangle defined in~\cref{p3p2}, where we note that 
region $\mathcal{R}_3$ excludes the rectangle
$[0,\eps\lambda_3]\times\big[\frac1{\sqrt{\lambda_2}},\frac1{\sqrt{\lambda_3}}\big]$
by construction; cf.~\cref{eq:R3}. Regions $\mathcal{R}_2$ and $\mathcal{R}_3$ overlap in the 
shaded rectangle (light pink) given by~\cref{R2R3}. The black curve corresponds to 
$\delta^2\lambda=\eps$, for given $0<\eps\ll 1$.}
\label{fig:ld}
\end{figure}
The last step in the proof of \cref{thm-1} consists in proving \Cref{eq-normu}. 
\begin{proposition}
For $\varepsilon\in(0,\varepsilon_0)$, with $\varepsilon_0>0$ sufficiently small, the upper branch of solutions in \cref{fig:Lin:a} has the expansion stated in \Cref{eq-normu}.
\end{proposition}
\begin{proof}
We first express $\Vert u\Vert_2^2$, with $u$ being the original variable considered in \Cref{eq-2}, in terms of our shifted variable $\tilde u$, as defined in \Cref{eq-shiftu}:
\begin{equation}\label{eq-normutil}
  \Vert u \Vert_2^2 = 2-2 \Vert \tilde u \Vert_1 + \Vert \tilde u \Vert_2^2.
\end{equation}
(While we had omitted the tilde in our notation following \cref{eq-shiftu}, we now include it again  for the sake of clarity.) Due to the symmetry of the boundary value problem \{\cref{eq-7},\cref{eq-5b}\}, we can focus our attention on the interval $[-1,0]$; correspondingly, we split the integrals occurring in \Cref{eq-normutil} into two parts, which are divided by the section $\Sigma_1^{\rm out}$ defined in \cref{eq-20b}. Since $\xi_1^{\rm out}=\xi_2^{\rm in}$, that split implies $[-1,0]=[-1,\xi_1^{\rm out}]\cup[\xi_2^{\rm in},0]$ and, hence, that these integrals can be investigated separately in charts $K_1$ and $K_2$:
 \begin{subequations}\label{eq-norm}
 \begin{align}
 \begin{split}
  \Vert \tilde u \Vert_1 &= \int_{-1}^1 \tilde u(\xi)\,\mathrm{d}\xi = 2 \int_{-1}^0 \tilde u(\xi)\,\mathrm{d}\xi   \\
  &= 2 \Bigg( \int_{-1}^{\xi_1^{\mathrm{out}}} r_1(\xi_1)\,\mathrm{d}\xi_1+\int_{\xi_2^{\mathrm{in}}}^0 r_2(\xi_2) u_2(\xi_2)\,\mathrm{d}\xi_2 \Bigg) \quad\text{and}
  \end{split} \label{eq-norm1} \\
  \begin{split}
  \Vert \tilde u \Vert_2^2 &= \int_{-1}^1 \tilde u^2(\xi)\,\mathrm{d}\xi = 2 \int_{-1}^0 \tilde u^2(\xi)\,\mathrm{d}\xi \\
  &= 2 \Bigg( \int_{-1}^{\xi_1^{\mathrm{out}}} r_1^2(\xi_1)\,\mathrm{d}\xi_1+\int_{\xi_2^{\mathrm{in}}}^0 (r_2(\xi_2) u_2(\xi_2))^2\,\mathrm{d}\xi_2 \Bigg),
  \end{split}  \label{eq-norm2}
 \end{align}
 \end{subequations}
where $\xi_1^{\mathrm{out}}(=\xi_1^{{\rm out}-})$ is approximated as in \Cref{eq-112}. Dividing \Cref{eq-11c} by \Cref{eq-11a} and using \cref{eq-26}, where we recall that $\eps_1=\frac{\eps}{r_1}$, we can rewrite the $\xi_1$-integrals in \Cref{eq-norm} as integrals in $r_1$, with $r_1\in[1,\frac\eps\sigma]$. 
Expanding the resulting integrands for small $\eps$ and evaluating the integrals to the corresponding order, we obtain
\begin{align*}
\int_{-1}^{\xi_1^{\mathrm{out}}} r_1(\xi_1)\,\mathrm{d}\xi_1=\frac{\sqrt3}2 \sqrt{\frac{\eps}{\lambda}} + \OO(\eps^{\frac32})\quad\text{and}\quad\int_{-1}^{\xi_1^{\mathrm{out}}} r_1^2(\xi_1)\,\mathrm{d}\xi_1=\frac{\sqrt3}3 \sqrt{\frac{\eps}{\lambda}} + \OO(\eps^2).
\end{align*}
As for the integrals in $\xi_2$, we recall from \cref{eq-9b} that $r_2=\eps$ in chart $K_2$. Moreover, given the fast-slow structure of \Cref{eq-13}, $u_2$ can be expressed as the sum of a slow and a fast component,
\begin{align*}
 u_2(\xi_2)=1+\check u_2(\tfrac{\xi_2}\eps);
\end{align*}
by the definition of the slow manifold $\SS_2^\eps$ in \cref{lem-1}, the slow contribution is given by $u_2(\xi_2)\sim1$, while the fast contribution $\check u_2$ is obtained from the corresponding stable foliation $\FF_2^{\rm s}(\SS_2^\eps)$. In particular, the latter yields higher-order terms in the $\xi_2$-integrals in \cref{eq-norm}, which implies
\begin{align*}
\int_{\xi_2^{\mathrm{in}}}^0 r_2(\xi_2) u_2(\xi_2)\,\mathrm{d}\xi_2=2 \eps + \OO(\eps^{\frac32}\ln \eps)\quad\text{and}\quad\int_{\xi_2^{\mathrm{in}}}^0 \big(r_2(\xi_2) u_2(\xi_2)\big)^2\,\mathrm{d}\xi_2=\OO(\eps^2).
\end{align*}
Combining these estimates into \Cref{eq-normutil}, we obtain
 \begin{align*}
 \Vert u\Vert_2^2 &=2-2 \bigg(\frac{\sqrt3}2 \sqrt{\frac{\eps}{\lambda}}+2 \eps+\OO(\eps^{\frac32} \ln\eps)\bigg)+\frac{\sqrt3}3 \sqrt{\frac{\eps}{\lambda}}+\OO(\eps^2) \\
 &= 2\bigg(1-\frac{\sqrt3}{3}\sqrt{\frac{\eps}{\lambda}}-2\eps+\OO(\eps^{\frac32}\ln\eps)\bigg),
\end{align*}
which is precisely \Cref{eq-normu}.
\end{proof}
\cref{thm-1} is hence proven.
\begin{remark}\label{rem-slope}
Our analysis suggests that the expansion for the upper solution branch in \Cref{eq-normu} is still valid up to an $\OO(\eps)$-neighborhood of the fold point at $\lambda_\ast(\eps)$; that expansion hence provides a good approximation close to the point where the middle and upper branches in \cref{fig:Lin:a} meet. Differentiating \Cref{eq-normu} with respect to $\lambda$, evaluating the derivative at \mbox{$\lambda=\lambda_\ast(\eps)$}, as given in \Cref{eq-42}, and expanding for $\eps$ small, we obtain
\begin{equation}
\frac{\mathrm{d}\Vert u \Vert_2^2}{\rm{d}\lambda}\bigg|_{\lambda=\lambda_\ast(\eps)} = \frac{8}{9 \eps}+ \frac29 \big(9+4 \sqrt6\big)\ln\eps + \frac{5}{36}\big(59+24 \sqrt6\big)\eps(\ln \eps)^2 +  \OO(\eps^2),
\end{equation}
 which tends to infinity for $\eps\to 0^+$.
\end{remark}

\section{Discussion and Outlook}\label{sec:diou}

In this article, we have investigated stationary solutions of a regularized model
for Micro-Electro Mechanical Systems (MEMS). 
In particular, we have unveiled the asymptotics of the bifurcation diagram 
for solutions of the boundary value problem \{\cref{eq-7},\cref{eq-5b}\},
as the regularization parameter $\eps$ tends to zero.
In the process, we have proven that the new branch of solutions which emerges in the 
bifurcation diagram of the regularized model derives from an underlying, very degenerate singular structure.
Applying tools from dynamical systems theory and, specifically, geometric singular perturbation theory and the blow-up method, we have considered 
separately three principal regions in the bifurcation diagram; cf.~\cref{fig:bdsegm}. 
We emphasize that our findings are consistent with formal asymptotics and numerical simulations 
of Lindsay {\it et al.}; see, in particular, Section~3 of \cite{Li14} and
Section~4 of \cite{Li16}.

One of the most interesting features of the regularized model considered here is the presence of a 
highly singular saddle-node bifurcation point. While Lindsay {\it et al.}~\cite{Li14}
were able to derive a formal leading-order asymptotic expansion in the regularization 
parameter at that point, the coefficients therein had remained undetermined thus far. 
Our approach, on the other hand, allows us to obtain the fold point as the minimum of 
an appropriately defined bifurcation equation and, hence, to calculate explicitly the coefficients in 
that expansion. (For completeness, we remark that the coefficient of the leading-order term therein
appeared in \cite[Section~3]{Li14} in a different context: $\lambda_{0c}=\frac{m-1}{2(m-2)}$, which evaluates to $\frac34$ for $m=4$; see also \Cref{rem:xc}. However, that correspondence does not seem to have been noted there.) 
For verification, a comparison with numerical data obtained with the 
continuation package~\texttt{AUTO} has been performed, showing very good agreement with 
our asymptotic expansion.

Finally, we have shown that the somewhat unexpected asymptotics of solutions to \Cref{eq-2}, as derived in~\cite{Li14}, arises naturally due to a resonance phenomenon in the blown-up vector field. In particular, we have justified the occurrence of logarithmic ``switchback" in that asymptotics via a careful description of the flow through one of the coordinate charts, viz. $K_1$, after blow-up; see also \cite{Po05}. 
Our analysis hence establishes a further connection between the geometric approach proposed here and the method of matched asymptotic expansions.

Our geometric approach to the boundary value problem \{\cref{eq-7},\cref{eq-5b}\}
can be extended to the analysis of steady states of the corresponding regularized fourth-order model, which 
has been studied in~\cite{Li14,Li16,Li15} both asymptotically and numerically.
A future aim is to establish analogous results for that case. Another possible topic for future research is
the geometric analysis of \Cref{eq-2} in higher dimensions, possibly under the simplifying
assumption of radial symmetry.

\section*{Acknowledgments}
AI and PS would like to thank Alan Lindsay for helpful discussions. They would also like to
acknowledge the Fonds zur F\"orderung der wissenschaftlichen Forschung (FWF) for support via the doctoral school ``Dissipation and Dispersion in Nonlinear PDEs'' (project number W1245).
Moreover, AI is grateful to the School of Mathematics at the University of Edinburgh for its hospitality during an extensive research visit. Finally, the authors thank two anonymous referees for insightful comments that greatly improved the original manuscript.

\bibliographystyle{siamplain}
\bibliography{references}

\begin{thebibliography}{10}

\bibitem{Do04}
{\sc N.~Doble and D.~Williams}, {\em The application of {MEMS} technology for
  adaptive optics in vision science}, IEEE J. Sel. Top. Quant., 10 (2004),
  pp.~629--635, \url{https://doi.org/10.1109/jstqe.2004.829202}.

\bibitem{AU}
{\sc E.~Doedel}, {\em A{UTO}: a program for the automatic bifurcation analysis
  of autonomous systems}, in Proceedings of the {T}enth {M}anitoba {C}onference
  on {N}umerical {M}athematics and {C}omputing, {V}ol. {I} ({W}innipeg, {M}an.,
  1980), vol.~30, 1981, pp.~265--284.

\bibitem{Du93}
{\sc F.~Dumortier}, {\em Techniques in the theory of local bifurcations:
  blow-up, normal forms, nilpotent bifurcations, singular perturbations}, in
  Bifurcations and periodic orbits of vector fields ({M}ontreal, {PQ}, 1992),
  vol.~408 of NATO Adv. Sci. Inst. Ser. C Math. Phys. Sci., Kluwer Acad. Publ.,
  Dordrecht, 1993, pp.~19--73,
  \url{https://doi.org/10.1007/978-94-015-8238-4_2}.

\bibitem{Dumortier_2014}
{\sc F.~Dumortier and T.~Kaper}, {\em Wave speeds for the {FKPP} equation with
  enhancements of the reaction function}, Zeitschrift für angewandte
  Mathematik und Physik, 66 (2014), pp.~607--629,
  \url{https://doi.org/10.1007/s00033-014-0422-9},
  \url{https://doi.org/10.1007%2Fs00033-014-0422-9}.

\bibitem{Dumortier_2007}
{\sc F.~Dumortier, N.~Popovi{\'{c}}, and T.~Kaper}, {\em The critical wave
  speed for the
  {Fisher{\textendash}Kolmogorov{\textendash}Petrowskii{\textendash}Piscounov}
  equation with cut-off}, Nonlinearity, 20 (2007), pp.~855--877,
  \url{https://doi.org/10.1088/0951-7715/20/4/004},
  \url{https://doi.org/10.1088%2F0951-7715%2F20%2F4%2F004}.

\bibitem{DR96}
{\sc F.~Dumortier and R.~Roussarie}, {\em Canard cycles and center manifolds},
  Mem. Amer. Math. Soc., 121 (1996), pp.~x+100,
  \url{https://doi.org/10.1090/memo/0577}.

\bibitem{Fe79}
{\sc N.~Fenichel}, {\em Geometric singular perturbation theory for ordinary
  differential equations}, J. Differential Equations, 31 (1979), pp.~53--98,
  \url{https://doi.org/10.1016/0022-0396(79)90152-9}.

\bibitem{GW05}
{\sc Y.~Guo, Z.~Pan, and M.~Ward}, {\em Touchdown and pull-in voltage behavior
  of a {MEMS} device with varying dielectric properties}, SIAM J. Appl. Math.,
  66 (2005), pp.~309--338, \url{https://doi.org/10.1137/040613391}.

\bibitem{Iv08}
{\sc B.~Iverson and S.~Garimella}, {\em Recent advances in microscale pumping
  technologies: a review and evaluation}, Microfluidics Nanofluidics, 5 (2008),
  pp.~145--174, \url{https://doi.org/10.1007/s10404-008-0266-8}.

\bibitem{GSPT}
{\sc C.~Jones}, {\em Geometric singular perturbation theory}, in Dynamical
  systems ({M}ontecatini {T}erme, 1994), vol.~1609 of Lecture Notes in Math.,
  Springer-Verlag, Berlin, 1995, pp.~44--118,
  \url{https://doi.org/10.1007/BFb0095239}.

\bibitem{JKK96}
{\sc C.~Jones, T.~Kaper, and N.~Kopell}, {\em Tracking invariant manifolds up
  to exponentially small errors}, SIAM J. Math. Anal., 27 (1996), pp.~558--577,
  \url{https://doi.org/10.1137/s003614109325966x}.

\bibitem{JK94}
{\sc C.~Jones and N.~Kopell}, {\em Tracking invariant manifolds with
  differential forms in singularly perturbed systems}, J. Differential
  Equations, 108 (1994), pp.~64--88,
  \url{https://doi.org/10.1006/jdeq.1994.1025}.

\bibitem{KS01}
{\sc M.~Krupa and P.~Szmolyan}, {\em Extending geometric singular perturbation
  theory to nonhyperbolic points---fold and canard points in two dimensions},
  SIAM J. Math. Anal., 33 (2001), pp.~286--314,
  \url{https://doi.org/10.1137/s0036141099360919}.

\bibitem{KS01b}
{\sc M.~Krupa and P.~Szmolyan}, {\em Geometric analysis of the singularly
  perturbed planar fold}, in Multiple-time-scale dynamical systems
  ({M}inneapolis, {MN}, 1997), vol.~122 of IMA Vol. Math. Appl., Springer, New
  York, 2001, pp.~89--116, \url{https://doi.org/10.1007/978-1-4613-0117-2_4}.

\bibitem{Kue}
{\sc C.~Kuehn}, {\em Multiple time scale dynamics}, vol.~191 of Applied
  Mathematical Sciences, Springer, Cham, 2015,
  \url{https://doi.org/10.1007/978-3-319-12316-5}.

\bibitem{La84}
{\sc P.~Lagerstrom and D.~Reinelt}, {\em Note on logarithmic switchback terms
  in regular and singular perturbation expansions}, SIAM J. Appl. Math., 44
  (1984), pp.~451--462, \url{https://doi.org/10.1137/0144030}.

\bibitem{La88}
{\sc P.~A. Lagerstrom}, {\em Matched asymptotic expansions. Ideas and
  techniques}, vol.~76 of Applied Mathematical Sciences, Springer-Verlag, New
  York, 1988, \url{https://doi.org/10.1007/978-1-4757-1990-1}.

\bibitem{LY07}
{\sc F.~Lin and Y.~Yang}, {\em Nonlinear non-local elliptic equation modelling
  electrostatic actuation}, Proc. R. Soc. Lond. Ser. A Math. Phys. Eng. Sci.,
  463 (2007), pp.~1323--1337, \url{https://doi.org/10.1098/rspa.2007.1816}.

\bibitem{LL12}
{\sc A.~Lindsay and J.~Lega}, {\em Multiple quenching solutions of a fourth
  order parabolic {PDE} with a singular nonlinearity modeling a {MEMS}
  capacitor}, SIAM J. Appl. Math., 72 (2012), pp.~935--958,
  \url{https://doi.org/10.1137/110832550}.

\bibitem{Li14}
{\sc A.~Lindsay, J.~Lega, and K.~Glasner}, {\em Regularized model of
  post-touchdown configurations in electrostatic {MEMS}: {E}quilibrium
  analysis}, Phys. D, 280 (2014), pp.~95--108,
  \url{https://doi.org/10.1016/j.physd.2014.04.007}.

\bibitem{Li16}
{\sc A.~E. Lindsay}, {\em Regularized model of post-touchdown configurations in
  electrostatic {MEMS}: bistability analysis}, J. Engrg. Math., 99 (2016),
  pp.~65--77, \url{https://doi.org/10.1007/s10665-015-9820-z}.

\bibitem{Li15}
{\sc A.~E. Lindsay, J.~Lega, and K.~B. Glasner}, {\em Regularized model of
  post-touchdown configurations in electrostatic {MEMS}: interface dynamics},
  IMA J. Appl. Math., 80 (2015), pp.~1635--1663,
  \url{https://doi.org/10.1093/imamat/hxv011}.

\bibitem{Li11}
{\sc A.~E. Lindsay and M.~J. Ward}, {\em Asymptotics of some nonlinear
  eigenvalue problems modelling a {MEMS} capacitor. {P}art {II}: multiple
  solutions and singular asymptotics}, European J. Appl. Math., 22 (2011),
  pp.~83--123, \url{https://doi.org/10.1017/S0956792510000318}.

\bibitem{Pe02}
{\sc J.~Pelesko}, {\em Mathematical modeling of electrostatic {MEMS} with
  tailored dielectric properties}, SIAM J. Appl. Math., 62 (2002),
  pp.~888--908, \url{https://doi.org/10.1137/s0036139900381079}.

\bibitem{PB02}
{\sc J.~A. Pelesko and D.~H. Bernstein}, {\em Modeling {MEMS} and {NEMS}},
  Chapman \& Hall/CRC, Boca Raton, FL, 2003.

\bibitem{Po05}
{\sc N.~Popovi{\'c}}, {\em A geometric analysis of logarithmic switchback
  phenomena}, in {HAMSA} 2004: {P}roceedings of the {I}nternational {W}orkshop
  on {H}ysteresis and {M}ulti-{S}cale {A}symptotics ({C}ork, Ireland, 2004),
  vol.~22, 2005, pp.~164--173,
  \url{https://doi.org/10.1088/1742-6596/22/1/011}.

\bibitem{PS041}
{\sc N.~Popovi{\'c} and P.~Szmolyan}, {\em A geometric analysis of the
  {Lagerstrom} model problem}, J. Differential Equations, 199 (2004),
  pp.~290--325, \url{https://doi.org/10.1016/j.jde.2003.08.004}.

\bibitem{PS042}
{\sc N.~Popovi{\'c} and P.~Szmolyan}, {\em Rigorous asymptotic expansions for
  {Lagerstrom's} model equation—a geometric approach}, Nonlinear Anal. Theory
  Methods Appl., 59 (2004), pp.~531--565,
  \url{https://doi.org/10.1016/j.na.2004.07.032}.

\bibitem{SS04}
{\sc B.~Sandstede and A.~Scheel}, {\em Evans function and blow-up methods in
  critical eigenvalue problems}, Discrete Contin. Dyn. Syst., 10 (2004),
  pp.~941--964, \url{https://doi.org/10.3934/dcds.2004.10.941}.

\bibitem{Ts07}
{\sc N.~Tsai and C.~Sue}, {\em Review of {MEMS}-based drug delivery and dosing
  systems}, Sens. Actuators A Phys., 134 (2007), pp.~555--564,
  \url{https://doi.org/10.1016/j.sna.2006.06.014}.

\bibitem{Wa09}
{\sc B.~Watson, J.~Friend, and L.~Yeo}, {\em Piezoelectric ultrasonic
  micro/milli-scale actuators}, Sens. Actuators A Phys., 152 (2009),
  pp.~219--233, \url{https://doi.org/10.1016/j.sna.2009.04.001}.

\bibitem{Wi03}
{\sc S.~Wiggins}, {\em Introduction to applied nonlinear dynamical systems and
  chaos}, vol.~2 of Texts in Applied Mathematics, Springer-Verlag, New York,
  second~ed., 2003.

\end{thebibliography}
\end{document}